\documentclass[11pt,reqno]{amsart}

\usepackage{graphicx}
\usepackage{amssymb,amsmath, amscd, latexsym,amsfonts,bbm, amsthm,stmaryrd, enumerate,phaistos}
\usepackage{yhmath}\usepackage{tikz-cd}
\usepackage{comment}\usepackage{todonotes}\usepackage[all]{xy}
\usepackage[vcentermath]{youngtab}

\usepackage[german, english]{babel} 
\newcommand{\dist}{3mm}

\numberwithin{equation}{section}

\newtheorem{theorem}{Theorem}[section]

\newtheorem{lemma}[theorem]{Lemma}
\newtheorem{proposition}[theorem]{Proposition}

\newtheorem{corollary}[theorem]{Corollary}
\theoremstyle{definition}
\newtheorem{definition}[theorem]{Definition}
\newtheorem{remark}[theorem]{Remark}
\newtheorem{example}[theorem]{Example}

\newcommand{\dimmatrix}{\mathtt{D}}
\newcommand{\dimv}{\mathtt{d}}
\newcommand{\typ}{\mathtt{t}}
\newcommand{\round}{\ring}
\newcommand{\lrarrow}{\vec}
\newcommand{\qq}{{\boxed{?}}}
\newcommand{\qqone}{{\boxed{?_1}}}
\newcommand{\qqtwo}{{\boxed{?_2}}}
\newcommand{\qqonetwo}{{\boxed{?_{1,2}}}}
\newcommand{\qqoneone}{{\boxed{?_{1,1}}}}

\newcommand{\op}{\operatorname}
\newcommand{\bmu}{{\boldsymbol{\mu}}}
\newcommand{\bmuh}{{\hat{\bmu}}}

\newcommand{\blah}{{\hat{\boldsymbol{\la}}}}

\newcommand{\ourOmega}{{\tau}}
\newcommand{\FF}{\Bbbk}

\newcommand{\Z}{\mathbb{Z}}
\newcommand{\mZ}{\mathbb{Z}}
\newcommand{\C}{\mathbb{C}} 
\newcommand{\Q}{\mathbb{Q}}
\newcommand{\N}{\mathbb{N}}
\newcommand{\bbI}{\mathbb{I}}

\newcommand{\fS}{\mathfrak{S}}
\newcommand{\fP}{\mathfrak{X}}

\newcommand{\G}{\mathrm{G}}
\newcommand{\I}{\mathrm{I}}
\newcommand{\W}{\mathrm{W}}
\newcommand{\rmP}{\mathrm{P}}

\newcommand{\Stab}{\mathrm{Stab}}
\newcommand{\Hom}{\mathrm{Hom}}
\newcommand{\End}{\mathrm{End}}

\newcommand{\cH}{\mathcal{H}}
\newcommand{\cS}{\mathcal{S}}

\newcommand{\cB}{\mathcal{B}}
\newcommand{\cP}{\mathcal{P}}
\newcommand{\cQ}{\mathcal{Q}}
\newcommand{\cI}{\mathcal{I}}
\newcommand{\cJ}{\mathcal{J}}

\newcommand{\ccH}{\widehat{\mathcal{H}}}\newcommand{\ccP}{\widehat{\mathcal{P}}}\newcommand{\ccQ}{\widehat{\mathcal{Q}}}
\newcommand{\cbR}{\widehat{\mathbf{R}}}\newcommand{\cbP}{\widehat{\mathbf{P}}}\newcommand{\cbQ}{\widehat{\mathbf{C}}}
\newcommand{\ccS}{\widehat{\mathcal{S}}}
\newcommand{\ccB}{\widehat{\mathcal{B}}}
\newcommand{\cbA}{\widehat{\mathbf{A}}}
\newcommand{\cbB}{\widehat{\mathbf{B}}}

\newcommand{\cGamma}{\widehat{\Theta}}
\newcommand{\Fa}{\mathbf{F}}
\newcommand{\cFa}{\widehat{\Fa}}
\newcommand{\cF}{\mathcal{F}}

\newcommand{\FH}{\mathbb{P}(\cH)}
\newcommand{\oFH}{\overline{\mathbb{P}(\cH)}}
\newcommand{\FS}{\mathbb{P}(\cS)}
\newcommand{\oFS}{\overline{\mathbb{P}}(\cS)}

\newcommand{\cFH}{\widehat{\mathbb{P}(\cH)}}
\newcommand{\coFH}{\widehat{\overline{\mathbb{P}}(\cH)}}
\newcommand{\cFS}{\widehat{\mathbb{P}(\cS)}}
\newcommand{\coFS}{\widehat{\overline{\mathbb{P}}(\cS)}}

\newcommand{\orho}{\overline{\rho}}
\newcommand{\crho}{\widehat{\rho}}
\newcommand{\corho}{\widehat{\overline{\rho}}}

\newcommand{\bA}{\mathbf{A}}\newcommand{\bB}{\mathbf{B}}
\newcommand{\bI}{\mathbf{I}}
\newcommand{\bJ}{\mathbf{J}} 
\newcommand{\bR}{\mathbf{R}}
\newcommand{\bP}{\mathbf{P}}
\newcommand{\bQ}{\mathbf{C}}
\newcommand{\bU}{\mathbf{U}}
\newcommand{\bC}{\boldsymbol{\pitchfork}}
\newcommand{\bM}{{}_{q}{\boldsymbol{\pitchfork}}}
\newcommand{\ba}{\mathbf{a}}
\newcommand{\bi}{\mathbf{i}}

\newcommand{\bm}{\mathbf{m}}
\newcommand{\bu}{\mathbf{u}}
\newcommand{\bx}{\mathbf{x}}
\newcommand{\unclear}{\bC_{a_1,\;1,\;b_1-1,a_2,b_2,\ldots,a_e,b_e}^{a_1+1, b_1-1,a_2,b_2,\ldots,a_e,b_e}} 
\newcommand{\uunclear}{\bC_{1,a_1-1,1,b_1-1,a_2,b_2,\ldots,a_e,b_e}^{1,a_1,\;\;\;\;\;\;\;\; b_1-1,a_2,b_2,\ldots,a_e,b_e}} 

\newcommand{\bv}{\mathbf{v}}

\newcommand{\bbv}{\overline{\mathbf{v}}}
\newcommand{\bb}{\mathbf{b}}

\newcommand{\cLa}{\widehat{\Lambda}}

\newcommand{\DM}{{\overrightarrow{\fP}}}
\newcommand{\AM}{{\overleftarrow{\fP}}}
\newcommand{\DMJ}{{\overrightarrow{\fP_J}}}
\newcommand{\AMJ}{{\overleftarrow{\fP_J}}}
\newcommand{\DMJp}{{\overrightarrow{\fP_J^>}}}
\newcommand{\AMJp}{{\overleftarrow{\fP_J^>}}}
\newcommand{\DMJm}{{\overrightarrow{\fP_J^<}}}
\newcommand{\AMJm}{{\overleftarrow{\fP_J^<}}}
\newcommand{\AMJum}{{\overleftarrow{\fP_{J_\bu}^<}}}

\newcommand{\DD}{\Delta}
\newcommand{\AD}{\nabla}

\newcommand{\e}{\overset{+}{e}}

\newcommand{\la}{\lambda}

\newcommand{\gt}{\theta}

\newcommand{\shift}{\kappa}

\title{Affine quiver Schur algebras and $p$-adic $GL_n$}
\author{Vanessa Miemietz, Catharina Stroppel}

\begin{document}
\begin{abstract}
In this paper we consider the (affine) Schur algebra 
which arises as the endomorphism algebra of certain permutation modules for the Iwahori-Matsumoto Hecke algebra. This algebra describes, for a general linear group over a $p$-adic field, a large part of the unipotent block over fields of characteristic different from $p$. We show that this Schur algebra is, after a suitable completion, isomorphic to the quiver Schur algebra attached to the cyclic quiver. The isomorphism is explicit, but nontrivial. As a consequence, the completed (affine) Schur algebra inherits a grading. As a  byproduct we  obtain a detailed description of the algebra with a basis adapted to the geometric basis of quiver Schur algebras. We illustrate the grading in the explicit example of $\op{GL}_2(\mathbb{Q}_5)$ in characteristic $3$.
\end{abstract}
\maketitle

\tableofcontents

\setlength{\parindent}{0pt} 
\setlength{\parskip}{3pt}

\section{Introduction}
This paper deals with affine Schur algebras for $p$-adic groups over fields of characteristic different from $p$. Classical Schur algebras were introduced by Sandy Green \cite{Green}  as an algebraic tool to study polynomial representations of the general linear group $\op{GL}_n$ over arbitrary fields and named after Schur because they arise as the endomorphism ring of the sum of certain permutation modules of the symmetric group $S_n$. Dipper and James \cite{DipperJames} introduced $q$-Schur algebras over arbitrary fields to study the modular  representation theory of the finite general linear groups $\op{GL}_n(\mathbb{F}_q)$ in non-describing or cross characteristic. The Schur algebras we consider in this paper are their analogues for the representation theory of the $p$-adic group $\op{GL}_n(E)$, where $E$ is a finite extension of $\Q_p$, over a field of characteristic different from $p$. As a main result, we show that (after a suitable completion) this algebraically defined algebra has a geometric realization as a convolution algebra with underlying vector space the equivariant cohomology of some partial quiver flag varieties introduced in \cite{SW} under the name {\it quiver Schur algebras}. 

Let $\FF$ be an algebraically closed field of 
characteristic $\ell \neq p$, such that the cardinality $q$ of the residue field of $E$ is not congruent to $1$ modulo $\ell$.
We are interested in the category of smooth representations of $\G=\op{GL}_n( E)$ over the field $\FF$ (or equivalently the category of nondegenerate representations of the global Hecke algebra of locally constant compactly supported functions on $G$). This is known to have a block decomposition by inertial classes of supercuspidal support \cite{B84}, \cite{Vigneras2},  \cite{SS}. In this article, we are interested in the so-called {\it unipotent block} $\cB$ which contains the trivial representation. The special case $q\equiv 1 \mod \ell$ is treated in Section~\ref{Appendix}.

As in the case of $\op{GL}_n(\mathbb{F}_q)$, the Schur algebra will not describe the whole unipotent block, but rather a proper subcategory $\cB^1$ which is the lowest layer in a finite filtration $\cB^1\subseteq\cB^2\subseteq\cB^3\subseteq\ldots\subseteq\cB$. Namely, let $I\subset \G$ be an Iwahori subgroup and let $\mathcal{I}$ be the annihilator of the $\G$-representation $\FF[I\backslash\G]$ (inside the global Hecke algebra). Then $\cB^i\subset\cB$ is the full subcategory consisting of all representations annihilated by $\mathcal{I}^i$. The categories $\cB^i$ are abelian. 
It is proved in \cite{V} that the first layer $\cB^1$ is equivalent to the category of all modules for the {\it affine Schur algebra} $\cS$,
\begin{eqnarray}
\label{BandS}
\cB^1\cong\cS-\op{Mod},
\end{eqnarray}
where $\cS$ is defined as the endomorphism ring 
\begin{eqnarray*}
\cS&=&\End_{{\FF[\I\backslash \G/\I]}}\left(\bigoplus_{J \subseteq \bbI} \FF[\rmP^J\backslash \G/\I]\right)\;=\;\End_{\cH}\left(\bigoplus_{J \subseteq \bbI} \bv_J \cH\right).
\end{eqnarray*}
Here $\cH=\FF[I\backslash \G/I]$ is the (affine) {\it Iwahori-Matsumoto Hecke algebra}, \cite{IM}, and the sum is taken over all standard parahoric subgroups $\rmP^J$ attached to a  subset $J$ of the set  $\bbI$ of (finite) simple reflections, and  $\bv_J \cH$ is the corresponding trivial representation induced to $\cH$.  
In particular, $\cS$ contains $\cH$ as an idempotent subalgebra from setting $J=\emptyset$. (We would like to stress however that, in general, multiplication with $\bv_J$ is not a projection, since $\bv_J$ does not need to be a quasi-idempotent. In case $q$ is a root of unity, it  might happen that $\bv_J^2=0$. This phenomenon is one of the technical difficulties in our paper.) 
Note that $\cB^1$, alias $\cS-\op{Mod}$, is in general not equivalent to $\cH-\op{Mod}$, since $\cB^1$ contains in addition the cuspidal representations, which are not included in the subcategory $\cH-\op{Mod}$ of $\cS-\op{Mod}$.  

We expect that $\cB$ in fact only differs from $\cB^1$ by self-extensions depending on the cuspidal support of the corresponding simple modules, and thus $\cB^1$ contains quite detailed information about the unipotent block $\cB$.
 
Note that the classification of irreducible representations in $\cB^1$ (or equivalently in the unipotent block $\cB$) is provided by \cite{Vigneras2}, \cite{MS}, and a convenient labelling set for the irreducible modules is given by certain multisegments, 
extending the Bernstein-Zelevinsky classification of irreducible modules for the Iwahori-Matsumoto Hecke algebras, \cite{BZ}, \cite{Zel1} in characteristic zero. The block decomposition and classification in \cite{Vigneras2} is via the local Langlands correspondence for $\mathrm{GL}_n$ in characteristic $\ell\neq p$, that is an extension of the local Langlands correspondence  over the complex numbers, \cite{HT}, \cite{Harris}, \cite{He}, \cite{Scholze} (or \cite{Wedhorn} for an overview). In particular, this gives the rank of the Grothendieck group of  $\cB^1$.

In this article, we take this one step further by providing tools for a better understanding of  extensions between simple modules and moreover of the structure and  the homological properties of the categories involved, as well as making a connection with geometry. To do so, we compare the affine Schur algebras  to the {\it quiver Schur algebras} from \cite{SW} attached to the cyclic quiver with $e$ vertices (viewed as the oriented affine Dynkin diagram for $\hat{\mathfrak{sl}}_e$). These algebras contain the so-called {\it quiver Hecke algebras} or {\it KLR}-algebras, originally introduced in \cite{KL}, \cite{Rouquier}, see also \cite{VV}. Over $\FF=\C$, their graded module categories furthermore provide by \cite{SW} a categorification of the generic Hall algebra (in the sense of \cite{Schiffmann}) for the cyclic quiver with $e$ vertices. Hereby  $e$ is the multiplicative order modulo $\ell$ of the cardinality of the residue field of $E$  and $e=\infty$ if $\ell=0$. 

Given a fixed dimension vector ${\bf d}$ for the cyclic quiver on $e$ vertices, one considers the space of {\it flagged nilpotent representations} with dimension vector ${\bf d}$, that is, representations together with a filtration such that the associated graded is semisimple. In contrast to the KLR-algebras we allow arbitrary partial flags instead of full flags only. Fixing a sequence $\blah$ of dimension vectors for the successive quotients we denote this space $\cQ(\blah)$. Following the ideas of Chriss and Ginzburg \cite{CG} we consider the ``Steinberg type'' variety  $\mathcal{Z}({\blah, \bmuh})=\cQ(\blah) \times_{\op{Rep}_{\bf d}}\cQ(\bmuh)$. The quiver Schur algebra $\bA_{\bf{d}}$ is then its $\op{GL}_{\bf{d}}(\mathbb{C})$-equivariant Borel-Moore homology  
\begin{eqnarray*}
\bA_{\bf{d}}&=&\bigoplus_{(\blah,\bmuh)}H_{\op{GL}_{\bf{d}}}^{\op{BM}}(\mathcal{Z}({\bmuh, \blah})),
\end{eqnarray*}
equipped with the convolution product.  By construction, this algebra comes along with {\it a $\mathbb{Z}$-grading} and with a faithful representation, see \cite{SW}.

Crucial for us here is that via the faithful representation we see that the quiver Schur algebra can be defined over any field, in particular over the field $\FF$.  Over the complex numbers,  the principal underlying constructions using convolution algebras, \cite{CG} and (other) geometric realisations of the affine Schur algebras, see e.g. \cite{GV},  are well-established.   In case of the field $\FF$, the technicalities and geometric tools are not as well developed as one might hope for. This also applies to the connection with the representation theory of affine Kac-Moody Lie algebras, but see \cite{VV2} for crucial steps in this direction.

{\bf Main result:} 
Our main result (Theorem~\ref{IsoTheorem}) is that the affine Schur algebra and the quiver Schur algebra (both over $\FF)$ are isomorphic after suitable completions. More precisely, we construct a sequence of isomorphisms of algebras
\begin{equation*}
\xymatrix{
\ccS_\bi\ar[rrr]^{\text{Proposition~\ref{StoQ}}\quad}&&&\cbQ_\bi\ar[rrr]^{\text{Proposition~\ref{twistA}}}&&&\cbB_\bi\ar[rrr]^{\text{Proposition~\ref{shiftA}}}&&& \cbA_\bi,
}
\end{equation*}
where the intermediate algebras are certain twisted versions of quiver Schur algebras.
This, in particular,  implies  that the category $\cB^1_\chi$ of representations $M$ in $\cB^1$ with fixed generalized central character $\chi$ 
(in the sense that each element in $M$ is annihilated by some power of $\chi$)
inherits a grading. The categories $\cB^1_\chi$, as $\chi$ varies over all central characters, in particular, contain all finite-length representations.

The existence of such a $\mZ$-grading seems to be quite unexpected and has no explanation in the $p$-adic representation theory at the moment. Although the modules are of infinite length,  the graded pieces are finite dimensional and so the grading allows us to consider Jordan-H\"older multiplicities degree-wise where they then, in fact, become finite and well-defined. Hence we can use formal power series to express the graded multiplicities. 

The additional algebras appearing in the main theorem interpolate between between the algebraic and geometric construction. Although they are in the current article more a technical tool than of own interest, their appearance shows subtle differences between the classical algebraic picture and the new graded version behind quiver Hecke algebras (like for instance the preference to work with the trivial versus the sign representation, the symmetric role of splits and merges in the graded version,  and the established labelling of the irreducible objects on each side). The isomorphism allows to make the explicit translation. 

We should remark that after having completed this article we found also a geometric construction of our intermediate {\it modified Schur algebra} which in fact connects the quiver Schur algebra more directly to Lusztig's original construction of quantum groups, \cite{Lusztigperv}. Namely instead of considering flagged representations  (that is representations equipped with a filtration)  such that the associated graded is semisimple, we consider the slightly weaker conditions of flagged representations without assuming the associated graded to be semisimple. Our isomorphism then identifies the two geometric constructions. More details on the geometric modified Schur algebra will appear in \cite{Tomasz}.

In small examples,  our isomorphism allows us to give a {\it complete} and {\it explicit} description of this category in terms of the path algebra of a quiver with generators and relations, an example is given in Section~\ref{Lastsection}. In particular, it allows us to compute extensions between simple modules in small examples. This provides a first step towards general results about the homological algebra of $\cB^1$, based on results on quiver Hecke algebras.

The proof of the main result relies on a very careful comparison of faithful representations of all involved algebras. The final result is then an {\it explicit} (non-trivial) isomorphism. 

Besides the main theorem, the paper contains some fundamental results about the algebras involved. For instance, we construct several generating sets for the affine Schur algebras (see in particular Corollary~\ref{corgenerating} and Proposition~\ref{generatingbetter}), explicit faithful representations (in Section~\ref{secfaithful}) and geometrically adapted bases (in Section~\ref{secgoodbasis}). The paper also contains (see Section~\ref{SectionDem}) explicit formulae for Demazure (divided difference) operators interacting with multiplication by polynomials, which we believe should play an important role in a possible categorification result. They generalise crucial formulae from the categorification of quantum groups, see e.g. \cite{KL3}, \cite{KLMS}, and well-known formulae from the geometry of flag varieties.

In characteristic zero and for generic $q$, the affine $q$-Schur algebra was studied in detail by Richard Green \cite{Gr2} who also realised it as a quotient of the quantum group for $\hat{\mathfrak{gl}}_n$. In this case a complete presentation of the algebra is available,  \cite{DG}. In our more general situation such a presentation does not exist yet, but our faithful representations turn the problem of finding a presentation into a problem of linear algebra. Moreover our explicit formulae should make it possible to generalise the geometric results for quiver Schur algebras defined over $\FF=\C$ to the positive characteristic case with $q$ a root of unity. 

We have tried to make this paper as self-contained as possible, in order to make it accessible to 
readers both from a representation theoretic or a number theoretic background. 

{\bf Acknowledgments} We thank G{\"u}nter Harder, David Helm, Peter Scholze, Shaun Stevens and Torsten Wedhorn for useful discussions on the background material of this paper, Ruslan Maksimau and Andrew Mathas for sharing their insight into Hecke algebras, and the referees for their advice. This work was partly supported by the DFG grant SFB/TR 45 and EPSRC grant EP/K011782/1.  

\section{Preliminaries}
We fix a prime $p$ and a natural number $n\geq 2$ and consider the general linear group $\G=\mathrm{GL}_n(E)$ for  a finite extension $E$ of $\Q_p$, the field of $p$-adic numbers. The field $E$ has a local ring $\mathfrak{o}$ of integers, whose quotient by its maximal ideal $\mathfrak{p}$ is a finite field of characteristic $p$. We let $q$ denote the cardinality of this residue field. We furthermore fix an algebraically closed  field $\FF$ of characteristic $\ell\geq 0$, $\ell\neq p$ and let  $e$ be the multiplicative order of  $q$ in  $\FF$. We assume $q\neq  1\mod \ell$.

\subsection{The extended affine Weyl group}
\label{extaff}
We start by recalling the definition and basic facts of the extended affine Weyl group attached to $\G$. For more details see e.g. \cite{IM} or \cite{Gr1}, \cite{Lu} for a description in terms of periodic permutations. 

The {\it extended affine Weyl group} associated to $\G$ is the group $\W$ generated by a set $\bbI_0=\{s_0, \dots, s_{n-1}\}$ of {\it simple reflections} of order two and an element $\ourOmega$ of infinite order, given by the following presentation: 
\begin{eqnarray}
\label{W} 
\W&=&\left\langle \ourOmega, s_i,  0\leq i\leq n-1\left|
\begin{array}{c} s_{\overline{i}}^2=1, \quad  \ourOmega s_{\overline{i}} = s_{\overline{i-1}} \ourOmega\\ s_{\overline{i+1}}s_{\overline{i}}s_{\overline{i+1}}=s_{\overline{i}}s_{\overline{i+1}}s_{\overline{i}} \end{array}\right.\right\rangle.
\end{eqnarray}
where $\overline{i}\in \{0,\dots, n-1\}$ with $\overline{i}\equiv i \mod n$. 

Using the relation $\ourOmega s_{\overline{i}}=s_{\overline{i-1}}\ourOmega$ we can
(in a unique way) write every element $w\in W$ as $x\ourOmega ^j$ for some $x$
 contained in the subgroup generated by $I_0$ and $j\in\mathbb{Z}$. Define  
the length of $w=x\ourOmega ^j$ as $\ell(w)=\ell(x)$, where $\ell(x)=r$ with $r$ minimal such that
 $x=s_{i_1}\cdots s_{i_r}$ for some $i_j\in 0,\ldots,n-1$. 

We view $\W$ as a subgroup of $\G$ by choosing lifts of its elements as follows: For $i=1,\dots, n-1$, we choose the corresponding permutation matrix interchanging the $i$th and $(i+1)$st rows and columns. For $s_0$ we take the matrix with entry $1$ in position $(j,j)$ for $j=2,\dots, n-1$, the uniformizer $\varpi$ (a fixed generator of $\mathfrak{p}$) in position $(n,1)$ and its inverse $\varpi^{-1}$ in position $(1,n)$, and all other entries being zero. Finally $\ourOmega$ has $\varpi$ in position $(n,1)$ and $1$ in positions $(j,j+1)$ for $j=1, \dots, n-1$, with again all other entries being zero.

There is another presentation of $\W$ as semi-direct product $\fS \rtimes \fP$,
 where $\fS$ is the symmetric subgroup generated by  $\bbI=\{s_1, \dots, s_{n-1}\}$ and $\fP$ is a free abelian (multiplicative) group generated by $X_1, \dots, X_n$, on which $\fS$ acts by permuting the generators.
More specifically, a general element of $\fP$ is a Laurent monomial $X_1^{a_1}\cdots X_n^{a_n}$ with $a_i\in \Z$ and $s_iX_is_i = X_{i+1}$.
A representative of $X_i$  in $\G$ can be chosen to be a matrix with $1$'s along the diagonal, except in position $(i,i)$ where we put $\varpi$. We record the following.

\begin{lemma}\label{weylgppres}
An isomorphism of groups $\W\cong \fS \rtimes \fP$ is given by 
\begin{eqnarray}
\label{tau}
&s_i\mapsto s_i\quad(i=1,\dots, n-1),\quad
s_0\mapsto s_{n-1}\cdots s_2s_1s_2\dots s_{n-1}X_1X_n^{-1},\quad\quad&\nonumber\\
&\ourOmega\mapsto s_{n-1}\cdots s_1X_1.\quad\quad&
\end{eqnarray}
Its inverse sends $X_1$ to $s_1\cdots s_{n-1}\ourOmega$ and of course  $s_i$ to $s_i$  for $1\leq i\leq n-1$.
\end{lemma}


From now on we will identify the two presentations (so that for instance the equality (*) in the next formula makes sense).
 
Let $\flat$  be the automorphism of $\W$ fixing generators $s_1, \dots, s_n$ and sending $X_i$ to $X_i^{-1}$ for $i=1,\dots, n$. For convenience, we record that 
\begin{eqnarray*}
&s_0^\flat = s_{n-1}\cdots s_1\cdots s_{n-1}X_1^{-1}X_n \stackrel{(*)}{=} s_{n-1}\cdots s_1\cdots s_{n-1}s_0s_{n-1}\cdots s_1\cdots s_{n-1}&\\
&\ourOmega^\flat = s_{n-1}\cdots s_1X_1^{-1} = s_{n-1}\cdots s_1\ourOmega^{-1}s_{n-1}\cdots s_1. &
\end{eqnarray*}


\begin{lemma}
Let $wp \in \W$ with $w \in \fS$ and $p=X_1^{a_1}\cdots X_n^{a_n}\in \fP$.
\begin{enumerate}[i.)]
\item If  $a_i\geq 0$ for all $i\in\{1,\dots, n\}$, then $wp$ can be expressed in terms of generators from $\bbI$ and $\ourOmega$ (involving only positive powers of $\ourOmega$).
\item If  $a_i\leq 0$ for all $i\in\{1,\dots, n\}$, then $wp$ can be expressed in terms of generators from $\bbI$ and $\ourOmega^{-1}$ (involving only negative powers of $\ourOmega$).
\end{enumerate}
\end{lemma}

\proof
This follows directly from the fact that $X_i = s_{i+1}\cdots s_{n-1}\ourOmega s_1\dots s_i$.
\endproof

\subsection{Parabolic subgroups and shortest double coset representatives}

For a subset $J\subsetneq \bbI_0$, we denote by $\W_J= \langle s_i \mid s_i \in J \rangle$ the parabolic subgroup generated by $J$. Note that this is a finite group, isomorphic to the direct product of some symmetric groups.

Let now $J,K\subseteq \bbI$.  Then each double coset in $\W_K\backslash\W/ \W_J$ contains a unique shortest (i.e. minimal length) coset representative. We denote the set of shortest double coset representatives by $D_{K,J}$. If moreover $J_1,J_2$ are both subsets of $K\subseteq \bbI$, we denote the (finite) set of shortest coset representatives in $\W_{J_1}\backslash \W_K/\W_{J_2}$ by $D^K_{J_1,J_2}$.

For $d\in D_{K,J}$, the set $dJ\cap K$ is defined as the intersection of $K$ with all elements in W of the form $ds_jd^{-1}$ for $s_j \in J$, i.e.
$dJ\cap K=\{s\in K\mid d^{-1}sd\in J\}$. Moreover we abbreviate $dJ=dJ\cap \bbI$.

For $d\in D_{K,J}$ and any element $w$ in $\W_Kd\W_J$, there exist unique elements  $w_K\in \W_K,w_J\in\W_J$, and $a\in D^J_{d^{-1}K\cap J ,\emptyset}$, respectively $b\in D^K_{ \emptyset, dJ\cap K}$ such that
\begin{equation}\label{eqaddlength}
w =w_kda = bdw_J \quad \hbox{ with }\quad l(w) = l(w_k)l(d)l(a) = l(b)l(d)l(w_J).
\end{equation}

\subsection{Another set of double coset representatives}\label{otherreps}

Let $J\subseteq\bbI$. A monomial $X_1^{a_1}\cdots X_n^{a_n}\in\fP$ is called $J$-{\it dominant} if $a_i \geq a_j$ for all $i\leq j$ such that $s_i$ and $s_j$ are conjugate in $W_J$. It is called $J$-{\it antidominant} if $a_i \leq a_j$ for all $i\leq j$ as above. We denote by $\DMJ$ and $\AMJ$ the set of $J$-dominant and the set of $J$-antidominant elements respectively. Note that each $W_J$-orbit in $\fP$ contains a unique $J$-dominant element and a unique $J$-antidominant element.

\begin{proposition}
\label{tiliting}
Let $J,K \subset \bbI$. The sets
\begin{eqnarray*}
\DD_{K,J}&=&\left\{ dp \mid d \in D^\bbI_{K,J}, p \in \DM_{d^{-1}K\cap J}\right\}\quad\text{and}\\
\AD_{K,J}&=&\left\{ dp \mid d \in D^\bbI_{K,J}, p \in \AM_{d^{-1}K\cap J}\right\}
\end{eqnarray*}
both form a complete set of inequivalent coset representatives in $\W_K\backslash \W/\W_J$.
\end{proposition}

\proof 
We only prove the first claim, since the second one is analogous. We first show that every double coset in $\W_K\backslash \W/\W_J$ contains an element from $\DD_{K,J}$.
We know it contains an element from $D_{K,J}$, so let $y \in D_{K,J}$ and write $y=wf$ for $w \in \fS,f\in \fP$.
Using \eqref{eqaddlength}, we can find $d\in D^\bbI_{K,J},w_K\in \W_K, a\in D^J_{d^{-1}K\cap J, \emptyset}$ such that $w = w_Kda$. Then $y=w_Kd (afa^{-1})a$. Now let $t \in \W_{d^{-1}K\cap J}$ such that $p=tafa^{-1}t^{-1}\in \fP$ is $J$-dominant.
Then $dtd^{-1} \in \W_{K\cap dJ}\subseteq \W_K$  and therefore $y= w_Kdtd^{-1}dpa \in \W_Kdp \W_J$ with $d(afa^{-1})\in\DD_{K,J}$ as claimed.

Conversely, we need to show that $dp$ is the unique element in $\W_Kdp \W_J$ with $d \in D^\bbI_{K,J}$ and $p \in \DM_{d^{-1}K\cap J}$, so take an element $w_1dpw_2 \in \W_Kdp \W_J$ and write it as $w_1dpw_2 = \sigma f$ with $\sigma \in \fS,f\in \fP$. Note that necessarily $\sigma = w_1d w_2$ and $f = w_2^{-1}p w_2$. Assume $\sigma \in D^\bbI_{K,J}$ and $f \in \DM_{d^{-1}K\cap J}$. Then $w_1dw_2 = d$, so writing 
$w_1 = ab$ with $a \in D^K_{\emptyset, dJ\cap K}, b \in \W_{dJ \cap K}$ and thus 
$d = w_1dw_2 = a d (d^{-1}bd)w_2$ with $(d^{-1}bd)w_2\in \W_J$, we have a presentation of $d$ of the form \eqref{eqaddlength}, from which we deduce $a=1$ and $(d^{-1}bd)w_2 = 1$, in particular $w_2 \in \W_{d^{-1}K\cap J}$. Hence $y = d w_2^{-1}p w_2$ with $w_2 \in \W_{d^{-1}K\cap J}$. Since $p$ is the unique $d^{-1}K\cap J$-dominant element in its $ \W_{d^{-1}K\cap J}$-orbit, it follows that $w_2^{-1}p w_2 = p$ and hence $y=dp$.
\endproof

\section{The Hecke algebra and Hecke modules}\label{heckesec}

The goal of this section is to define the Iwahori-Matsumoto-Hecke algebra, originally introduced in \cite{IM}, and to construct a faithful representation. Most of the statements can be found in \cite{LuHecke}. We collect some basic facts and give detailed proofs for those for which we could not find an appropriate reference.  
\subsection{The Iwahori-Matsumoto Hecke algebra of $\G$}
We start with the following presentation of the Hecke algebra due to Bernstein:
\begin{definition}{
The \emph{Iwahori-Matsumoto Hecke algebra} associated with $\G$ is the unitary $\FF$-algebra $\cH=\cH_n$ 
generated  by $T_1, \dots, T_{n-1}$, and $X_1^{\pm 1}, \dots, X_n^{\pm 1}$, subject to the defining relations 
\begin{equation*}
\begin{array}[t]{cllll}
\text{(H-1)}&&(T_i-q)(T_i+1)= 0, && \hbox{for } 1 \leq i \leq n-1,\\
\text{(H-2)}&&T_iT_j=T_jT_i&\hbox{if }|i-j|>1, &\hbox{for }1\leq i,j \leq n-1,\\
\text{(H-3)}&&T_iT_{i+1}T_i = T_{i+1}T_iT_{i+1} &&\hbox{for } 1 \leq i \leq n-2,\\
\text{(H-4)}&&X_iX_i^{-1}= 1=X_i^{-1}X_i&& \hbox{for } 1 \leq i \leq n-1,\\
\text{(H-5)}&&X_iX_j= X_jX_i&& \hbox{for } 1 \leq i \leq n-1,\\
\text{(H-6)}&&T_iX_j = X_jT_i& \hbox{if }|i-j|>1, &\hbox{for }1\leq i,j \leq n-1,\\
\text{(H-7)}&&T_iX_iT_i = qX_{i+1}&& \hbox{for } 1 \leq i \leq n-2,
\end{array}
\end{equation*}
where $q$ is the cardinality of the residue field of $E$.}
\end{definition}

Note that in particular, for $i=1,\dots, n-1$,
\begin{eqnarray}
\label{easyformel}
(T_i-q)T_i=-(T_i-q)&\text{and}&(T_i+1)T_i=q(T_i+1).
\end{eqnarray}
Moreover, the $T_i$ are invertible with
\begin{eqnarray}
T_i^{-1}=q^{-1}T_i+(q^{-1}-1)&\text{and}&T_i^2=(q-1)T_i+q.
\end{eqnarray}

We denote by $\cP=\cP_n=\FF[X_1^{\pm 1}, \dots, X_n^{\pm 1}]$ the subalgebra of $\cH$ generated by the $X_i^{\pm 1}$, where $i=1,\dots,n$. Note that the subalgebra in $\cP$ given by symmetric (Laurent) polynomials is central by (H-7).

The original definition of the Iwahori-Matsumoto Hecke algebra of $\G$ over the field $\FF$ is the convolution algebra $\FF[\I\backslash \G/\I]$ of compactly supported $\I$-bi-invariant functions on $\G$ with values in $\FF$, where $\I$ is the Iwahori subgroup.
Any such function can be written as $\sum_{w \in W}a_w \chi_{\I w\I}$ with some $a_w\in \FF$ where $\chi_{\I w\I}$ is the characteristic function on the double coset $\I w\I$.  

Abbreviating $T_w=\chi_{\I w\I}$ gives us the following presentation, \cite[Theorem 3.5]{IM}, of $\FF[\I\backslash \G/\I]$: The algebra is generated by 
$T_{s_i}$ for $i=0,\dots n-1$ and  $T_\ourOmega$ 
subject to relations (using notation as in \eqref{W}):
\begin{equation}\label{conv}
T_{s_i}^2=(q-1)T_{s_i}+q , \quad T_{s_{\overline{i+1}}} T_{s_{\overline{i}}} T_{s_{\overline{i+1}}}= T_{s_{\overline{i}}} T_{s_{\overline{i+1}}} T_{s_{\overline{i}}},\quad   T_\ourOmega T_{s_{\overline{i}}} = T_{s_{\overline{i-1}}} T_{\ourOmega}.
\end{equation}

The following isomorphism  justifies the twofold use of the same notation:

\begin{lemma}\label{heckepres}
We have an isomorphism of $\FF$-algebras
$$
\begin{array}{cccl}
\Theta: & \FF[\I\backslash \G/\I] & \to & \cH \\
& T_{s_i} & \mapsto & T_i, \quad (i=1,\dots, n-1),\\
& T_{s_0}  & \mapsto & q^{n-1}X_1^{-1}X_n (T_{n-1}\cdots T_2T_1T_2\cdots T_{n-1})^{-1},\\
& T_\ourOmega &  \mapsto & q^{-(n-1)/2} T_{n-1}\cdots T_1 X_1.
\end{array}$$
This isomorphism sends $T_\ourOmega^n$ to $X_1 \cdots X_n$.
\end{lemma}

From now on we will freely identify the two presentations.

The following two sets are $\FF$-bases of $\cH$, \cite[Proposition 3.7]{LuHecke}:
\begin{equation}\label{heckebasis}
\{X_1^{a_1}\cdots X_n^{a_n}T_w \mid w \in \fS, a_i \in \Z \},\;\{T_wX_1^{a_1}\cdots X_n^{a_n}\mid w \in \fS, a_i \in \Z \}.
\end{equation}

\subsection{The ideals $\bv_J\cH$}

For $J \subseteq \bbI$  with corresponding parabolic subgroup $W_J$ of $W$ we denote by  $\cH_J\subset \cH$ the (finite-dimensional) Hecke subalgebra generated by $\{ T_i \mid s_i \in J\}$ and define
\begin{eqnarray}
\label{v}
\bv_J\;=\; \sum_{w \in W_J} T_w,&\text{and}& \bbv_J =  \sum_{w \in W_J} (-q)^{-l(w)}T_w.
\end{eqnarray}
We often abbreviate $\bv=\bv_\bbI$ and $\bbv=\bbv_\bbI$. Note that $\FF\bv_J$ and  $\FF\bbv_J$ are the $1$-dimensional trivial respectively sign (right) $\cH_J$-modules via \eqref{easyformel}. They generate the following right ideals in $\cH$ which play the role of permutation modules in the representation theory of the symmetric group.


\begin{lemma}
\label{cyclic}
The right ideals 
\begin{eqnarray*}
\cH^J_{\op{triv}}&=&\{h \in \cH \mid (T_{i}-q)h=0 \hbox{ for all $i$ such that } s_i \in J\},\\  
\cH^J_{\op{sgn}}&=&\{h \in \cH \mid (T_{i}+1)h=0 \hbox{ for all $i$ such that } s_i \in J\}
\end{eqnarray*}
are principal right ideals in $\cH$, generated by $\bv_J$ respectively $\bbv_J$. 
\end{lemma}

\proof
Clearly,  $\bv_J\in\cH^J_{\op{triv}}$. Since $\cH$ is by \eqref{heckebasis} a free left module over $\cH_J$, 
$\cH^J_{\op{triv}} = \{h \in \cH_J\otimes_{\cH_J} \cH \mid (T_{i}-q)h=0 \hbox{ for all $i$ such that } s_i \in J\}=  \bv_J\cH,$
and the first claim follows. The second is similar.
\endproof

\begin{corollary} 
\label{basisHJ}
In case $J \subseteq \bbI$, a $\FF$-basis for $\bv_J\cH$ respectively $\bbv_J\cH$ is given by 
$$\left\{T_wX_1^{a_1}\cdots X_n^{a_n}\left|  w \in D^\bbI_{J,\emptyset}, a_i \in \Z \right.\right\}.$$
\end{corollary}

\proof
This follows directly from Lemma~\ref{cyclic} and \eqref{heckebasis}.
\endproof

Note that the ideals $\bv_J\cH$ and  $\bbv_J\cH$ have isomorphic endomorphism rings. To pass between them we will later need the algebra automorphism $\sharp$ of $\cH$, which is the $q$-analogue of $\flat$, defined on the generators by 
\begin{equation}\label{eqhash}
\begin{array}{rclc}
T_i&\mapsto&T_i^\sharp =q-1-T_i=-qT_i^{-1},& \qquad (i=1,\dots,n-1),\\
X_j&\mapsto&X_j^\sharp= X_j^{-1} &\qquad (j=1,\dots,n).
\end{array}
\end{equation}
\begin{remark}
\label{commute}
If $f\in\cP$ is $s_i$-invariant, then $f T_i^\sharp =T_i^\sharp f$ for $i=1, \dots, n-1$.
\end{remark}

\subsection{A completion of $\cH$}\label{Hcompletionsec}
Recall, \cite[Proposition 3.11]{LuHecke}, that the centre of $\cH$ is given by 
$Z(\cH) = \FF[X_1^{\pm 1},\dots, X_n^{\pm 1}]^{\fS}$.
For $\ba =(a_1,\dots, a_n)\in (\FF^*)^n$ define the corresponding central character 
$\chi_\ba:\;Z(\cH)\to \FF,$ by restriction from $X_i\mapsto a_i.$
Two characters $\chi_\ba$ and $\chi_{\ba'}$ coincide if and only if $\ba$ and $\ba'$ belong to the same $\fS$-orbit. We can decompose any finite-dimensional representation $M$ of $\cH$ as $M = \bigoplus_\chi M_\chi$, where $\chi$ runs over a set of representatives of $\fS$ -orbits on $\FF^n$ and $M_\chi$ consists of all elements of $M$ which are annihilated by a sufficiently large power of $\bm_\chi = \ker \chi$.

{\bf Conventions:}
For the following, the most interesting cases are those where the components of $\ba=(a_1,\dots, a_n)$ belong to the same multiplicative $q$-orbit; in other words, where there is an $a \in \FF$, such that for each $j=1,\dots n$, we have  $a_j=q^{i_j}a$ for some integer $i_j$. We will therefore stick to these cases. 
Moreover, our constructions in fact turn out to be independent of $a$, so without loss of generality, we chose $a=1$, i.e.  $\ba = (q^{i_1}, \dots, q^{i_n})$ with $\bi=(i_1,\dots, i_n)\in\mZ^n$, and write $\chi_\bi$ for the central character $\chi_\ba$. If $q$ is an $e$th root of unity we usually choose the exponents $i_j$ from the representatives $0,\dots, e-1$ for $\Z/e\Z$.

\begin{definition}\label{idef}
From now on for the rest of this paper, we fix $\bi \in \mZ^n$, viewed as an element of $\Z/e\Z$ if $e \neq 0$.
\end{definition}

\begin{definition}
\label{defcompl}
{
Given a central character $\chi=\chi_\bi$, we define the \emph{completion} $\ccH_\bi$ of $\cH$ with respect to powers of the ideal $\cI_\bm$ of $\cH$ generated by $\bm=\bm_\chi$. We have a decomposition
\begin{equation}\label{compldecomp}\ccH_\bi = \bigoplus_{\bu \in \fS\bi} \ccH_\bi e_\bu\end{equation} 
where 
\begin{eqnarray}
\label{eigenspaces}
\ccH_\bi e_\bu &=&\left \{h \in \ccH_\bi \;\left|
\begin{array}[c]{c}
  \forall m \in \N\; \exists\; N \in \N\;\text{such that }\\ \forall j \in \{1 \dots, n\}\;
h(X_j-q^{u_j})^N\in \bm_\chi^m
\end{array}
\right.
\right\}
\end{eqnarray}
The $e_\bu$ form a complete set of pairwise orthogonal idempotents in $\ccH_\bi$.}
\end{definition}

The following is a direct consequence of \eqref{heckebasis} and the description of the centre:
\begin{lemma} 
\label{basisH}
The following sets 
\begin{eqnarray*}
&\left\{ T_wX_1^{a_1}\cdots X_n^{a_n} e_\bu\left|\bu \in \fS(i_1,\dots, i_n), w \in \fS, a_i \in \Z_{\geq 0}\right .\right \},&\\ &\left\{T_w X_1^{a_1}\cdots X_n^{a_n}e_\bu \left|\bu \in  \fS(i_1,\dots, i_n), w \in \fS, a_i \in \Z_{\leq 0}\right .\right \},&
\end{eqnarray*}
both form a topological basis of $\ccH_\bi$.
\end{lemma}
 Since $T_re_\bu-e_{s_r(\bu)}T_r \in \bigoplus_{\bu \in \fS\bi}\FF[[X_1,\dots, X_n]]$, induction on the length of $w$ shows that another topological $\FF$-basis of $\ccH_\bi$ is given by $$\left\{e_\bu T_wX_1^{a_1}\cdots X_n^{a_n} \mid\bu \in \fS(i_1,\dots, i_n), w \in \fS, a_i \in \Z_{\geq 0}\right \}$$ and similarly for the version with negative powers of the polynomial generators.

In $\ccH_\bi$, we have for  
   for $r=1,\dots, n-1$ the \emph{intertwining elements}
\begin{eqnarray}
\label{Phi}
\Phi_r&=&T_r + \sum_{u_{r+1}\neq u_r}\frac{1-q}{1-X_rX_{r+1}^{-1}}e_\bu + \sum_{u_{r+1}=u_r}e_\bu.
\end{eqnarray}
Their properties were, for instance, studied in \cite{BK}.

For each $w \in \fS$ we fix a reduced expression $w=s_{i_1}\cdots s_{i_r}$ and define $\Phi_{[w]}=\Phi_{i_1}\cdots \Phi_{i_r}$. We indicate by $[w]$ that this does depend on the choice of reduced expression.
It follows then directly that another topological $\FF$-basis of $\ccH_\bi$ is 
\begin{equation}
\left\{e_\bu \Phi_{[w]}X_1^{a_1}\cdots X_n^{a_n}\mid \bu \in \fS(i_1,\dots, i_n), w \in \fS, a_i \in \Z_{\geq 0} \right\}
\end{equation}
and similarly for the version with negative powers of the polynomial generators.

Then, similarly to Corollary~\ref{basisHJ}, we obtain several topological bases of $\bv_J\ccH_\bi$:
\begin{lemma}
\label{basislemma}
Any of the following sets is a  topological $\FF$-basis of $\bv_J\ccH_\bi$:
\begin{eqnarray*}
&\left\{\bv_Je_\bu T_w X_1^{a_1}\cdots X_n^{a_n}\mid\bu \in  \fS(i_1,\dots, i_n), w \in D_J, a_i \in \Z_{\geq 0} \right\},&\\
&\left\{\bv_J e_\bu T_w X_1^{a_1}\cdots X_n^{a_n}\mid\bu \in  \fS(i_1,\dots, i_n), w \in D_J, a_i \in \Z_{\leq 0} \right\},&\\
&\left\{\bv_Je_\bu \Phi_{[w]} X_1^{a_1}\cdots X_n^{a_n}\mid\bu \in  \fS(i_1,\dots, i_n), w \in D_J, a_i \in \Z_{\geq 0} \right\},&\\
&\left\{\bv_J e_\bu \Phi_{[w]} X_1^{a_1}\cdots X_n^{a_n}\mid\bu \in  \fS(i_1,\dots, i_n), w \in D_J, a_i \in \Z_{\leq 0} \right\}.&
\end{eqnarray*}
Analogously, we obtain bases for $\bbv_J\ccH_\bi$ if we replace $\bv$ by $\bbv$.
\end{lemma}




\subsection{A faithful representation of the Hecke algebra}\label{heckefaith}

We now construct a faithful representation of $\cH$ respectively $\ccH_\bi$ which allows us to realise either algebra as a subalgebra of the endomorphisms of some Laurent polynomial respectively power series ring. 

Fix the left ideal $\overline{U}=\left(\sum_{1 \leq i \leq n-1}\cH(T_i+1)\right)$.  We obtain the following two faithful representations; the second one agrees with  \cite[3.1.4]{Rouquier}, see e.g. \cite[2.17 with 3.2.5]{Rouquier2} for the faithfulness. 

\begin{proposition}[Faithful representation of Hecke algebra I]\hfill
\label{Ben}
\begin{enumerate}[i.)]
\item The natural action of  $\cH$ on $\cH/\overline{U}$ by left multiplication is faithful. 
\item This representation is canonically isomorphic to $\FF[X_1^{\pm 1}, \dots,X_n^{\pm 1}]\bbv$, where the generators $X_i^\pm$, $1\leq i\leq n$, and $T_i+1$, $1\leq i\leq n-1$ act just by left multiplication respectively by   
\begin{eqnarray}
\label{faithfulHecke1}
 (T_i+1)f\bbv &=&\frac{qX_{i+1}-X_i}{X_{i+1}-X_i}(f-{s_i}(f))\bbv, 
 \end{eqnarray}
where ${s_i}(f)$ is the Laurent polynomial $f$ with $X_{i+1}$ and $X_i$ interchanged. 
\end{enumerate}
\end{proposition}

By twisting with the automorphism $\sharp$ from \eqref{eqhash} we obtain:
\begin{proposition}[Faithful representation of Hecke algebra II]\hfill
\label{Bensign}

Let $U=\sum_{1\leq i <n}\cH(T_i-q)$. Then there is a faithful representation of $\cH$ on $$\cH/U\cong \FF[X_1^{\pm 1}, \dots,X_n^{\pm 1}]\bv.$$ 
Explicitly, the action is given by 
\begin{eqnarray}
\label{faithfulHecke2}
 (T_i-q)f\bv &=&\frac{qX_{i}-X_{i+1}}{X_{i+1}-X_{i}}(f-{s_i}(f))\bv.
 \end{eqnarray}
\end{proposition}

We denote by $\oFH$ respectively $\FH$ the faithful representations from \eqref{faithfulHecke1} and \eqref{faithfulHecke2} respectively. The following is immediate.

\begin{corollary}\label{normalhash}
We have an isomorphism of $\cH$-modules 
\begin{eqnarray*}
\oFH \to {}^\sharp(\FH)&\text{given by}& f\bv \mapsto  f^\sharp\bbv.
\end{eqnarray*}
\end{corollary}

Completion gives us faithful representations of the completed algebra:
\begin{corollary}
\label{faithfulcompl} 
\begin{enumerate}[i.)]
\item 
There is a  faithful representation of $\ccH_\bi$ on $$\coFH_\bi=\ccH_\bi\otimes_\cH \cH/\overline{U}\cong \bigoplus_{\bu \in \fS\bi}\FF[[X_1, \dots,X_n]] e_\bu\bbv$$ by completing the representation from Proposition~\ref{Ben} with respect to the maximal ideal generated by the elements $(X_r-i_r)e_\bu$, $1\leq r\leq n$.
\item There is a  faithful representation of $\ccH_\bi$ on $$\cFH_\bi=\ccH_\bi\otimes_\cH \cH/U\cong \bigoplus_{\bu \in \fS\bi}\FF[[X_1^{- 1}, \dots,X_n^{-1}]] e_\bu\bv$$ by completing the representation from Proposition~\ref{Bensign} with respect to the ideal generated by  $Y^{-1}_r e_\bu=(X^{-1}_r+i_r)e_\bu$, $1\leq r\leq n$.
\end{enumerate}
\end{corollary}

The definitions directly imply the following connection:

\begin{corollary}
\label{corhash}
There is an isomorphism of $\ccH_\bi$-modules 
\begin{eqnarray*}
\cFH_\bi\;\cong\;{}^\sharp(\coFH_{-\bi})&\text{via}&f\bv\;\mapsto\; f^\sharp\bbv,
\end{eqnarray*}
identifying $\FF[[X_1^{-1}, \dots,X_n^{-1}]] e_\bu\bv$ with $\FF[[X_1, \dots,X_n]] e_{-\bu}\bbv$, where the minus signs applies to all entries, i.e. $-\bi=(-i_1,\dots, -i_n)$ and $-\bu=(-u_1,\ldots, -u_n)$.
\end{corollary}
\proof
The first statement is clear. The identification follows directly from the fact that $X_i^\sharp =X_i^{-1}$ for all $i$, Corollary~\ref{faithfulcompl} and the definition in \eqref{eigenspaces}.
\endproof

We finish this section with a few important explicit formulae for the action of the intertwining elements from \eqref{Phi}.

\begin{lemma}
\label{erstesLemma}
For any idempotent $e_\bu$ as in Definition~\ref{defcompl}, the following equalities hold. For $1\leq r\leq n-1$, we have
\begin{eqnarray}
\label{firstformula}
e_{s_r\cdot\bu}\Phi_r\bbv&=&
\begin{cases}
\frac{X_r-qX_{r+1}}{X_{r+1}-X_r}e_{s_r\cdot\bu}\bbv& \hbox{if $u_{r+1}\neq u_r$,}\\ 
0&\text{otherwise},
\end{cases}
\end{eqnarray}
and
\begin{eqnarray}\label{phiaction2}
e_{s_r\cdot\bu}\Phi _r(X_{r+1}-X_r)\bbv&=&
\begin{cases}
(qX_{r+1}-X_r)e_{s_r\cdot\bu}\bbv& \hbox{if $u_{r+1}\neq u_r$,}\\ 
2(qX_{r+1}-X_r)e_\bu\bbv& \hbox{if $u_{r+1}= u_r$.}
\end{cases}
\end{eqnarray}
\end{lemma}

\begin{lemma}
For any idempotent $e_\bu$ as in Definition~\ref{defcompl}, the following equalities hold: For $1\leq r\leq n-1$, we have
\begin{equation}\label{phiaction3}
e_{s_r\cdot\bu}\Phi_r \bv =  \left\{\begin{array}{ll} \frac{X_{r+1}-qX_{r}}{X_{r+1}-X_r}e_{s_r\cdot\bu}\bv & \hbox{if $u_{r+1}\neq u_r$,}\\ (q+1) e_{s_r\cdot\bu}\bv& \hbox{if $u_{r+1}= u_r.$}\end{array}\right.\end{equation} and
\begin{equation}\label{phiaction4}
e_{s_r\cdot\bu}\Phi (X_{r+1}-X_r) \bv=  \left\{\begin{array}{ll} (qX_{r}-X_{r+1})e_{s_r\cdot\bu}\bv & \hbox{if $u_{r+1}\neq u_r,$} \\ (q-1)(X_{r+1}+X_r)e_\bu\bv& \hbox{if $u_{r+1}= u_r.$}\end{array}\right.\end{equation}
\end{lemma}

\section{Affine Schur algebra}
In this section we recall the {\it (affine) Schur algebra} and construct a faithful representation for this algebra as well. We describe in detail the basis used by Vign\'eras. These two tools allow us to give an alternative basis together with a set of algebra generators more in the spirit of the geometric basis of the quiver Schur algebra from \cite{SW}. This will then finally allow us to connect the two algebras in the last section.

\begin{definition}
\label{affSchur}
The (affine) Schur algebra $\cS$ is defined as 
\begin{eqnarray*}
\cS&=&\End_{\FF[\I\backslash \G/\I]}\left(\bigoplus_{J \subseteq \bbI} \FF[\rmP^J\backslash \G/\I]\right)\cong\End_{\cH}\left(\bigoplus_{J \subseteq \bbI} \bv_J \cH\right),
\end{eqnarray*}
where $P^J$ denotes the standard parahoric subalgebra (containing $I$) attached to $J$, and the isomorphism 
is that from Lemma~\ref{heckepres}. 
The product of two elements, $f,f'$ in $\cS$ is denoted by $f\circ f'$ or just $ff'$. 
\end{definition}

We start our study of $\cS$ by recalling a basis from \cite[4.2.13]{V} (which generalizes the basis for finite Schur algebras from \cite[Theorem~4.7]{Mathasbook}).
\begin{lemma} 
\label{VSchurbasis}
A basis of $\cS$ is given by
$\{\bb^d_{K,J} \mid J, K \subseteq \bbI, d \in D_{K,J}\}$,
where $\bb^d_{K,J} \in \Hom_{\cH_\bi}( \bv_J \cH, \bv_K \cH)$ is defined by 
\begin{equation}
\label{bbdef}\bb^d_{K,J}(\bv_J) = \sum_{w \in W_K d W_J} T_w. 
\end{equation}
\end{lemma}

\begin{remark}
\label{stupidremark}
Note that  \eqref{bbdef} is indeed well-defined, since $\sum_{w \in W_K d W_J} T_w\in \bv_K \cH$ by Lemma~\ref{cyclic}. Moreover, any element in $\Hom_{\cH}(\bv_J\cH,\bv_K\cH)$, and in particular $\bb^d_{K,J}$, is already uniquely determined by its image of $\bv_J$. 
\end{remark}

\begin{example}
\label{chicken}
If, for instance, $\bbI=\{s_1,s_2\}$ and $K=\{s_1\}$, $J=\{s_2\}$, then $D_{K,J} \supset D^\bbI_{K,J}=\{1,s_2s_1\}$ and for these two shortest double coset representatives and we have
$\bb^1_{K,J}=1+T_{1}+T_2+T_2T_1$ and $\bb^{s_2s_1}_{K,J}=T_2T_1+T_1T_2T_1$.
Note that we just sum over all basis elements from a fixed double coset. 
\end{example}

\begin{example}
\label{specialmorph}
If  $J\subseteq K$, we have $\bb^1_{K,J}\bv_J=\sum_{w\in\W_K}T_w=\bv_K$ and $\bb^1_{J,K}\bv_K=\sum_{w\in\W_K}T_w=\bv_J(\sum_{d'\in D^K_{J,\emptyset}}T_{d'})$. Hence  $\bb^1_{K,J}$ is the projection sending $ \bv_Jh$ to $\bv_Kh$ and $\bb^1_{J,K}$ is the inclusion sending $\bv_Kh$ to $\bv_J(\sum_{d'\in D^K_{J,\emptyset}}T_{d'})h$, for $h\in\cH$. 
\end{example}

We like to point out that the labelling of the basis vector $\bb^d_{K,J}$ involves a choice $d$ of a shortest double coset representative, although the basis element itself only depends on the coset containing $d$. In particular the basis can be relabelled when chosing different representatives.  
If, for instance, for $K,J\subset \bbI, w \in \fS, p \in \fP$ we define
the element  $\bb^{wp}_{K,J} \in \Hom_{\cH_\bi}( \bv_J \cH, \bv_K \cH)$ via
\begin{equation*}
\bb^{wp}_{K,J}( \bv_J) = \sum_{v \in W_K wp W_J} T_v,
\end{equation*}
then $\bb^{wp}_{K,J} = \bb^d_{K,J}$ for $d\in D_{K,J}\cap W_K wp W_J$, and  with the choice of double coset representatives from Section~\ref{otherreps} we directly obtain the following.

\begin{lemma}
\label{samebasis}
Both sets
\begin{equation}
\{\bb^{wp}_{K,J} \mid J, K \subseteq \bbI, w \in D^\bbI_{K,J}, p \in \DM_{d^{-1}K\cap J}\},
\end{equation}
and 
\begin{equation}
\{\bb^{wp}_{K,J} \mid J, K \subseteq \bbI, w \in D^\bbI_{K,J}, p \in \AM_{d^{-1}K\cap J}\}
\end{equation}
form the same basis of $\cS$ as the one in Lemma~\ref{VSchurbasis}, just labelled differently.
\end{lemma}

\subsection{A faithful representation of $\mathcal{S}$}
\label{secfaithful}
To construct a faithful representation of the Schur algebra we enlarge the space $\FH$. 

For any parabolic subgroup ${W_K}$ of $W$ with $K\subseteq \bbI$, let $\FF[X_1^{\pm 1}, \dots,X_n^{\pm 1}]^{W_K}$ denote the $W_K$-invariants under the usual permutation action. We set 
\begin{eqnarray*}
\FS^K&=&\FF[X_1^{\pm 1}, \dots,X_n^{\pm 1}]^{W_K}\bv^{(K)},
\end{eqnarray*}
where $\bv^{(K)}= \bv$ and the superscript $(K)$ is just a formal index.
We have the following characterisation of invariants:

\begin{lemma}\label{qifsymm}
Let $f\bv \in \FF[X_1^{\pm 1}, \dots,X_n^{\pm 1}]\bv$ and $K\subseteq \bbI$. Then 
$$ f\in \FS^K\quad\Leftrightarrow \quad (T_i-q)f\bv = 0 \text{ for all } s_i \in K.$$
\end{lemma}

\proof A direct computation shows that, for $s_i \in K$,
\begin{eqnarray*}
(T_i-q)f\bv &=& T_if\bv-qf\bv = \left({s_i}(f)T_i+(q-1)X_{i+1}\frac{f-{s_i}(f)}{X_{i+1}-X_i} -qf\right)\bv\\\
&= &\frac{(X_{i+1}-qX_i)(s_i(f)-f)}{X_{i+1}-X_i}\bv.
\end{eqnarray*}
Hence $(T_i-q)f\bv=0$ if and only if $f=s_i(f)$.
\endproof

The following is the main theorem of this section.
\begin{theorem}
\label{faithfulrep}
There is a faithful representation $\rho$ of $\cS$ on 
$$\FS=\bigoplus_{K\subseteq \bbI}\FS^K .$$
 In this representation a basis element $\bb^d_{K,J}$ of $\cS$ acts via
\begin{eqnarray}\label{rhodef} 
\rho(\bb^d_{K,J})f\bv^{(J)}& = &\sum_{a \in D^K_{\emptyset,K\cap dJ}}T_aT_df\bv^{(K)}.
\end{eqnarray}
\end{theorem}
The proof will follow directly from the next three lemmas. The first of these makes sure that the right hand-side of \eqref{rhodef} is at least in the correct space.

\begin{lemma}\label{inrightinv}
For $J,K\subseteq \bbI,d\in D_{K,J}$ and $f \in \FF[X_1^{\pm 1}, \dots,X_n^{\pm 1}]^{\W_J}$, we have
\begin{eqnarray}
\sum_{a \in D^K_{\emptyset,K\cap dJ}}T_aT_df\bv^{(K)}&\in&\FS^K.
\end{eqnarray}
\end{lemma}

\proof
In view of Lemma~\ref{qifsymm}, it suffices to check that  for all $s_i\in K$ we have
$\sum_{a \in D^K_{\emptyset,K\cap dJ}}(T_i-q)T_aT_df\bv^{(K)} = 0.$
The left hand side equals 
\begin{eqnarray}
\label{twosummands}
\sum_{\substack{a \in D^K_{\emptyset,K\cap dJ}\\ s_ia\in D^K_{\emptyset,K\cap dJ}}}(T_i-q)T_aT_df\bv^{(K)}&+& \sum_{\substack{a \in D^K_{\emptyset,K\cap dJ} \\ s_ia\notin D^K_{\emptyset,K\cap dJ} }}(T_i-q)T_aT_df\bv^{(K)}. 
\end{eqnarray}
Denote by $S_1$ and $S_2$ the two summands in \eqref{twosummands} respectively.

In the first summand $S_1$, the summands appear in pairs $a,s_ia$. Since we have $(T_i-q)(T_a+T_{s_ia}) = 0$, they cancel and so $S_1=0$. In the second summand $S_2$, we have  $a \in D^K_{\emptyset,K\cap dJ}$ but $s_ia\notin D^K_{\emptyset,K\cap dJ}$. Then Deodhar's Lemma, see e.g. \cite[Lemma 2.1.2]{GeckPfeiffer}, shows that there exists $s_j\in K \cap dJ$ such that $s_ia = as_j$, and that, in particular, $l(s_ia)>l(a)$. In this case, $(T_i-q)T_a = T_a(T_j-q)$. Again using Deodhar's Lemma, we see that $T_jT_d = T_dT_k$ for some $s_k\in d^{-1}K\cap J$, and thus
\begin{eqnarray*}
S_2& \in& \sum_{s_j\in K\cap dJ} \cH(T_j-q)T_df\bv^{(K)} \subseteq \cH T_d  \sum_{s_k\in J} (T_k-q)f\bv^{(K)} = 0
\end{eqnarray*}
 by Lemma~\ref{qifsymm}, since $f \in\FF[X_1^{\pm 1}, \dots,X_n^{\pm 1}]^{\W_J}$. Hence we are done.
\endproof

\begin{lemma}\label{rhorep}
The assignment \eqref{rhodef} defines a representation of $\cS$ on $\FS$.
 \end{lemma}

\proof
It suffices to check that, for basis vectors as in \eqref{VSchurbasis},
\begin{eqnarray}
\label{product}
\rho(\bb^{d_2}_{L,K})\rho(\bb^{d_1}_{K,J})= \rho(\bb^{d_2}_{L,K}\bb^{d_1}_{K,J}).
\end{eqnarray} 
We start with some preparation.  Using the basis from Lemma~\ref{VSchurbasis}, we write 
$\bb^{d_2}_{L,K}\bb^{d_1}_{K,J} = \sum_{d \in D_{L,J}} c_d \bb^d_{L,J}$
for some coefficients $c_d\in \FF$.
On the one hand, 
\begin{eqnarray*}
\bb^{d_2}_{L,K}\bb^{d_1}_{K,J}(\bv_J)&=&\left(\sum_{d \in D_{L,J}} c_d  \sum_{b''\in D^J_{\emptyset, L\cap dJ}}T_{b''}T_d\right)\bv_J.
\end{eqnarray*}
On the other hand, repeatedly using \eqref{eqaddlength}, we obtain
\begin{equation*}\begin{split}
\bb&^{d_2}_{L,K}\bb^{d_1}_{K,J}\bv_J 
\overset{\eqref{bbdef}}{=}\bb^{d_2}_{L,K}\sum_{w \in W_K d_1 W_J} T_w\quad=\quad\bb^{d_2}_{L,K} \bv_K T_{d_1}\sum_{a\in D^J_{d_1^{-1}K\cap J,\emptyset}}T_a.
\end{split}\end{equation*}
Using \eqref{bbdef}, this equals
\begin{equation*}\begin{split}
\sum_{w \in W_{L} d_2 W_K} T_w T_{d_1}\sum_{a\in D^J_{d_1^{-1}K\cap J,\emptyset}}T_a
&= \sum_{b \in D^{L}_{\emptyset, d_2K\cap L}}T_b T_{d_2} \bv_KT_{d_1}\sum_{a\in D^J_{d_1^{-1}K\cap J,\emptyset}}T_a\\
&=\left(\sum_{b \in D^{L}_{\emptyset, d_2K\cap L}}T_{b} T_{d_2} \sum_{b'\in D^J_{\emptyset,K\cap d_1J}}T_{b'} T_{d_1}\right) \bv_{J}.
\end{split}\end{equation*}
Thus, 
$\sum_{b \in D^{L}_{\emptyset, d_2K\cap L}}T_{b} T_{d_2} \sum_{b'\in D^J_{\emptyset, K\cap d_1J}}T_{b'} T_{d_1}   -  \sum_{d \in D_{L,J}} c_d  \sum_{b''\in D^J_{\emptyset,L\cap dJ}}T_{b''}T_d$
is contained in $\sum_{s_i\in J}\cH(T_i-q)$.
To verify formula \eqref{product}, we now calculate 
$\rho(\bb^{d_2}_{L,K})\rho(\bb^{d_1}_{K,J})f\bv^{(J)} =\rho(\bb^{d_2}_{L,K}) \sum_{a \in D^K_{\emptyset,K\cap d_1J}}T_aT_{d_1}f\bv^{(K)}$
%
which equals
\begin{eqnarray}
  \sum_{b \in D^{L}_{\emptyset, d_2K\cap L}}\!\!T_{b} T_{d_2} \sum_{b' \in D^K_{\emptyset,K\cap d_1J}}T_{b'}T_{d_1}f\bv^{(L)},\quad\;\;\label{Teil1}
\end{eqnarray}
 where $f\in\FF[X_1^{\pm 1}, \dots,X_n^{\pm 1}]^{\W_J}$, and 
\begin{eqnarray}
\label{Teil2}
\sum_{d \in D_{L,J}} c_d \rho(\bb^d_{L,J})f \bv^{(J)}  = \sum_{d \in D_{L,J}} c_d  \sum_{b''\in D^J_{\emptyset,L\cap dJ}}T_{b''}T_df\bv^{(L)}.
\end{eqnarray}
Thus \eqref{Teil1}-\eqref{Teil2} equals $\big( \rho(\bb^{d_2}_{L,K})\rho(\bb^{d_1}_{K,J})-\sum_{d \in D_{L,J}} c_d \rho(\bb^d_{L,J})\big) f \bv^{(J)}$
\begin{eqnarray*}
=\;\left(\sum_{b \in D^{L}_{\emptyset, d_2K\cap L}}T_{b} T_{d_2} \sum_{b'\in D^J_{\emptyset,K\cap d_1J}}T_{b'} T_{d_1}   -  \sum_{d \in D_{L,J}} c_d  \sum_{b''\in D^J_{\emptyset,L\cap dJ}}T_{b''}T_d\right)f \bv^{(L)}.
\end{eqnarray*}
By  the above, this is, however, contained in $\sum_{s_i\in J}\cH(T_i-q)f \bv^{(L)}$ and hence  must be zero by  Lemma~\ref{qifsymm}, as $f \in \FF[X_1^{\pm 1}, \dots,X_n^{\pm 1}]^{\W_J}$.
\endproof

\begin{lemma}\label{rhofaith}
The representation $\rho$ from \eqref{rhodef} is faithful.
\end{lemma}

\proof
Let $J, K \subset\bbI$ and $Z=\sum_{\substack{J,K \subseteq \bbI \\ d \in D_{K,J}}} c_d\; \rho(\bb^d_{K,J})$ with arbitrary $c_d\in\FF$. Then it is  
enough to show that $Z=0$ implies $\sum_{d \in D_{K,J}} c_d\bb^d_{K,J} = 0$ for each pair $J,K\subset \bbI$.  If $h \in \cH$ and $L\subset\bbI$, then 
\begin{eqnarray}
\label{hvanish}
h\;\FS^L = 0&\hbox{ implies }&h\left(\sum_{w \in \W_L}T_w\right) = 0.
\end{eqnarray} 
Indeed,  assume $h\;\FS^L = 0$. Since $(T_i-q)\sum_{w \in \W_L}T_w=0$ for any $s_i\in L$,  Lemma~\ref{qifsymm} yields that
$\sum_{w \in \W_L}T_w \FF[X_1^{\pm 1}, \dots,X_n^{\pm 1}]\bv\subseteq \FF[X_1^{\pm 1}, \dots,X_n^{\pm 1}]^{\W_L}\bv.$
In particular, 
\begin{eqnarray*}
\left(h\sum_{w \in \W_L}T_w\right)\FF[X_1^{\pm 1}, \dots,X_n^{\pm 1}]\bv&\subseteq h\;\FF[X_1^{\pm 1}, \dots,X_n^{\pm 1}]^{\W_L}\bv\stackrel{\text{(ass.)}}{=}0.
\end{eqnarray*}
Together with the  faithfulness of the representation of $\cH$ on $\FF[X_1^{\pm 1}, \dots,X_n^{\pm 1}]\bv$ in Lemma~\ref{Bensign}, the claim in \eqref{hvanish} follows.

Now suppose that $Z = 0$.
Projecting onto the different summands of $\FS$ gives that for any $J,K\subseteq \bbI$ we have
\begin{eqnarray*}
\sum_{ d \in D_{K,J}} c_d \rho(\bb^d_{K,J}) \FS^J&= &\left(\sum_{d \in D_{K,J}} c_d  \sum_{a\in D^J_{\emptyset,K\cap dJ}}T_aT_d \right)\FS^K =0.
\end{eqnarray*}

Using \eqref{v}, observation \eqref{hvanish} implies that 
 $$0= \sum_{d \in D_{K,J}} c_d  \sum_{a\in D^J_{\emptyset,K\cap dJ}}T_aT_d\sum_{w \in \W_J}T_w = \sum_{d \in D_{K,J}} c_d\bb^d_{K,J}\bv_J.$$
Thus, $\sum_{d \in D_{K,J}} c_d\bb^d_{K,J} = 0$ for every $J,K\subseteq \bbI$, which completes the proof.
\endproof
Theorem~\ref{faithfulrep} is proved.

\subsection{Generators and (some) relations for $\cS$}\label{sgenrelsec}
In this section, we determine a nice generating set for the algebra $\cS$. We start with a few technical tools.
\begin{lemma}
Let $K_1,K_2 \subseteq \bbI, d \in D_{K_2,K_1}$ and let $J=K_1\cap d^{-1}K_2$.
Then \begin{eqnarray}\label{eqschurrel}\bb^d_{K_2,K_1} &=& \bb^1_{K_2,dJ}\bb^d_{dJ,J}\bb^1_{J,K_1}.\end{eqnarray}
\end{lemma}

\proof
We first note that for $|J|=|dJ|$ (which holds for $J$ as in the lemma), 
\begin{equation}
\label{niceformula1}
\bb^d_{dJ,J}\bv_{J}\quad =\quad\bv_{dJ} T_d \sum_{b\in D^J_{J\cap J, \emptyset}}T_b \quad=\quad\bv_{dJ} T_d 
\end{equation}
   and that, if $J\subseteq K$, then
\begin{equation}
\label{niceformula2}
\bb^1_{K,dJ}\bv_{dJ}  = \sum_{w \in \W_K1\W_{dJ}}T_w = \bv_K.
\end{equation}

We now apply the right hand side of \eqref{eqschurrel} to $\bv_{K_1}$ and deduce
\begin{eqnarray*}
&&\bb^1_{K_2,dJ}\bb^d_{dJ,J}\bb^1_{J,K_1}(\bv_{K_1})
\quad=\quad\bb^1_{K_2,dJ}\bb^d_{dJ,J}\left( \sum_{w \in \W_{K_1}}T_w\right)\\
&=&\bb^1_{K_2,dJ} \bb^d_{dJ,J}\left( \bv_{J} \sum_{a \in D^{K_1}_{J, \emptyset }} T_a\right)
\quad=\quad\bb^1_{K_2,dJ}\left(\bv_{dJ} T_d \sum_{a \in D^{K_1}_{J, \emptyset }} T_a\right)\\
&\stackrel{\eqref{niceformula2}}{=}&\bv_{K_2}T_d \sum_{a \in D^{K_1}_{d^{-1}K_2\cap K_1, \emptyset }} T_a
\quad=\quad\bb^d_{K_2,K_1}\bv_{K_1} \qedhere
\end{eqnarray*}
as desired.
\endproof

\begin{corollary}
\label{corgenerating}
The Schur algebra $\cS$ is generated (as an algebra) by
\begin{equation}\label{eqschurgens}
\{\bb^1_{K,J}, \bb^w_{wJ,J} \mid J,K \subseteq \bbI, w \in \W \textrm{ with } |J|=|wJ| \}.
\end{equation}
\end{corollary}
Based on this,  we will give another generating set in Proposition~\ref{generatingbetter}.


\subsection{The centre of $\cS$}
We next show that the centre of the Schur algebra is just given by multiplication with elements from the centre of the Hecke algebra.
\begin{lemma}\label{Scentre}
The centre $Z(\cS)$ of $\cS$ equals $Z(\cH)$ in the sense that
\begin{eqnarray*}
Z(\cS)=\{\cdot z\mid z\in Z(\cH)\}&\subseteq& \End_{\cH}(\bigoplus_{J \subseteq \bbI} \bv_J \cH)=\cS.
\end{eqnarray*}
\end{lemma}

\proof
For the inclusion $\supseteq$, it is clear that multiplication with $z \in Z(\cH)$ commutes with any $\cH$-endomorphism of $\bigoplus_{J \subseteq \bbI} \bv_J \cH$ and hence belongs to $\cS$. It furthermore commutes with any element in $\cS$ and hence belongs to $Z(\cS)$.

For the inclusion $\subseteq$, let $f\in Z(\cS)$ and test with the generators given in \eqref{eqschurgens}.
For $\bb^d_{dJ, J}$  we see that 
\begin{equation*}\begin{split}
\bb^d_{dJ, J}\circ f \left(\sum_{K\subseteq\bbI} \bv_K\right) 
&=\bb^d_{dJ, J} \left(\sum_{L\in\subseteq\bbI} \bv_Lh_L\right) 
\stackrel{\eqref{niceformula1}}{=} \bv_{dJ}T_d h_{J}
\end{split}\end{equation*}
for some $h_L\in\cH$. On the other hand, since $f$ is central, 
$$\bb^d_{dJ, J}\circ f \left(\sum_{K\subseteq\bbI} \bv_K\right)=f\circ \bb^d_{dJ, J}\left (\sum_{K\subseteq\bbI} \bv_K\right)  \stackrel{\eqref{niceformula1}}{=} f(\bv_{dJ} T_d)=f(\bv_{dJ})T_d.$$
By comparing the two formulae,  $f(\bv_J) \in \bv_{dJ}\cH$ and thus $f \in \sum_{J\subseteq\bbI} \End_\cH(\bv_J \cH)$ is a diagonal endomorphism. Therefore, for any $K\subseteq \bbI$,
\begin{eqnarray}
\label{comm}
f(\bv_L)=\bv_{L}h_L,&\text{  and moreover   } &h_{dJ}T_d = T_d h_{J},
\end{eqnarray}
again by comparing the above two formulae.

Now we test with $\bb^1_{K,J}$ for $J \subseteq K$ and obtain
$$f \circ\bb^1_{K,J} \left(\sum_{L\subseteq\bbI} \bv_L\right)= f \circ\bb^1_{K,J} \left( \bv_J\right)\stackrel{\eqref{niceformula2}}{=}f\left(\bv_K\right) \stackrel{\eqref{comm}}{=} \bv_Kh_K.$$
Since $f$ is central, this equals
$$ \bb^1_{K,J}\circ f \left(\sum_{L\subseteq\bbI} \bv_L\right)  \stackrel{\eqref{comm}}{=} 
\sum_{L\subseteq\bbI} \bb^1_{K,J}(\bv_L h_L)=\sum_{L\subseteq\bbI} \bb^1_{K,J}(\bv_L) h_L=\bv_Kh_J.$$
As all $J \subseteq \bbI$ contain $\emptyset$, this yields, for all $J,K\subseteq \bbI$, the equality
$h_J = h_K =:h.$
Using \eqref{comm} we obtain $T_d h = h T_d$ for all $T_d \in \cH$, and thus $f=\cdot h \in Z\left(\cH\right)$.
\endproof

\subsection{Another basis for $\cS$}
\label{secgoodbasis}
The main goal of this section is to construct an alternative basis of $\cS$ which mimics the geometric basis of the quiver Schur algebra defined in \cite{SW}. We start by investigating the sets $D_{K_2,K_1}$ further.

\begin{lemma}\label{tech2}
Assume $K\subseteq \bbI$ and
suppose $v=wp$ with $w \in \fS$ and $p = X_1^{a_1}\cdots X_n^{a_n} \in \fP$ with $a_i \leq 0$ for $i=1,\dots n$.
Then $T_v = \sum_{u \in \fS}c_uT_up$, for some coefficients $c_u\in\FF$.
\end{lemma}

\proof
Note first that, using induction on the degree of $p$ and commutativity of $X_1, \dots, X_n$, it suffices to check this for $p=X_i^{-1}$. So suppose $v=wX_i^{-1}$. Then by \eqref{tau} $v=ws_{i-1}s_{i-2}\cdots s_{1}\ourOmega^{-1} s_{n-1}s_{n-2}\cdots s_{i}$, and thus $v$ has a reduced expression $\tilde{v}\ourOmega^{-1}s_{n-1}\cdots s_{i}$ with some $\tilde{v}\in \fS$. Thus $T_v= T_{\tilde{v}}T_{\ourOmega}^{-1}T_{n-1}\cdots T_{i} =a T_{\tilde{v}}T_1T_2\cdots T_{i-1}X_i^{-1}$ where $a$ is a power of $q$. Writing $aT_{\tilde{v}}T_1T_2\cdots T_{i-1} = \sum_{u \in \fS} c_u T_u$, the claim follows.\endproof

\begin{lemma}
\label{dwp}
Let $d \in D_{K_2,K_1}$, $J=K_1\cap d^{-1}K_2$ and write $d=wp$ with $w\in \fS$ and $p \in \fP$. Then $wJ=dJ$ and $p \in \fP^{\W_J}$.
\end{lemma}

\proof
We have $d(K_1\cap d^{-1}K_2) = dK_1\cap K_2$ and thus for any $i \in (K_1\cap d^{-1}K_2)$, $ds_id^{-1} = s_j$ for some $j \in dK_1\cap K_2$. On the other hand, 
$$ds_id^{-1} = wps_ip^{-1}w^{-1} = ws_iw^{-1}ws_ips_ip^{-1}w^{-1} = vh,$$
where $v=ws_iw^{-1} \in \fS$ and $h=ws_ips_ip^{-1}w^{-1}\in \fP$, hence $v=s_j, h=1$.
Now $ws_ips_ip^{-1}w^{-1}= 1$ if and only if $s_ips_ip^{-1} = 1$ if and only if $p \in \fP^{\W_{K_1\cap d^{-1}K_2}} = \fP^{\W_J}$. Moreover, $v=ws_iw^{-1} = s_j$ implies $wJ=dJ$.
\endproof

Let $d,w,p$ be as in Lemma~\ref{dwp} and let $d'\in D_{dJ,J}$ be the shortest double coset representative for the coset of $w$. Since $p$ commutes with $\W_J$, the $(\W_{dJ},\W_J)$-double coset defined by $d=wp$ is the same as the one  defined by $d'p$. Hence, by the previous paragraph, and \eqref{eqschurrel}, any basis element can be written as 
\begin{eqnarray}\label{fullfactor}
\bb^{d}_{K_2,K_1}=\bb^{wp}_{K_2,K_1}=\bb^1_{K_2,dJ}\bb^d_{dJ,J}\bb^1_{J,K_1}= \bb^1_{K_2,dJ}\bb^{d'p}_{dJ,J}\bb^1_{J,K_1} 
\end{eqnarray}
for $J=d^{-1}K_2\cap K_1$. Keeping this notation, we next obtain from Lemma~\ref{tech2},
\begin{equation}\label{fullfactor2}
\begin{split}
 \bb^{d'p}_{dJ,J}(\bv_J)& =\bv_{dJ}T_{wp}=\bv_{dJ}\left(\sum_{u \in \fS} c_uT_u\right) p\\
&= T_{wp}\bv_J = \left(\sum_{u \in \fS} c_uT_u\right) p\bv_{J} = \left(\sum_{u \in \fS} c_uT_u\right)\bv_{J}p
\end{split}
\end{equation}
where, in the last equality,  we have used that $p \in \fP^{\W_J}$ by Lemma~\ref{dwp}. Hence
$$\bv_{dJ}\left(\sum_{u \in \fS} c_uT_u\right)=\left(\sum_{u \in \fS} c_uT_u\right)\bv_{J}$$
meaning that left multiplication by $(\sum_{u \in \fS} c_uT_u)$ is in $\Hom_\cH(\bv_J\cH,\bv_{dJ}\cH)$.

Using these relations, we will now give a different basis of $\cS$, which will be more convenient to work with later on.

\begin{proposition}
\label{Propschurbasis2}
Let $K_1,K_2\subset\bbI$. Then the set
\begin{equation}
\label{schurbasis2}
\mathcal{B}_{K_2,K_1}=\left\{  \bb^1_{K_2,wJ}\bb^w_{wJ,J}\bb^p_{J,J}\bb^1_{J,K_1}  \left|
 \begin{array}{l}
 J = K_1\cap w^{-1}K_2,\\
w \in D^\bbI_{K_2,K_1},  p \in \DMJ
 \end{array}
 \right.\right\},\end{equation}
is a basis of the space of homomorphisms $\Hom_\cH(\bv_{K_1}\cH,\bv_{K_2}\cH)$ and hence 
\begin{eqnarray}
\label{schurbasisB}
\cB^\cS&=& \bigcup_{(K_1,K_2)}\mathcal{B}_{K_1,K_2}
\end{eqnarray}
is a basis of the Schur algebra $\cS$.  
\end{proposition}

\proof
We first compute the evaluation of these elements on $\bv_{K_1}$. We have
\begin{eqnarray*}
&&\bb^1_{K_2,wJ}\bb^w_{wJ,J}\bb^p_{J,J}\bb^1_{J,K_1} (\bv_{K_1})\quad=\quad
  \bb^1_{K_2,wJ}(T_wb^p_{J,J}\bv_J\sum_{a\in D^{K_1}_{J,\emptyset}}T_a)\\
&=& \bb^1_{K_2,wJ}(T_w \sum_{b \in D^J_{\emptyset, J\cap p'J}}T_bT_{p'}\bv_J\sum_{a\in D^{K_1}_{J,\emptyset}}T_a)\\
&=&\bb^1_{K_2,wJ}\left(\sum_{b \in D^{wJ}_{\emptyset, wJ\cap wp'J}}T_b T_wT_{p'}\bv_{K_1}\right)
\quad=\quad\bb^1_{K_2,wp'J}(T_wT_{p'}\bv_{K_1}),
\end{eqnarray*}
where $p'$ is the unique element in $\W_Jp\W_J \cap D_{J,J}$. Now, let $d$ be the unique element in $\W_{K_2}wp'\W_{K_1}\cap D_{K_2,K_1}$ and write $wp'=dv$ with $v \in \W_{K_1}$ and $l(dv) = l(d)+l(v)$.
We would like to show that \begin{equation}\label{basechange} \bb^1_{K_2,wJ}\bb^w_{wJ,J}\bb^p_{J,J}\bb^1_{J,K_1} = c_d \bb^d_{K_2,K_1} + \sum_{d' \in D_{K_2,K_1},d'>d}c_{d'}\bb^{d'}_{K_2,K_1}\end{equation} for some $c_{d'}\in \FF$ with $c_d\not=0$. Then the proposition follows from Lemma~\ref{VSchurbasis}.

{\bf Claim 1:} For any $x,y\in \W$ we have $T_xT_y = q^{a(x,y)}T_{xy}+ (q-1)\sum_{z>xy}c_zT_z$ for some $c_{z}\in \FF$ and $a(x,y)=\frac{1}{2}(l(x)+l(y)-l(xy))$.
\proof
This formula is deduced in the proof of  \cite[Proposition 1.16]{Ma} for lengths instead of Bruhat orders and only for $\fS$. However, replacing the permutation realisation of $\fS$ by the realisation of $\W$ as permutations on $\Z$ (see \cite{Gr1}) to define the sets $N(x)$ used in  \cite[Proposition 1.16]{Ma}, the proof generalises verbatim to our situation.
\endproof

{\bf Claim 2:} Let $u\in \W$ with  $u>wp'$. Letting $d_u$ denote the unique element in $\W_{K_2}u\W_{K_1}\cap D_{K_2,K_1}$, we have $d_u\geq d$.

\proof
By definition $u=a_2d_ua_1$ for some unique $a_i\in K_i$ and we can choose a reduced expression of $u$ compatible with this decomposition. Now $wp'<u$ means $wp'$ can be obtained by deleting some of the simple reflections. In particular, $d\leq d_u$.
\endproof

From Claim 1, we see that  $\bb^1_{K_2,wJ}\bb^w_{wJ,J}\bb^p_{J,J}\bb^1_{J,K_1} (\bv_{K_1}) = \bb^1_{K_2,wp'J}(T_wT_{p'}\bv_{K_1} )$, which equals $q^a \bb^1_{K_2,wp'J}\left((T_{wp'}\bv_{K_1} + (q-1) \sum_{z>xy}c_zT_z)\bv_{K_1}\right)$
for some $c_z\in\FF$ and $a=a(w,p')$. Now Claim 2 implies that, when rewriting this in the basis of the $\bb^{d'}_{K_2,K_1}$ only basis elements indexed with $d'\geq d$ occur. Moreover, the leading term $T_{wp'}$ contributes $q^{\frac{1}{2} (l(w)+l(p')-l(wp'))+l(v)} $ to the  coefficient of $\bb^{d}_{K_2,K_1}$ while any other $T_z$ that might contribute to the coefficient has coefficient of the form $c(q-1)^a$ for some integer $a$ and nonzero scalar $c$, so the coefficient of $\bb^{d}_{K_2,K_1}$ is nonzero and
  \eqref{basechange} follows.
  We conclude that the set given in \eqref{schurbasisB} is indeed a basis for $\cS$.
  \endproof
\begin{remark}\label{actiononGamma}
In the faithful representation $\rho$ from Theorem~\ref{faithfulrep}, a basis element $\bb^1_{K_2,wJ}\bb^w_{wJ,J}\bb^p_{J,J}\bb^1_{J,K_1}$ as in Proposition~\ref{schurbasis2} sends $f\bv^{(K_1)}$ to
$$ \left(\sum_{a \in D^{K_2}_{\emptyset,wJ}}T_a\right)T_wg_p\left(\sum_{b\in D^{K_1}_{ \emptyset, J}}T_b\right)f\bv^{(K_2)}.$$
where $g_p$ is defined in Lemma~\ref{pmult} below.
\end{remark}

\subsection{The subalgebra $\cQ$}
We now construct an important commutative subalgebra  $\cQ$ of $\cS$.  For this let $J\subset\bbI$ and recall the notation from Section~\ref{otherreps}.

\begin{lemma}\label{pmult}
For $p \in \DMJ$ and $J\subseteq\bbI$ we have
$\bb^p_{J,J}(\bv_J) = g_p\bv_J$,
for $g_p$ a scalar multiple of the sum $\sum_{x\in\W_J}x(p)$ over all monomials in the $\W_J$-orbit of $p$.
\end{lemma}

\proof
On the one hand, $\bb_{J,J}^{p}(\bv_J )= \sum_{u \in \W_Jp\W_J}T_u$.
Noting that $$\W_Jp\W_J = \{v f \mid v \in \W_J, f =\sigma p \sigma^{-1} \textrm{for some } \sigma \in \W_J\}$$
we see that $\bb_{J,J}^{p}\bv_K = \sum_{v \in \W_K, p' \in \W_J\cdot p}T_{v p'}$, which, using Lemma~\ref{tech2} can be expressed as $\sum_{v \in \fS}T_vf_v$ for some polynomials $f_v$, all of whose monomial terms are conjugate to $p$ under $\W_J$.

On the other hand, writing $p=td = dt'$ for $d\in D_{J,J}, t, t' \in \W_J$, we have
\begin{equation*}
\begin{array}[t]{lll}
\bb_{J,J}^{p}(\bv_J)  &= \bv_JT_d (\sum_{a\in D^J_{d^{-1}J\cap J}}T_a )&= \bv_J q^{-l(t)}T_tT_d \sum_{a\in D^J_{d^{-1}J\cap J}}T_a\\
& = \bv_J (q^{-l(t)}T_p\sum_{a\in D^J_{d^{-1}J\cap J}}T_a) &=  \bv_J q^{-l(t)}p\sum_{a\in D^J_{d^{-1}J\cap J}}T_a 
\\
 &=\bv_J A \in \bv_J\cH_J
\end{array}
\end{equation*}
and similarly
\begin{equation*}
\begin{array}[t]{lll}
\bb_{J,J}^{p}(\bv_J)  &= \sum_{a'\in D^J_{J\cap J}}T_{a'}T_d  \bv_J&= q^{-l(t')}\sum_{a'\in D^J_{dJ\cap J}}T_{a'} T_d T_{t'}\bv_J \\
& =  q^{-l(t')}\sum_{a'\in D^J_{d^{-1}J\cap J}}T_{a'}T_p \bv_J&=  q^{-l(t')}\sum_{a'\in D^J_{dJ\cap J}}T_{a'}p \bv_J  \\&=A'\bv_J\in \bv_J\cH_J.
\end{array}
\end{equation*}
Since $T_i\bv_J=\bv_JT_i = q\bv_J$ for all $s_i\in J$, moreover $A,A' \in \FF[X_1^{\pm 1}, \dots X_n^{\pm 1}]$. Observing $ \bv_J AT_i = \bb_{J,J}^{p}\bv_JT_i =  A'\bv_JT_i = q A'\bv_J$ for all $s_i \in J$, we furthermore deduce, using Lemma~\ref{qifsymm}, that $A,A' \in \FF[X_1^{\pm 1}, \dots X_n^{\pm 1}]^{\W_J}$ and hence $A=A'$. 

We obtain $\bb_{J,J}^{p}(\bv_J)  = \bv_J A \in \bv_J\FF[X_1^{\pm 1}, \dots X_n^{\pm 1}]^{\W_J}$ from the two preceding paragraphs. Here, all summands in $A$ are $\W_K$-conjugate to $p$. Therefore $\bb_{J,J}^{p}(\bv_J)$ is a scalar multiple of $g_p \bv_J$, as claimed. 
\endproof

\begin{proposition}\label{Qsub}
The $\FF$-vector space spanned by the $b^p_{J,J}$, for $J \subset \bbI$, $p \in \DMJ$, forms a commutative subalgebra $\cQ$ of $\cS$ which contains the centre $Z(\cS)$ of $\cS$. 
\end{proposition}
\proof
This follows directly from Lemmas~\ref{pmult} and~\ref{Scentre} via Remark~\ref{stupidremark}.
\endproof

\begin{remark}\label{Qdesc}
Note that by Lemma~\ref{pmult} and Lemma~\ref{qifsymm}, the subalgebra $\cQ$ consists of precisely those elements in $\Hom_{\cH}(\bv_J\cH,\bv_J\cH)$ (for some $J\subseteq \bbI$) which are given by left multiplication with some $f\in \cP\subset\cH$ which satisfies $(T_i-q)f\bv=0$ for all $i\in J$. 
\end{remark}
\begin{remark}
\label{Qasring}
The elements  $b^p_{J,J}$, for $J \subset \bbI$, $p \in \DMJ$, are linearly independent by Lemma~\ref{pmult}, hence form a basis of $\cQ$.  As an algebra, $\cQ$ is a direct sum of algebras indexed by $J\subset\bbI$ with factors isomorphic to  $\FF[X_1^{\pm 1},\dots, X_n^{\pm 1}]^{\W_J}$. The centre $Z(\cS)=\FF[X_1^{\pm 1},\dots, X_n^{\pm 1}]^\fS$, see Lemma~~\ref{Scentre}, embeds diagonally.
\end{remark}

\begin{lemma}\label{SfreeoverQ}
The $\FF$-vector space $\cS$ carries the structure of a finitely generated free $\cQ$-module on basis $$\cB^\cS_\cQ=\left\{ \bb^1_{K_2,wJ}\bb^w_{wJ,J}\bb^1_{J,K_1}  \left|
 \begin{array} {l}
 K_1,K_2\subseteq \bbI ,w \in D^\bbI_{K_2,K_1},\\ J = K_1\cap w^{-1}K_2
 \end{array}
 \right.\right\}.$$
\end{lemma}
(We do not claim that $\cS$ is a free $\cQ$-module by restriction of the regular action.)
\proof
We define the action of a basis element $\bb_{K,K}^{p'} \in \cQ$ on a basis element $\bb^1_{K_2,wJ}\bb^w_{wJ,J}\bb_{J,J}^{p}\bb^1_{J,K_1} \in \cB^\cS_\cQ$ (for some $K_1,K_2\subseteq \bbI$) by 
\begin{eqnarray*}
\bb_{K,K}^{p'} \circledast \bb^1_{K_2,wJ}\bb^w_{wJ,J}\bb_{J,J}^{p} \bb^1_{J,K_1} & =& \begin{cases}\bb^1_{K_2,wJ}\bb^w_{wJ,J}\bb_{J,J}^{pp'} \bb^1_{J,K_1} & \hbox{if } J=K,\\ 0&\hbox{otherwise.}\end{cases}
\end{eqnarray*}
By Remark~\ref{Qasring} this is a well-defined action of $\cQ$. Obviously, the module is generated by $\cB^\cS_\cQ$. Freeness follows from Remark~\ref{Qasring} and Proposition~\ref{Propschurbasis2}.
\endproof

\subsection{The twisted faithful representation of the Schur algebra}
The automorphism $\sharp$ from \eqref{eqhash} allows us to define the Schur algebra using $\bbv_K\cH$ instead of $\bv_K\cH$  via the obvious isomorphism of algebras
$$\cS\cong \Hom_{\cH}\bigoplus_{J,K\subseteq \bbI}((\bv_J\cH)^\sharp, (\bv_K\cH)^\sharp)  = \Hom_{\cH}\bigoplus_{J,K\subseteq \bbI}(\bbv_J\cH, \bbv_K\cH).$$

Similarly to Section~\ref{heckefaith}, we also  have a faithful representation $\orho$ of $\cS$ on 
\begin{eqnarray}
\label{twistedfaithful}
\oFS&=&\bigoplus_{K\subseteq \bbI}\FF[X_1^{\pm 1}, \dots,X_n^{\pm 1}]^{W_K}\bbv^{(K)},
\end{eqnarray}
where again the superscript on $\bbv^{(K)}$ is just a book-keeping device, given by 
$$\orho(\bb^d_{K,J})f\bbv^{(J)} = \sum_{a \in D^K_{\emptyset,K\cap dJ}}T^\sharp_aT^\sharp_df\bbv^{(K)}.$$
We also obtain the following analogue to Corollary~\ref{normalhash}.
\begin{corollary}
We have an isomorphism of representation 
\begin{eqnarray*}
\oFS \cong {}^\sharp\FS&\text{given by}&f\bbv^{(K)} \mapsto f^\sharp \bv^{(K)}.
\end{eqnarray*}
\end{corollary}

\section{A completion of $\cS$}

Recall the character $\chi=\chi_\bi$ for our fixed $\bi=(i_1,\dots, i_n)\in\mZ^n$ from Definition~\ref{idef}, and the ideals $\bm_\chi$  and $\cI_m$ in $Z(\cH)$ respectively $\cH$ from
Section~\ref{Hcompletionsec}. By Lemma~\ref{Scentre}, we can identify the centre of $\cS$ with the centre of $\cH$. Define $\ccS_\bi$ to be the completion of $\cS$ at the nested sequence of ideals  
\begin{eqnarray}
\label{Im}
\cJ_m&=&\bigoplus_{J, K \subseteq \bbI} \Hom_{\cH}(\cH^J, \bv_K \cI_m)
\end{eqnarray}
in $\cS$ generated by the maximal ideal $\bm_\chi$ of $Z(\cS)$.

\subsection{Compatibility with the completion of $\cH$}
The following gives an alternative definition of  $\ccS_\bi$, analogous to \eqref{affSchur}, using the completed Hecke algebras.

\begin{proposition}
\label{complS}
There is an isomorphism of algebras $$\ccS_\bi \cong \End_{\ccH_\bi}\left (\bigoplus_{J \subseteq \bbI} \bv_J \ccH_\bi\right).$$
\end{proposition}

\proof
We have 
\begin{equation*}
\begin{split}
\End_{\ccH_\bi}  \left(\bigoplus_{J \subseteq \bbI} \bv_J \ccH_\bi\right)&\cong
\bigoplus_{J, K \subseteq \bbI}\Hom_{\ccH_\bi}\left( \underleftarrow{\lim}\, \bv_J \cH/\cI_m, \underleftarrow{\lim}\,\bv_K \cH/\cI_k\right)\\
&\cong \bigoplus_{J, K \subseteq \bbI}\underleftarrow{\lim}\, \Hom_{\ccH_\bi}\left( \underleftarrow{\lim}\, \bv_J \cH/\cI_m, \bv_K \cH/\cI_k\right)\\
&\cong \bigoplus_{J, K \subseteq \bbI}\underleftarrow{\lim}\, \Hom_{\ccH_\bi}\left( \bv_J \cH/\cI_k, \bv_K \cH/\cI_k\right)\\
&\cong \bigoplus_{J, K \subseteq \bbI}\underleftarrow{\lim}\, \left(\Hom_{\ccH_\bi}( \bv_J \cH, \bv_K \cH)/  \Hom_{\ccH}( \bv_J \cH, \bv_K \cI_m)\right)\\
&\cong \underleftarrow{\lim}\, \cS/\cJ_m = \ccS_\bi.
\end{split}
\end{equation*}
The proposition follows.
\endproof
We like to generalise the idempotent decomposition of $\ccH_\bi$ from \eqref{compldecomp} to $\ccS_\bi$.  Our notation  follows the setup in \cite{SW} and \cite{KL}.

\subsection{Idempotent decomposition}
\label{secidem}
Recall our fixed $\bi=(i_1,\dots, i_n)\in\mZ^n$. Let  $J\subseteq \bbI$ and  $\bu =(u_1,\ldots, u_n)\in \fS\bi$.  
It will be convenient to encode the pair by splitting the tuple $\bu=(u_1,\ldots u_n)$ into blocks determined by $J$,  more precisely,  
\begin{eqnarray}
\label{uJ}
(\bu,J)&=&(u_1\cdots u_{t_1}|u_{t_1+1}\cdots u_{t_2}|\cdots |u_{t_{r-1} +1} \cdots u_{t_r})
\end{eqnarray}
where $t_r=n$, and a line is drawn between $u_k$ and $u_{k+1}$ if and only if $k \notin J$.  In the extreme cases we have $(\bu,{\emptyset}) = (u_1|u_2|\cdots |u_n)$ and $(\bu,{\bbI}) = (u_1,u_2,\cdots, u_{n})$.

For $(\bu,J)$ denote $\bu_J=(\bu',J)$, where $\bu'$ is the unique element in the $\W_J$-orbit of $\bu$  with ascending integers in the parts between the lines in \eqref{uJ},  i.e. 
\begin{eqnarray}
\label{idempotentsSchur}
&\quad u'_1\leq\cdots\leq u'_{t_1},\quad u'_{t_1+1}\leq\cdots\leq u'_{t_2},\quad  \ldots \quad, \quad u'_{t_{r-1}+1}\leq\cdots \leq u'_n.\quad\quad&
\end{eqnarray}
Here, if $q$ is an $e-th$ roots of unity, we  order our chosen  representatives $1,\dots, e$ for $\Z/e\Z$ as $1<\dots <e$. For $ J\subseteq\bbI$, we denote by $$\bU_J=\{\bu_J\mid \bu\in\fS\bi\}$$ the set of such representatives of $\W_J$-orbits in $\fS\bi$. Given $\bu\in\fS\bi$ and $J\subset K\subseteq\bbI$ we call $\bu_J$ a {\it refinement} of $\bu_K$.

\begin{example}\label{mergex}
Let $n=7$ and $J = \{1,3,5\}\subset K=\{1,2,3,5,6\}$. With  $\bu = (1,2,1,1,2,1,1)$ we have $\bu_J = (1,2|1,1|1,2|1)$ and  $\bu_K = (1,1,1,2|1,1,2)$.  Then $\bu_J$ a refinement of $\bu_K$. Note that indeed the additional vertical lines in $\bu_J$ provide a refinement of the parts of $\bu_K$.
\end{example}

Attached to $\bu_J =  (\bu,J)\in \bU_J$ we have the idempotent
\begin{eqnarray}
\label{eplus}
\e_{\bu_J}&=&\sum_{\bu' \in \W_J\cdot \bu} e_{\bu'}\quad\in \ccH_\bi.
\end{eqnarray}

\begin{lemma}\label{idempcomm}
For $J\subset \bbI$ and  $\bu_J \in \bU_J$, the elements $\e_{\bu_J}, \bv_J$ (in $\ccH_\bi$) commute.
\end{lemma} 

\proof
We may assume $J\not=\emptyset$, since $\bv_\emptyset=1$, in which case the claim is obvious.  Suppose first that $\W_J\cong\fS_k$ for some $k$ and  consider the subalgebra $ \ccH_{J,\bi}$. Its identity element is $\e_{\bu_J}$ and we obtain $\ccH_{J,\bi}=\e_{\bu_J} \ccH_{J,\bi}= \ccH_{J,\bi}\e_{\bu_J}$. In particular $\e_{\bu_J}$ commutes with $\bv_J$. Note that in the extremal case $J=\bbI$, the element $\e_{\bu_\bbI}$ is just the identity element in $ \ccH_{J,\bi}$.
Otherwise, we can find a proper decomposition $J=J_1\cup J_2\subset  \bbI$ such that $W_J=\W_{J_1}\times\W_{J_2}$. Then $\e_{\bu_J}=\e_{\bu_{J_1}}\e_{\bu_{J_2}}$ and $\bv_J = \bv_{J_1}\bv_{J_2}$. By definition,  $\bv_{J_i}$ commutes with $\e_{\bu_{J_j}}$ if $i\not=j$ and also, by the extremal case treated above, with $\e_{\bu_{J_i}}$. Then the  claim follows. 
\endproof

Lemma~\ref{idempcomm} directly implies the following result.
\begin{corollary} There is an isomorphism of algebras
\begin{equation*}\begin{split}
\bigoplus_{J, K \subseteq \{s_1, \dots, s_{n-1}\}}&\Hom_{\ccH_\bi}(\bv_J \ccH_\bi, \bv_K \ccH_\bi) \\ &\cong \bigoplus_{J, K \subseteq \{s_1, \dots, s_{n-1}\}} \bigoplus_{\bu_J\in \bU_J, \bu'_K \in \bU_K}  \Hom_{\ccH_\bi}(\e_{\bu_J}\bv_J \ccH_\bi, \e_{\bu'_K}\bv_K \ccH_\bi).
\end{split}\end{equation*}
\end{corollary}

\subsection{Splits and merges in in the algebraic basis}
\label{splitsandmerges}
We now define certain split and merge maps motivated by the construction in \cite{SW}.

Let $J\subset K\subseteq\bbI$, $\bu_J=(\bu,J)\in \bU_J$, which uniquely defines $\bu_K\in\bU_K$, of which $\bu_J$ is a refinement. We use the notation from Section~\ref{secidem}.
\begin{definition}
\label{smsigma}
\begin{enumerate}[i.)]
\item  We have $\e_{\bu_J} \bb^1_{J,K}\e_{\bu_K}\not=0$, and this is called a \emph{split} of $\bu_K$. If $|K\setminus J|=1$, we call $J\subset K$ a \emph{simple inclusion}, and  $\e_{\bu_J} \bb^1_{J,K}\e_{\bu_K}$ is called a \emph{simple split}.
\item We have $\e_{\bu_K} \bb^1_{K,J}\e_{\bu_J}\neq 0$, and this is called a \emph{merge} of $\bu_J$. Again, if $|K\setminus J|=1$, it is called a \emph{simple merge} of $\bu_J$. 
\item We denote by $\sigma_{\bu_J}^{\bu_K}\in D_J^K$ the unique element with $\sigma_{\bu_J}^{\bu_K}\bu_J= \bu_K$.
\end{enumerate}
\end{definition}
Note that any split (resp. merge) can be written as a sequence of simple splits (respectively simple merges).

\begin{example}
In the setup of Example~\ref{mergex} we have a (non-simple) split of $\bu_K$. In this case  $\sigma_{\bu_J}^{\bu_K} = s_3s_2s_6$.
\end{example}

\subsection{Dimension matrix and dimension vectors}
To $(\bu, J)$ as in \eqref{uJ} 
we  now associate several combinatorial objects and groups of permutations. 
\begin{definition}
\label{dimmatrix}
The {\it dimension matrix} attached to $(\bu,J)$  is the $e\times r$-matrix $\dimmatrix(\bu_J)$ with entries in $\mathbb{Z}_{\geq 0}$ defined as
\begin{eqnarray}
&\dimmatrix(\bu,J)=\left(d_i^j(\bu,J)\right),  \quad 1\leq i\leq e, 1\leq j\leq r,&\\
\text{where}& d_i^j=d_i^j(\bu,J)=|\{k\mid t_{j-1}+1\leq k\leq t_{j}, u_k=i\}|\quad\text{with $t_0=0$}. &\nonumber
\end{eqnarray}
Let $d_i=d_i(\bu,J)=\sum_{j=1}^r d_i^j(\bu,J)$ and $d^j=d^j(\bu,J)=\sum_{i=1}^e d_i^j(\bu_J)$.
Note that $d_i^j$ counts precisely the number of occurrences of $i$ in the $j$th block of $(\bu,J)$, whereas $d_i$ is the total number of $i$, and $d^j$ gives the size of the $j$th block. For fixed $1\leq i\leq e$, set 
$\dimmatrix_i=\dimmatrix_i(\bu,J)=(d_i^1(\bu,J),\ldots, d_i^r(\bu,J))$, and for fixed $1\leq j\leq e$, set $\dimmatrix^j=\dimmatrix^j(\bu,J)=(d_1^j(\bu,J),\ldots, d_e^j(\bu,J))$. The first encodes the multiplicities how often a certain number appears in each part, the second encodes for a fixed part the multiplicities of the numbers occurring in it. 

 We call $\dimv=\dimv(\bu,J)=(d_1,d_2,\ldots,d_e)$ the {\it dimension vector} 
and $\typ=\typ(\bu,J)=(d^1,d^2,\ldots,d^r)$ the {\it type vector} 
attached to $(\bu,J)$. Hence the dimension vector encodes the multiplicities how often each number occurs in total, whereas the type encodes the sizes of the parts ignoring which numbers occur. Note that the dimension vector only depends in $\bi$ and we thus also write $\dimv=\dimv(\bi)$.
\end{definition}

\begin{example}
\label{mergex2}
In the setup of Example~\ref{mergex} the dimension matrix for $\bu_J$ is given by $d_1^1=d_1^3=d_1^4=1$ and $d_1^2=2$, whereas $d_2^1=d_2^3=1$ and $d_2^2=d_2^4=0$. The dimension vector is $(5,2)$ and its type vector is $(2,2,2,1)$. On the other hand, for $\bu_K$,  we have for  the values $d_1^1=3$ and $d_1^2=2$, and $d_2^1=d_2^3=1$. The dimension vector is again $(5,2)$, but the type vector is $(4,3)$.
\end{example}

Given $(\bu,J)$ we have now several (sub)groups of permutations attached to it (where we omit the $(\bu,J)$ in the notation on the right hand side):
\begin{eqnarray}
\fS_{\dimmatrix_i(\bu,J)}&=&\fS_{d_i^1}\times\fS_{d_i^2}\cdots\times \fS_{d_i^r}<\fS_{d_i},\label{Young1}\\
\fS_{\dimmatrix_\dimv(\bu,J)}&=&\fS_{\dimmatrix_1}\times\cdots\times \fS_{\dimmatrix_e}<\fS,\label{Young2}\\
\fS_{\dimmatrix^j(\bu,J)}&=&\fS_{d_1^j}\times\fS_{d_2^j}\cdots\times \fS_{d_e^j}<\fS_{d^j},\label{Young3}\\
\fS_{\dimmatrix_\typ(\bu,J)}&=&\fS_{\dimmatrix^1}\times\cdots\times \fS_{\dimmatrix^r}<\fS\label{Young4}.
\end{eqnarray}
Note that choosing $\bu_J=(\bu,J)\in \bU_J$ has the nice effect that $\fS_{\dimmatrix_\typ(\bu_J)}=\W_J\cap \Stab_{\fS}\bu$ is a standard parabolic subgroup. Note that $\fS_{\dimmatrix_\dimv(\bu_J)}\cong\fS_{\dimmatrix_\typ(\bu_J)}$, since both groups precisely describe all permutations of $\bu_J$ such that the number as well as the parts given by $J$ are preserved.

\begin{example}
In the setup of Example~\ref{mergex} we have  $\fS_{\dimmatrix_\dimv(\bu_J)}=\fS_1\times \fS_2\times\fS_1\times\fS_1\times\fS_1\times\fS_0\times\fS_1\times\fS_0$ and $\fS_{\dimmatrix_\typ(\bu_J)}=\fS_1\times \fS_1\times\fS_2\times\fS_0\times\fS_1\times\fS_1\times\fS_1\times\fS_0$. 
\end{example}

In the following we will often drop the dependence on $(\bu,J)$ in the notation, if we have some fixed $(\bu,J)$, and we will only ever consider the case where $(\bu,J) = \bu_J$ for some $\bu_J\in \bU_J$. 
In this case, the groups \eqref{Young2} and \eqref{Young4} and then also \eqref{Young1} and \eqref{Young3} are generated by certain standard generators $s_i\in\fS$ labeled by a subset of $\bbI$.  It will be convenient to use also  different labellings of the generators which reflect directly the respective product decompositions. 
\begin{definition}
For the group \eqref{Young2} the $a$th generator in the $i$th factor is denoted $s_{i,a}$, whereas for \eqref{Young4} the $a$th generator of the $j$th factor is called $s_a^{(j)}$.
\end{definition}
In this notation we can make the above isomorphism explicit:

\begin{lemma} \label{zetaiso1}
There is an isomorphism of groups 
\begin{eqnarray}
\label{zeta}
\zeta_{\bu_J}: \quad\fS_{\dimmatrix_\dimv(\bu_J)}\cong\fS_{\dimmatrix_\typ(\bu_J)}&&s_{i,a}\mapsto s_l^{(t)}, 
\end{eqnarray}
where $t$ is such that $\sum_{k=1}^{t-1}d^k_i<a\leq \sum_{k=1}^{t}d^k_i$ and $l=a+\sum_{k=1}^{i-1} d_k^t$. 
\end{lemma}

\begin{proof}
Since the two groups define the same subgroup of $\fS$, it suffices to compare their images there. But $s_{i,a}$ corresponds to  $s_b\in\fS$, where $b=a+(\sum_{k=1}^{t-1}d^k)+(\sum_{k=1}^{i-1}d_k^t)$ whereas $s_l^{(t)}$ corresponds to $s_b\in\fS$,  where $b=c+(\sum_{k=1}^{t-1}d^k)$. Hence the claim follows.
\end{proof}

\begin{definition}
\label{WuJ}
We will abbreviate the group in \eqref{zeta} by $\fS_{\bu_J}$, but keep the two realisations in mind.
Note that it is a standard parabolic subgroup of $\fS$ and we define $I_{\bu_J}$ by $\W_{I_{\bu_J}}=\fS_{\bu_J}$.
For $1\leq i\leq e$ and $1\leq a\leq d_i$ we abbreviate
\begin{eqnarray}
\label{iapos}
(i,a)_{\bu_J}&=&a+\left(\sum_{k=1}^{t-1}d^k\right)+\left(\sum_{k=1}^{i-1}d_k^t\right)
\end{eqnarray}
with $t$ such that $\sum_{k=1}^{t-1}d^k_i<a\leq \sum_{k=1}^{t}d^k_i$. 
\end{definition}
Note that $(i,a)_{u_J}$ is just the position where the $a$th number $i$ occurs in $\bu$. 

\begin{example}
Let us consider  $\bu_J = (1,1,2|1,1,1,2,2|1,1,2)$. Hence $n=11, r=3$ and $J = \{s_1,s_2,s_4,s_5,s_6,s_7,s_9,s_{10}\}$ with $\W_J\cong\fS_3\times\fS_5\times\fS_3$. In the usual generators of $\fS$ we have $\fS_{\dimmatrix_1}\cong\langle s_1,s_4,s_5,s_9\rangle\cong\fS_2\times\fS_3\times\fS_2$ and $\fS_{\dimmatrix_2}=\langle s_7\rangle\cong\fS_1\times\fS_2\times\fS_1$ and then $\fS_{\dimmatrix_\dimv}=\langle s_1, s_4, s_5, s_7, s_9\rangle\cong\fS_2\times\fS_3\times\fS_2\times\fS_1\times\fS_2\times\fS_1$. It agrees with $\fS_{\dimmatrix_\typ}$ as a subgroup of $\fS$ (or $\W_J$).  

Now $\fS_{\dimmatrix_\dimv}=\langle s_{1,1},s_{1,3},s_{1,4},s_{1,6},s_{2,2}\rangle$ and $\fS_{\dimmatrix_\typ(\bu_J)}=\langle s_1^{(1)},s_1^{(2)},s_2^{(2)},s_3^{(2)},s_1^{(3)}\rangle$. The isomorphism $\zeta_{u_J}$ sends $s_{1,1},s_{1,3},s_{1,4},s_{1,6},s_{2,2}$ to  $s_1^{(1)},s_1^{(2)},s_2^{(2)},s_3^{(2),s_1^{(3)}}$ respectively. The corresponding elements $s_{(i,a)_{u_J}}\in\fS$ are $s_1,s_4,s_5,s_9,s_7$. 
\end{example}

\subsection{Rings of invariants}\label{inv}

We now consider invariant polynomials for \eqref{Young2} and \eqref{Young4}. Our different choices of generators come along with different labellings of the variables. We attach, to our fixed $\bi$, the following polynomial rings
\begin{equation*}
\begin{split} 
R_+&=\FF[Y_{1,1},Y_{1,2},\ldots, Y_{1,d_1},Y_{2,1}, \ldots,Y_{2,d_2},\ldots Y_{e,d_e}]\\
&=\FF[Y_{1,1},\ldots, Y_{1,d_1}]\otimes_\FF \FF[Y_{2,1}, \ldots,Y_{2,d_2}]\otimes_\FF\cdots\otimes_\FF
\FF[Y_{e,1},\ldots, Y_{e,d_e}],\\
\text{and}&\text{ moreover}\\
R_-&=\FF[Y_{1,1}^{-1},Y_{1,2}^{-1},\ldots, Y_{1,d_1}^{-1},Y_{2,1}^{-1}, \ldots, Y_{2,d_2}^{-1},\ldots, Y_{e,d_e}^{-1}]\\
&=\FF[Y_{1,1}^{-1},\ldots, Y_{1,d_1}^{-1}]\otimes_\FF \FF[Y_{2,1}^{-1}, \ldots,Y_{2,d_2}^{-1}]\otimes_\FF\cdots\otimes_\FF
\FF[Y_{e,1}^{-1},\ldots, Y_{e, d_e}^{-1}],
\end{split}
\end{equation*}
in $n$ variables. For each $\bu_J\in \bU_J$  for some $J\subset\bbI$, we fix the ring isomorphisms
\begin{eqnarray*}
{\zeta}_{+,\bu_J}: \: R_+\:\cong\:\FF[X_1,\ldots, X_n]&\text{and}& {\zeta}_{-,\bu_J} :\: R_-\:\cong\:\FF[X_1^{-1},\ldots, X_n^{-1}].
\end{eqnarray*}
sending $Y_{i,a}^\pm$ to $X_{(i,a)_{\bu_J}}^\pm$ with the notation from Definition~\ref{WuJ}. Together with Lemma~\ref{zetaiso1} this also gives canonical identifications of invariants
\begin{eqnarray*}
R_+^{\fS_{\dimmatrix_\dimv(\bu_J)}}\:=\: \FF[X_1,\ldots, X_n]^{\fS_{\dimmatrix_\typ(\bu_J)}},&&R_-^{\fS_{\dimmatrix_\dimv(\bu_J)}}\:=\:\FF[X_1^{-1},\ldots, X_n^{-1}]^{\fS_{\dimmatrix_\typ(\bu_J)}}
\end{eqnarray*}
Again, we will often abbreviate these invariants as $R_+^{\fS_{\bu_J}}= \FF[X_1,\ldots, X_n]^{\fS_{\bu_J}}$, 
respectively $R_-^{\fS_{\bu_J}}=\FF[X_1^{-1},\ldots, X_n^{-1}]^{\fS_{\bu_J}}$.

Let $\hat{R}_+$ and $\hat{R}_-$ be the completions of $R_+$ and $R_-$ at the maximal ideals generated by all the $Y_{i,j}$ respectively $Y_{i,j}^{-1}$. We have isomorphisms
\begin{eqnarray*}
\hat{\zeta}_{+,\bu_J}: \: R_+\:\cong\:\FF[[X_1,\ldots, X_n]]&\text{and}& \hat{\zeta}_{-,\bu_J} :\: R_-\:\cong\:\FF[[X_1^{-1},\ldots, X_n^{-1}]].
\end{eqnarray*}
induced  by $\zeta_{\pm,\bu_J}$.

For $\bu_J=(\bu,J)\in\bU_J$, again with the notation from Definition~\ref{WuJ}, define the $\FF$-linear inclusions
\begin{equation*}
\begin{split}
\xi_\bu: \: \hat R_+&\hookrightarrow \bigoplus_{w \in D^J_{\emptyset, I_{\bu_J}}} \FF[[X_1,\ldots, X_n]]\:e_{w\cdot \bu}, \;\;Y_{i,a}\mapsto \sum_{w \in D^J_{\emptyset, I_{\bu_J}}}\left( X_{(i,a)_{w\cdot\bu}}\;e_{w \cdot\bu}\right),\\
\xi_\bu^\sharp:\: \hat R_-&\hookrightarrow\bigoplus_{w \in D^J_{\emptyset, I_{\bu_J}}
} \FF[[X_1^{-1},\ldots, X_n^{-1}]]\:e_{w\cdot\bu},\;\;
Y_{i,a}^{-1}\mapsto\sum_{w \in D^J_{\emptyset, I_{{\bu}_J}}} (X^{-1}_{(i,a)_{w\cdot\bu}\;}e_{w \cdot\bu}).
\end{split}
\end{equation*}
Denote by $\hat R_{\pm,\bu_J}$  the images of  $\hat R_\pm^{\fS_{\bu_J}}$ under  $\xi_\bu$ respectively $\xi_\bu^\sharp$. In the following, we will identify elements with their images, i.e. we will view elements of $\hat R_{+,\bu_J}$ alternatively as formal power series in the variables $Y_{i,a}$ or as formal power series of the form $\sum_{w \in D^J_{\emptyset, I_{\bu_J}}} f_we_{w \cdot\bu}$ where each $f_w$ is a formal power series in $X_1, \ldots, X_n$, and similarly for elements of $\hat R_{-,\bu_J}$.

\subsection{The completion of $\cQ$}
\label{Qsection}
The goal of this subsection is to describe the completion of the subalgebra $\cQ$, which will play a similar role to the completion of the subalgebra $\cP$ in $\cH$ in giving rise to a completed faithful representation.

The {\it completion} $\ccQ_\bi$ of $\cQ$ is spanned by the elements in $\Hom_{\ccH_\bi}(\bv_J \ccH_\bi, \bv_J \ccH_\bi)$, for $J\subseteq \bbI$, 
which are equal to left multiplication with an element $f\in \ccP$ such that $(T_i-q)f\e_{\bu_J}\bv_J =0$, see Remark~\ref{Qdesc}.
Equivalently, $\ccQ_\bi$ is, via the decomposition $\ccP_\bi\cong \bigoplus_{\bu\in\fS\bi}e_\bu\ccP_\bi$ and Lemma~~\ref{idempcomm},  spanned by those elements in $\Hom_{\ccH_\bi}(\e_{\bu_J}\bv_J  \ccH_\bi, \e_{\bu_J}\bv_J \ccH_\bi)$ for some $J\subseteq \bbI$ and $\bu_J=(\bu,J)\in\bU_J$,  
which are equal to left multiplication with some $f\in \bigoplus_{\bu' \in \W_J\cdot \bu }e_{\bu'}\ccP$ such that $(T_i-q)f\e_{\bu_J}\bv_J =0$.
This last property can be rephrased, similarly to Lemma~\ref{qifsymm},  as follows.

\begin{lemma}
\label{Lakritz}
Let $K\subseteq \bbI$ and $\bu_K=(\bu,K) \in U_K$. Furthermore,  assume that $f \in \bigoplus_{w \in D^K_{\emptyset, I_{\bu_K}}
} \FF[[X_1^{-1},\ldots, X_n^{-1}]]\:e_{w\cdot\bu}$. Then 
$f\in  \hat R_{-,\bu_K}$ if and only if 
$(T_i-q) f\e_{\bu_K}\bv^{(K)}= 0$ for all $i\in K$. 
\end{lemma}
\proof
Assume that $f\in  \hat R_{-,\bu_K}$ and pick  $i\in K$. We have to show that  $(T_i-q) f\sum_{\bu' \in \W_K\cdot \bu} e_{\bu'}\bv^{(K)}= 0$. It then suffices to verify, for any $w \in D^K_{\emptyset, K_\bu}$ $$(T_i-q) f\left(e_{w\cdot\bu}+e_{s_iw\cdot\bu}\right)\bv^{(K)}=0,$$ if $s_iw\cdot\bu\not=w\cdot\bu$ and $(T_i-q) fe_{w\cdot\bu}\bv^{(K)}=0$
otherwise. Assume first $s_iw\cdot\bu\not=w\cdot\bu$. The, setting $\beta_i=\frac{(q-1)}{1-X_iX_{i+1}^{-1}}$, we have, using $\xi_\bu^\sharp$, the equalities
\begin{eqnarray*}
&&(T_i-q)Y_{c,a}^{-1} \left(e_{w\cdot\bu}+e_{s_iw\cdot\bu}\right)\bv^{(K)}\\
&{=}&(T_i-q)\left(X^{-1}_{(c,a)_{w\cdot \bu}}e_{(w\cdot\bu)}+X^{-1}_{(c,a)_{s_iw\cdot\bu}}e_{s_iw\cdot\bu}\right)\bv^{(K)}\\
&{=}&(T_i-q)\left(X^{-1}_{(c,a)_{w\cdot \bu}}e_{(w\cdot\bu)}+X^{-1}_{s_i((c,a)_{w\cdot \bu})}e_{s_iw\cdot\bu}\right)\bv^{(K)}\\
&\stackrel{\eqref{Phi}}{=}&(\Phi_i-q+\beta_i)\left(X^{-1}_{(c,a)_{w\cdot \bu}}e_{w\cdot\bu}+X^{-1}_{s_i((c,a)_{w\cdot \bu})}e_{s_iw\cdot\bu}\right)\bv^{(K)}\\
&=&X^{-1}_{s_i((c,a)_{w\cdot \bu})}e_{s_iw\cdot\bu}\Phi_ie_{w\cdot\bu}+X^{-1}_{(c,a)_{w\cdot \bu}}e_{w\cdot\bu}\Phi_ie_{s_iw\cdot\bu}\\
&&\quad\quad\quad-(q-\beta_i)
(X^{-1}_{(c,a)_{w\cdot \bu}}e_{w\cdot\bu}+X^{-1}_{s_i((c,a)_{w\cdot \bu)}}e_{s_iw\cdot\bu}\bv^{(K)}\\
&=&(X^{-1}_{s_i((c,a)_{w\cdot \bu})}e_{s_iw\cdot\bu}+X^{-1}_{(c,a)_{w\cdot \bu}}e_{w\cdot\bu})(\Phi_i-q+\beta_i)\left(e_{w\cdot\bu}+e_{s_iw\cdot\bu}\right)\bv^{(K)}\\
&\stackrel{\eqref{Phi}}{=}&
(X^{-1}_{(c,a)_{s_iw\cdot\bu}}e_{s_iw\cdot\bu}+X^{-1}_{(c,a)_{w\cdot \bu}}e_{w\cdot\bu})(T_i-q)\bv^{(K)}
\stackrel{\text{(H-1)}}{=}0.
\end{eqnarray*}
Let now $s_iw\cdot\bu=w\cdot\bu$. Then $(w^{-1}s_iw)\cdot \bu=\bu$, hence $t=w^{-1}s_iw\in {\W_{\bu_K}}$. If we view $t\in \fS_{\dimmatrix_\dimv(\bu_J)}$ via Definition~\ref{WuJ}, then by assumption $f$ is $t$-invariant. On the other hand, by our definitions, $f$ being invariant under $t$ when written in the $Y_{c,a}^{-1}$ is equivalent to $f$ being invariant under $s_i$ when written in the $X_j^{-1}$. But now Lemma~\ref{qifsymm} implies $(T_i-q)  fe_{w\cdot\bu}\bv^{(K)} =  e_{w\cdot\bu}(T_i-q)f\bv^{(K)}  =0.$ Hence the ``if'' part of the statement follows.

Now assume that $(T_i-q) f\e_{\bu_K}\bv^{(K)}=0$ for all  $i \in K$. If $\fS_{\bu_K}$
is trivial, then there is nothing to show. Otherwise let $s_b$ be a standard generator in $\fS_{\bu_K}$, in particular $s_b\bu=\bu$. Let $b=(c,a)_\bu$. Then  $s_b\bu=\bu$ implies that $(c,a+1)_\bu=b+1=(c,a)_\bu+1$. 

We can write $f=Y_{c,a+1}^{-1}g+h$ for some unique $g,h\in \hat R_-^{s_{c,a}}=\hat R_-^{s_{b}}$  and, noting that $\xi_\bu^\sharp(g)\in \FF[[X_1^{-1},\ldots,X_n^{-1}]]^{s_b}e_\bu\bv$ for any  $g\in \hat R_-^{s_{c,a}}$, we have 
\begin{eqnarray*}
0&=&e_\bu(T_{b}-q) fe_\bu\bv^{(K)}\\
&=&e_\bu(T_{b}-q) Y_{c,a+1}^{-1}g e_\bu\bv^{(K)}+e_\bu (T_b-q)he_\bu\bv^{(K)}\\
&=&e_\bu(T_{b}-q) X_{b+1}^{-1}ge_\bu\bv^{(K)}+0\quad {(\text{since }h\in \hat R_-^{s_{b}}})\\
&\stackrel{\text{(H-7)}}{=}&e_\bu (qX^{-1}_{b}T_{b}^{-1}- qX^{-1}_{b+1}) ge_\bu\bv^{(K)}\\
&\stackrel{\eqref{easyformel}}{=}&e_\bu (X^{-1}_{b}-q X^{-1}_{b+1})g e_\bu\bv^{(K)}.
\end{eqnarray*}
This is nonzero if $g\not=0$. Hence $(T_j-q) f\e_{\bu_K}\bv^{(K)}=0$ implies $g=0$ and thus $f$ has to be $s_{b}$-invariant, and then even $f\in \hat R_{-,\bu_K}$.\endproof

For $J\subseteq \bbI$, denote by $\DMJp\subset\DMJ$ (respectively $\AMJp\subset\AMJ$) the subset of monomials $X_1^{a_1}\cdots X_n^{a_n}$ with $a_i \geq 0$ for all $i$, and by $\DMJm\subset\DMJ$ (respectively $\AMJm\subset\AMJ$) the subset of monomials $X_1^{a_1}\cdots X_n^{a_n}$ with $a_i \leq 0$ for all $i$.  Note that $\flat$ from Section~\ref{extaff} induces a bijection between $\DMJp$ and  $\AMJm$ and between $\AMJp$ and $\DMJm$.

As a consequence we obtain a topological basis of  $\ccQ_\bi$.

\begin{proposition}\label{Qcomplete}
The completion $\ccQ_\bi$ of $\cQ$ has topological $\FF$-basis given by 
\begin{eqnarray}
\ccB_\cQ&=&\{\hat\bb^{p}_{J,J}\e_{\bu_J} \mid J\subseteq \bbI,\bu_J\in \bU_J, p \in \AMJum\}
\end{eqnarray}
where $\hat\bb^{p}_{J,J}\e_{\bu_J}\in \Hom_{\ccH_\bi}(\e_{\bu_J}\bv_J  \ccH_\bi, \e_{\bu_J}\bv_J \ccH_\bi)$ is the homomorphism given by  $$\hat\bb^{p}_{J,J}\e_{\bu_J}(\e_{\bu_J}\bv_J ) =\sum_{w\in D^J_{\emptyset,I_{\bu_J}}}w(p)e_{w\cdot \bu_J} \e_{\bu_J}\bv_J . $$
\end{proposition}
\proof
The fact that this set spans $\ccQ_\bi$ follows directly from the first paragraph of this subsection reformulated using the equivalence from Lemma~\ref{Lakritz}.  It is linearly independent as a direct consequence of \eqref{schurbasisB}.
\endproof

\subsection{A basis of $\ccS_\bi$}
The main goal of this subsection is the following basis theorem (with the notation from Proposition~\ref{Qcomplete}).

\begin{proposition}\label{cschurbasis}
For fixed $K_1,K_2\subseteq\bbI$, the set
\begin{eqnarray*}
\ccB_{K_2,K_1}&=&\left\{\bb^1_{K_2,dJ}\bb^d_{dJ,J}\hat\bb^p_{J,J}\e_{\bu_{J}}\bb^1_{J,K_1} \left |
 \begin{array}[c]{l}
d \in D^\bbI_{K_2,K_1}, \\J = K_1\cap d^{-1}K_2,\\
 \bu_{J}\in \bU_J, p \in \DM_{I_{\bu_J}}^-
\end{array}
\right.\right\}
\end{eqnarray*}
is a topological $\FF$-basis for $\Hom_{\ccH_\bi}(\bv_{K_1}\ccH_\bi,\bv_{K_2}\ccH_\bi)$.
\end{proposition}

\proof
By Proposition~\ref{Qsub} we have (with  the notation from \eqref{Im}) 
$$\cS(\cQ\cap\cJ_m)\cS =\cS(\cQ\cap\cS\bm_\chi^m\cS)\cS =\cS(\cQ\bm_\chi^m\cQ)\cS = \cS\bm_\chi^m\cS = \cJ_m $$ where the third equality follows from $\bm^m_\chi\subset Z(\cH)\subset \cQ.$

Now, as a $\cQ$-module, we have $\cS\cong \bigoplus_{\bx\in\cB^\cS_\cQ}\cQ \circledast \bx$ by Lemma~\ref{SfreeoverQ} (with the notation defined there). Since $\bm_\chi$ is central, the actions by left multiplication, right multiplication, or the $\circledast$-action induced by $\bm_\chi\subset \cQ$ all coincide, so 
$$\cS/\cJ_m = \cS/\cS\bm_\chi^m \cong 
(\bigoplus_{\bx\in\cB^\cS_\cQ} \cQ\circledast\bx )/
(\bigoplus_{\bx\in\cB^\cS_\cQ}\cQ\bm_\chi^m \circledast \bx)\cong
\bigoplus_{\bx\in\cB^\cS_\cQ}(\cQ/ \cQ\bm_\chi^m) \circledast\bx.$$
Thus we obtain that $\ccS_\bi = \underleftarrow{\lim} \cS/\cJ_m $ equals
\begin{eqnarray*}
 \underleftarrow{\lim} \bigoplus_{\bx\in\cB^\cS_\cQ} (\cQ/ \cQ\bm_\chi^m)  \circledast\bx= \bigoplus_{\bx\in\cB^\cS_\cQ} \underleftarrow{\lim}  (\cQ/ \cQ\bm_\chi^m)\circledast\bx
=\bigoplus_{\bx\in\cB^\cS_\cQ}  \ccQ_\bi\circledast\bx.
\end{eqnarray*}

In particular $\ccS_\bi$ is now free over $ \ccQ_\bi$ on basis $\cB^\cS_\cQ$, which, together with Proposition~\ref{Qcomplete} and the definition of $\circledast$ implies the desired basis for $\ccS_\bi$.
\endproof

We have the following direct consequence.
\begin{corollary}\label{cschurbasisidem}
Let $K_1,K_2\subseteq\bbI$ and moreover let $\bu'_{K_1} = (\bu',K_1)\in U_{K_1}$ and $\bu''_{K_2}= (\bu'',K_2)\in U_{K_2}$. Then a basis of $\Hom_{\ccH_\bi}(\e_{\bu'_{K_1}}\bv_{K_1}\ccH_\bi,\e_{\bu''_{K_2}}\bv_{K_2}\ccH_\bi)$ is
\begin{eqnarray}
\label{cschurbasisidemf}  
&\quad\left\{\e_{\bu''_{K_2}}\bb^1_{K_2,dJ}\bb^d_{dJ,J}\hat\bb^p_{J,J}\e_{\bu_{J}}\bb^1_{J,K_1}\e_{\bu'_{K_1}} \:\left | \:
 \begin{array}[c]{l}
d \in D^\bbI_{K_2,K_1}, \\J = K_1\cap d^{-1}K_2,\\ p \in \DMJ^-,\\
 \bu_{J}=(\bu,J)\in \bU_J \textrm{ with } \\   \bu'_{K_1}=\bu_{K_1},\\ \bu''_{K_2}=(d\cdot\bu)_{K_2} 
\end{array}
\right.\right\} \qquad&
\end{eqnarray}
where we note that $\bu'_{K_1}=\bu_{K_1}$ means $\bu'\in \W_{K_1}\cdot\bu$ and similarly $\bu''_{K_2}=(d\cdot\bu)_{K_2}$ means $\bu'' \in \W_{K_2} \cdot(d\cdot\bu)$.
\end{corollary}

\subsection{A faithful representation of $\ccS_\bi$}
In this subsection we will construct a faithful representation of  $\ccS_\bi$ similar to the construction in Theorem~\ref{faithfulrep} .

For $K\subseteq\bbI$ and $\bu_K\in\bU_K$,
define
$$\cFS^{\bu_K}_\bi=\hat{R}_{-,\bu_K}\bv^{(K)}\quad\hbox{and}\quad\coFS^{\bu_K}_\bi=\hat R_{+^,\bu_K}\bbv^{(K)},$$ where,  as in Theorem~\ref{faithfulrep}, the superscript in $\bbv^{(K)}$ is just a formal index. Set 
\begin{eqnarray}
\label{faithSchur}
\cFS_\bi= \bigoplus_{K\subseteq \bbI}\bigoplus_{\bu_K\in\bU_K} \cFS^{\bu_K}_\bi&\hbox{and}&\coFS_\bi= \bigoplus_{K\subseteq \bbI}\bigoplus_{\bu_K\in\bU_K} \coFS^{\bu_K}_\bi.\quad\quad
\end{eqnarray}
These are the underlying spaces for two faithful representations:

\begin{proposition}\label{cfaithful}
\begin{enumerate}[i.)]
\item\label{fac} There is a faithful representation $\crho$ of $\ccS_\bi$ on $\cFS_\bi$
where the basis elements 
 of $\ccS_\bi$ as in \eqref{cschurbasisidemf} act via
\begin{eqnarray*} 
&\quad\quad\crho\left(\e_{(w\cdot\bu)_{K_2}}\bb^1_{K_2,wJ}\bb^w_{wJ,J}\hat\bb^p_{J,J}\e_{\bu_{J}}\bb^1_{J,K_1}\e_{\bu_{K_1}}\right)
f\bv^{(K_1)}& \\&=\quad\e_{(w\cdot\bu)_{K_2}}\left(\sum_{a \in D^{K_2}_{\emptyset,wJ}}T_a\right)T_wg_p\e_{\bu_{J}}\left(\sum_{b\in D^{K_1}_{ \emptyset, J}}T_b\right)f\bv^{(K_2)}&
\end{eqnarray*}
for $f\bv^{(K_1)} \in \cFS^{\bu_{K_1}}_\bi$, where $g_p$ is as defined in \eqref{gp} and $\bu_K=(\bu,K)$.
\item\label{faco} There is a faithful representation $\corho$ of $\ccS_\bi$ on $\coFS_\bi$
where basis elements of $\ccS_\bi$ as in \eqref{cschurbasisidemf}  act via
\begin{eqnarray*}  
&\quad\quad\corho\left(\e_{(w\cdot\bu)_{K_2}}\bb^1_{K_2,wJ}\bb^w_{wJ,J}\hat\bb^p_{J,J}\e_{\bu_{J}}\bb^1_{J,K_1}
\e_{\bu_{K_1}}\right)f\bbv^{(K_1)}& \\
&=\e_{(w\cdot\bu)_{K_2}}\left(\sum_{a \in D^{K_2}_{\emptyset,wJ}}T_a^\sharp\right)T_w^\sharp g_p^\sharp\e_{\bu_{J}}\left(\sum_{b\in D^{K_1}_{ \emptyset, J}}T_b^\sharp\right)f\bbv^{(K_2)}.
\end{eqnarray*}
for $f\bbv^{(K_1)}\in  \coFS^{\bu_{K_1}}_\bi$,
where $g_p$ is as defined in \eqref{gp} and $\bu_K=(\bu,K)$.
\end{enumerate}
\end{proposition}
\proof We prove \eqref{fac}, the proof of \eqref{faco} being analogous. We first claim  that 
$\cFS_\bi \cong \ccS_\bi\otimes_\cS \FS.$ Indeed, using that 
$$\FS^K = \left\{f \bv^{(K)}\left |  \begin{array}[c]{l} f \bv^{(K)} \in \FF[X_1^{\pm 1}, \dots, X_n^{\pm 1}]\bv^{(K)}\\(T_i-q)f\bv^{(K)} =0 \text{ for all } i \in K\end{array}\right.\right\},$$ we see that 
\small
\begin{eqnarray*}
\e_{\bu_K}\ccS_\bi\otimes_\cS \FS^K& =&  \left \{ f  \bv^{(K)}  \left | \begin{array}[c]{l} f  \bv^{(K)}  \in\displaystyle\bigoplus_{\bu' \in \W_K\cdot \bu_K}e_{\bu'}\FF[[X_1^{-1}, \dots, X_n^{-1}]] \bv^{(K)} \\ (T_i-q) f\bv^{(K)} =0 \text{ for all }  i \in K\end{array}\right.\right\}\\& = &\cFS^{\bu_K}_\bi
\end{eqnarray*} 
\normalsize
by Lemma~\ref{Lakritz}; hence the claim follows. Since we defined our action $\crho$ to coincide with $\rho$ (cf. \eqref{rhodef} and Remark~\ref{actiononGamma}) on elements of $\cS$, the fact that $\crho$ is a faithful representation follows immediately from Theorem~\ref{faithfulrep}.
\endproof
\begin{remark}
Similarly to Corollary~\ref{corhash}, we have an isomorphism of $\ccS_\bi$-modules 
$\cFS_\bi \cong {}^\sharp(\coFS_{-\bi})$ via  $f\e_{\bu_{K_1}}\bv^{(K_1)} \mapsto f^{\sharp}\e_{-\bu_{K_1}}\bbv^{(K_1)}.$
\end{remark}

\section{The action of (algebraic) merges on $\coFS_\bi$}
\label{Section6}

In this section, we describe, explicitly and in detail, the action of a simple merge on the twisted faithful representation $\coFS_\bi$, as we will later use this to compare $\cS_\bi$ to the quiver Schur algebra.  In Proposition~\ref{generatingbetter},  we will deduce a generating set for the Schur algebra which refines Corollary~\ref{corgenerating}.

\subsection{Basic formulae for algebraic merges}
We start by describing some combinatorics of distinguished coset representatives in $D_{\emptyset,J}$, where $J\subseteq\bbI$. 

Thus let $J\subseteq\bbI$ be fixed. Let $\bu=(u_1,u_2,\ldots, u_n)$ and consider $(\bu,J)$ as in \eqref{uJ}. Then for a permutation $w\in\fS$ we have that
\begin{eqnarray*}
\quad w\in D_{\emptyset,J}&\Leftrightarrow&w(k_j+1)<w(k_j+2)<\cdots <w(k_j+d^j),\;\text{ for  $1\leq j\leq r$},
\end{eqnarray*}
where $k_j=\sum_{j'\leq j} d^{j'}$ (and $d^0=0$). 
This means the numbers inside each part of $J$ are kept increasing when applying the permutation $w\in \fS$; i.e.\  in the its permutation diagram, the two strands from the same $J$-part do not cross.
\begin{lemma}
\label{maxparab}
Let $J=I-\{a\}$ for some $a\in\bbI$ and set $b=n-a$. Then $D_{\emptyset,J}$ consists precisely of the elements
\begin{equation}
\label{segments}
(s_{c_{b}}s_{{c_b}+1}\cdots s_{n-1}) (s_{c_{b-1}}\ldots s_{n-2})\cdots(s_{c_2}\ldots s_{a+1}) (s_{c_1}s_{{c_1}+1}\ldots s_a),
\end{equation}
where $1\leq c_1<c_2<\cdots<c_b$ and by convention $(s_r s_{r+1}\cdots s_k)=1$ if $r>k$. 
\end{lemma}
\proof
This is a standard fact, see for e.g. \cite[Proposition A.2]{StrTL}.
\endproof
\begin{definition}
For  $J\subseteq K\subseteq \bbI$, define the {\it algebraic merges}
\begin{eqnarray}
\label{twistedmerge}
\bC^{K}_{J}=\sum_{w \in D^{K}_{\emptyset, J}}w, &\text{and}& \bM^{K}_{J}=\sum_{w \in D^{K}_{\emptyset, J}}T_w^\sharp. 
\end{eqnarray}
We also write $\bC^{\typ(K)}_{\typ(J)}$ instead of $\bC^{K}_{J}$, respectively $\bM^{\typ(K)}_{\typ(J)}$ instead of $\bM^{K}_{J}$.
\end{definition}
\begin{remark}
\label{movethrough}
Note that $\bC^{K}_{J}=\bC^{K,c}_{J,c}$ and $\bM^{K}_{J}=\bM^{K,c}_{J,c}$ for any $c$ if we interpret the smaller symmetric group as a subgroup of the larger one. We will use this fact tacitly. Moreover, by definition and using (H-7) and Remark~\ref{commute}, we have  $\bC^{K}_{J}f=f\bC^{K}_{J}$ and $\bM^{K}_{J}f=f\bM^{K}_{J}$  for any $W_J$-invariant polynomial. 
\end{remark}
\begin{example}
\label{indstart1}
For instance, if $J=\bbI-\{a\}$ for some $a\in\bbI$ and $K=\bbI$, then $\bC^{K}_{J}=\bC_{a,n-a}^n$ is precisely the sum over all elements of the form \eqref{segments}.  
\end{example}

We first state a few easy properties of these algebraic merges.

\begin{lemma}[Associativity] 
Assume $n=1+(a-1)+(b-1)$ with each summand in $\mZ_{\geq 0}$. Then 
\begin{eqnarray}
\label{associativity}
&\bC^{a+b-1}_{a,b-1}\bC^{a, b-1}_{a-1,1,b-1}\quad =\quad \bC^{a+b-1}_{a-1,1,b-1}\quad=\quad\bC^{a+b-1}_{a-1,b} \bC^{a-1,b}_{a-1,1,b-1}.&
\end{eqnarray}
The analogous formula holds for the $\bM$ as well.
\end{lemma}
\proof
The middle term of \eqref{associativity} is precisely the sum of all permutations $w$  of $1,2,\ldots, n$ such that the numbers stay increasing inside the parts of size $a-1,1,b-1$. But each such $w$ can be written as a composition of the form $w=xy$, where $y\in\fS_a$ permutes the first $a$ numbers, but keeps the first $a-1$ increasing, followed by a permutation $x$ which keeps the first $a$ numbers  $y(i)$, ${1\leq i\leq n}$ in order and keeps the last $b-1$ numbers increasing. Moreover, each such $xy$ gives rise to a unique $w$ in the sum. This proves the first equality. The second is similar starting with permuting the last $b$ numbers instead. The statement for $\bM$ follows by the same arguments.
\endproof

\begin{lemma}[Splitting off a simple reflection]
\label{cosetsum}
Let $n=a+b$ with $a,b\in\mZ_{>0}$. Then we have the following equalities
\begin{eqnarray}
\bC^{a+b}_{a,b} &= &
 \bC_{1,a,b-1}^{1,a+b-1}s_{1}\cdots s_a + \bC_{1,a-1,b}^{1,a+b-1}\label{f1}\\
&=&\bC^{1,a+b-1}_{1,a,b-1}\bC^{a+1, b-1}_{a,1,b-1}   -\bC^{1,a+b-1}_{1,a-1,b}  \bC^{a, b}_{a,1,b-1}  +\bC^{1,a+b-1}_{1,a-1,b}.\label{f2}
%
\end{eqnarray}
The analogous formulae hold for the $\bM$ as well.
\end{lemma}
\proof
Consider  the set of elements from \eqref{segments} and divide them into those which  contain $s_1$ (in the rightmost factor) 
 and those which do not. This division corresponds precisely to the two summands on the right hand side of \eqref{f1}. Note that as a special case of \eqref{f1} we obtain  
 \begin{eqnarray}
 \label{specialcase}
\bC_{a,1}^{a+1}&=&\bC_{1,a}^{1,a}s_1s_2\cdots s_a+\bC_{1,a-1,1}^{1,a}=s_1s_2\cdots s_a+\bC_{1,a-1,1}^{1,a}.
 \end{eqnarray}
 To verify \eqref{f2}, it suffices to show
\begin{eqnarray}
\bC^{1,a+b-1}_{1,a,b-1}\bC^{a+1, b-1}_{a,1,b-1}   -\bC^{1,a+b-1}_{1,a-1,b}  \bC^{a, b}_{a,1,b-1} 
&=& \bC_{1,a,b-1}^{1,a+b-1}s_{1}\cdots s_a.
\end{eqnarray}
However, thanks to \eqref{specialcase}, the left hand side equals 
\begin{eqnarray*}
LHS&=&\bC^{1,a+b-1}_{1,a,b-1}\left(s_1\cdots s_a + \bC^{1,a,\;\;\;\;\;\;\;b-1}_{1,a-1,1,b-1}  \right) -\bC^{1,a+b-1}_{1,a-1,b}\bC^{a, \;\;\;b}_{a,1,b-1} \\
&=&\bC^{1,a+b-1}_{1,a,b-1}s_1\cdots s_a +\bC^{1,a+b-1}_{1,a-1,b} \bC^{1,a-1,b}_{1,a-1,1,b-1} -\bC^{1,a+b-1}_{1,a-1,b}  \bC^{a, b}_{a,1,b-1}\\ 
&=&\bC^{1,a+b-1}_{1,a,b-1}s_1\cdots s_a +\bC^{1,a+b-1}_{1,a-1,b} \left(\bC^{1,a-1,b}_{1,a-1,1,b-1} -\bC^{a, b}_{a,1,b-1}\right)\\
&=&\bC^{1,a+b-1}_{1,a,b-1}s_1\cdots s_a.
\end{eqnarray*}
Here the second equality follows from \eqref{associativity}, the third is clear and the last one follows from the obvious fact that the expression inside the brackets is zero.  Hence the claim follows. Note that the same arguments work for $\bM$ as well, since we have not used the quadratic relation (H-1).
\endproof

\begin{definition}
For $1\leq i\not=j\leq n$ set  
\begin{equation}
\beta_{i,j}=qX_i-X_j\quad\text{and}\quad \gamma_{i,j}=X_i-X_j\quad\text{and finally}\quad\gt_{i,j}=\frac{\beta_{ij}}{\gamma_{ij}}.
\end{equation}
\end{definition}

\begin{lemma}
\label{oje}
The  equality 
$\gt_{1,2}\gt_{2,3}-\gt_{1,2}\gt_{1,3}+\gt_{1,3}\gt_{3,2}=q$ holds.
\end{lemma}
\proof
One easily checks that $\beta_{2,3}\gamma_{1,3}-\beta_{1,3}\gamma_{2,3}=(q-1)X_3\gamma_{1,2}$. 
Thus, if we set $\gamma=\gamma_{2,3}\gamma_{1,3}$,  then 
\begin{eqnarray}
\label{oje1}
\gt_{1,2}(\gt_{2,3}-\gt_{1,3})&=&\frac{\beta_{1,2}}{\gamma_{1,2}}\left(\frac{\beta_{2,3}}{\gamma_{{2,3}}}-\frac{\beta_{1,3}}{\gamma_{1,3}}\right)\;=\;\frac{(q-1)X_3\beta_{1,2}}{\gamma}.
\end{eqnarray}
On the other hand one checks easily that $q\gamma_{1,3}\gamma{2,3}+\beta_{1,3}\beta_{3,2}=(q-1)\beta_{1,2}$ and thus
\begin{eqnarray}
\label{oje2}
q-\gt_{1,3}\gt_{3,2}&=&\frac{q\gamma_{1,3}\gamma_{2,3}+\beta_{1,3}\beta_{3,2}}{\gamma}  \;=\; \frac{(q-1)X_3\beta_{1,2}}{\gamma}.
\end{eqnarray}
Subtracting \eqref{oje2} from \eqref{oje1} gives $\gt_{1,2}(\gt_{2,3}-\gt_{1,3})-q+\gt_{1,3}\gt_{3,2}=0$. 
\endproof
The action of simple merges on  polynomials and rational functions in the $X_i$ is quite subtle, but produces interesting formulae.

\begin{example}
\label{indstart2}
For instance we have $\bC^{2}_{1,1}(\gt_{1,2})=(1+q)$. This is because $\bC^{2}_{1,1}(\gt_{1,2})=(1+s_1)(\gt_{1,2})=\frac{qX_1-X_2-qX_2+X_1}{X_1-X_2}=(1+q)$.
\end{example}
More generally, we have the following equalities of rational functions in the $X_i$:

\begin{lemma}\label{tech3}
Let $1\leq c\leq n-1$ and $0\leq a\leq n$. Then the following holds
\begin{enumerate}[i.)]
\item\label{first} For any $a$: $\bC^{a,c}_{a,1,c-1}(\prod_{k=a+2}^{a+c}\gt_{a+1,k}) =\sum_{k=0}^{c-1} q^r$. 
\item\label{second} For $a\geq 1$: $\bC^{a,c}_{a,1,c-1}(\gt_{1,a+1}\prod_{k=a+2}^{a+c}\gt_{a+1,k}) = \prod_{k=a+1}^{a+c}\gt_{1,k} +\sum_{r=1}^{c-1} q^r$. 
\end{enumerate}
The same formulae hold for the $\bM$ as well.
\end{lemma}
\proof
Without loss of generality, we may assume $a=0$ in \eqref{first} and $a=1$ in \eqref{second}, since the general case then follows by shifting labels. We prove both statements in parallel by induction on $c$. The base case is $c=2$. (For the extreme case $c=1$ we have $\bC_{1,1}^{1,1}(\gt_{1,2})$ respectively $\bC_{1,0}^{1}(1)=1$ by convention.)
For  \eqref{first}, the base case is Example~\ref{indstart2}, while for \eqref{second} we need to show $(1+s_2)(\gt_{1,2}\gt_{2,3})=\gt_{1,2}\gt_{1,3}+q$, or equivalently $\gt_{1,2}(\gt_{2,3}-\gt_{1,3})=q-\gt_{1,3}\gt_{3,2}$. This, however, is Lemma~\ref{oje}.  So assume now both, \eqref{first} and \eqref{second}, are true for $c-1$.

For \eqref{first} we abbreviate  $\Pi=\gt_{1,2}\prod_{k=3}^{c+1}\gt_{2,k}=\gt_{1,2}\Pi'$ and obtain
\begin{eqnarray*}
\bC^{c}_{1,c-1}\Pi&\stackrel{\eqref{f1}}{=}&
 \left(\bC^{1,c-1}_{1,1,c-2}(1+s_{1})\right)\Pi\\
&=&\Pi+\bC^{1,c-1}_{1,1,c-2} \left(s_{1}(\gt_{1,2}) \Pi'\right)\\
&=&\Pi+\bC^{1,c-1}_{1,1,c-2}  (1+q- \gt_{1,2})\Pi'\\
&=&\Pi+ (1+q)\left(\bC^{1,c-1}_{1,1,c-2} \Pi' \right) - \bC^{1,c-1}_{1,1,c-2} \gt_{1,2} \Pi' \\
&\stackrel{\text{ind. hyp.}}{=}&\prod_{k=2}^{c}\gt_{1,k} + (1+q)\sum_{r=0}^{c-2}q^r - \prod_{k=2}^{c}\gt_{1,k} -\sum_{r=1}^{c-2}q^r \;=\;\sum_{r=0}^{c-1}q^r,
\end{eqnarray*}
where in the penultimate line we have used the induction hypothesis for $c-1$, namely \eqref{second} for the first summand and \eqref{first} for the second summand. In the third line we have also used the induction hypothesis for $c=2$ for \eqref{second}.

For \eqref{second}, we abbreviate  $Z=\prod_{k=2}^{c}\gt_{2,k}=\gt_{2,3}Z'$ and obtain 
\begin{eqnarray*}
&&\bC^{1,c}_{1,1,c-1}(\gt_{1,2}Z)\\
&\stackrel{\eqref{f1}}=&\left((\bC^{1,1,c-1}_{1,1,1,c-2}s_{2})+1\right)(\gt_{1,2}Z) \\
&=&\gt_{1,2}Z+ \bC^{1,1,c-1}_{1,1,1,c-2}\left(s_{2}(\gt_{1,2}\gt_{2,3})Z'\right) \\ 
& =  &\gt_{1,2}Z+\bC^{1,1,c-1}_{1,1,1,c-2}(\gt_{1,2}\gt_{1,3} -\gt_{1,2}\gt_{2,3}  +q) Z'\\ 
&=& \gt_{1,2}Z
+\left(\gt_{1,2} \bC^{1,1,c-1}_{1,1,1,c-2}\gt_{1,3} Z'\right)
-\left(\gt_{1,2} \bC^{1,1,c-1}_{1,1,1,c-2}\gt_{2,3} Z' \right)+ q\left( \bC^{1,1,c-1}_{1,1,1,c-2}Z' \right).
\end{eqnarray*}
Now we use the induction hypothesis for $c-1$, namely \eqref{second} for the middle and \eqref{first} for the last summand and obtain 
\begin{equation*}
( \gt_{1,2}\prod_{k=3}^{c+1}\gt_{2,k})+(\prod_{k=2}^{c+1}\gt_{1,k}) +(\gt_{1,2}\sum_{r=1}^{c-2} q^r ) 
 -( \gt_{1,2}\prod_{k=3}^{c+1}\gt_{2,k} )- (\gt_{1,2}\sum_{r=1}^{c-2} q^r )
  + q\sum_{r=0}^{c-2}q^r
  \end{equation*}
  Hence, altogether we have
\begin{eqnarray*}
 \bC^{1,c}_{1,1,c-1}(\gt_{1,2}Z)&=& \prod_{k=2}^{c+1}\gt_{1,k} +  q\sum_{r=0}^{c-2}q^r\quad=\quad \prod_{k=2}^{c+1}\gt_{1,k} +  \sum_{r=1}^{c-1}q^r
\end{eqnarray*}
This completes the proof.
\endproof

\subsection{Algebraic merges in the faithful representation}
In this subsection we give explicit formulae for the action of the simple merges on the faithful representation $\coFS_\bi$ from \eqref{faithSchur}. 

{\it Setup for the whole subsection:} Assume that  $K=\bbI$, so $\bu_K=(\bu,K)$ with $\bu=(1^{d_1},2^{d_2},\ldots, e^{d_e})$ and $J=\bbI\setminus\{a\}$. Set $b=n-a$. Let $\bu_J\in \bU_J$. Set $a_i=d_i^1(\bu_J)$ and $b_i=d_i^2(\bu_J)$ in the notation of Definition~\ref{dimmatrix}. In particular, $a_i+b_i=d_i$ for  all $i=1,\ldots,e$. Then the action of the merge from \eqref{bbdef} can be expressed in the $\#$-twisted versions as follows:
\begin{proposition}\label{mergeactionschur}
For $\coFS_\bi$ and $f\in  \hat{R}_{+,\bu_J}$ as in Proposition~\ref{cfaithful} \eqref{faco} we have
\begin{eqnarray*}
e_{\bu}\corho(\bb^1_{K,J})f\bbv^{(J)}&=&e_{\bu_K}\;{\bC_{a_1,b_1,a_2,b_2,\ldots,a_e,b_e}^{d_1,\;\;\;\;d_2,\;\;\;\;\;\ldots,   d_e}}\;\sigma_{\bu_J}^{\bu_K} \prod_{k=a+1}^n\prod_{l=1}^a \gt_{l,k}f\bbv^{(K)}.
\end{eqnarray*}
where $\sigma_{\bu_J}^{\bu_K}$ is as in Definition~\ref{smsigma}, explicitly 
\begin{eqnarray}
\label{oursigma}
\sigma_{\bu_J}^{\bu_K}&=&\sigma^{(1^{d_1}, 2^{d_2},\ldots , e^{d_e})}_{(1^{a_1}, 2^{a_2},\ldots , e^{a_e}|1^{b_1}, \ldots , e^{b_e})}.
\end{eqnarray}
\end{proposition}
\begin{remark}
Note that by Lemma~\ref{Lakritz} the component at $e_{\bu}$ completely determines the element in $\coFS_\bi$.
\end{remark}

As a direct consequence from the definitions, \eqref{twistedmerge}, we obtain
\begin{corollary}
We have
\begin{eqnarray*}
e_{\bu}\corho(\bb^1_{K,J})f\bbv^{(J)}= e_{\bu}\left(\sum_{w}w\sigma_{\bu_J}^{\bu_K}\right) \prod_{k=a+1}^n\prod_{l=1}^a \gt_{l,k}f\bbv^{(K)}.
\end{eqnarray*}
where the sum runs over $w\in D^{I_{\bu_K}}_{\emptyset,\sigma_{\bu_J}^{\bu_K} (I_{\bu_J})}$ with notation from Definition~\ref{WuJ}.
\end{corollary}

The proof of Proposition~\ref{mergeactionschur} is rather technical and occupies the rest of the subsection, proceeding by induction on $a+b$, with the base case being trivial.  We start with some preparations. 

\begin{remark}
\label{sigmas}
Note that $\sigma_{\bu_J}^{\bu_K}$ in \eqref{oursigma} factorizes as $\sigma_{\bu_J}^{\bu_K}=\sigma_1\sigma_2$  where
\begin{eqnarray*}
\sigma_1&=&\sigma^{{(1\mid 1^{d_1-1}, 2^{d_2},3^{d_3},\;\;\;\;\ldots\;\;\;\;\;\;, e^{d_e})}}_{(1\mid 1^{a_1},\ldots, e^{a_e}\mid 1^{b_1-1},2^{b_2},\ldots, e^{b_e})} \\
\sigma_2&=&\sigma^{(1^{a_1}\mid 1,2^{a_2},\ldots \;e^{a_e}\mid 1^{b_1-1},2^{b_2},\ldots, e^{b_e})}_{(1^{a_1}\mid 2^{a_2},\ldots, e^{a_e}\mid 1\mid 1^{b_1-1},2^{b_2},\ldots, e^{b_e})}\quad  = \quad s_{a_1+1}\cdots  s_a
\end{eqnarray*}
 Moreover, we have $\sigma_{\bu_J}^{\bu_K}=\sigma_3\sigma_4\sigma_5^{-1}$, where 
\begin{eqnarray*}
\sigma_3&=&\sigma^{(1\mid 1^{d_1-1},\;2^{d_2},\; 3^{d_3},\;\;\;\ldots\;\;\;\;\ldots,\;\; t^{d_t},\;\;\ldots\;\;\;\;\ldots\;\; e^{a_e} )}
_{(1\mid 1^{a_1-1},2^{a_2},\ldots ,t^{a_t+1},\ldots, e^{a_e}\mid 1^{b_1}, \ldots  t^{b_t-1},\; \ldots,\; e^{b_e}) }\\
\quad\sigma_4 &= &\sigma^{(1^{a_1},\ldots, t^{a_t+1},\ldots,e^{a_e}\mid 1^{b_1},  \ldots, t^{b_t-1},\ldots, e^{b_e}) }_{(1^{a_1},2^{a_2}, \;\: \ldots\;\ldots\:,\;e^{a_e}\mid t\mid 1^{b_1},  \ldots ,t^{b_t-1},\ldots ,e^{b_e}) }= s_{a_1+a_2+\cdots+a_t+1}\cdots s_{a-1}s_a\nonumber,\\
\sigma_5 &= &\sigma^{(1^{a_1}\ldots e^{a_e}|1^{b_1}\;\;\;\;\;  \ldots \;\;\;\;\;\;\;\;e^{b_e}) }_{(1^{a_1}\ldots e^{a_e}|t|1^{b_1}  \ldots t^{b_t-1}\ldots e^{b_e}) }= s_{(a+\sum_{j=1}^{t-1}b_j)} \cdots s_{a+2}s_{a+1},
\end{eqnarray*}
\end{remark}

By definition of the representation in Proposition~\ref{cfaithful} \eqref{faco}  we have
\begin{eqnarray*}
e_{\bu}\corho(\bb^1_{K,J})f\bbv^{(J)} \quad=\quad e_{\bu}\:\bM^{a+b}_{a,b}\:f\bbv^{(J)}&=:&\qq
\end{eqnarray*}
Applying Lemma~\ref{cosetsum} to the right hand side we obtain
\begin{eqnarray*}
\qq=e_{\bu}( \bM^{1,a+b-1}_{1,\:a,\;b-1}\;\bM^{a+1, b-1}_{a,\;1,\;b-1}   +\bM^{1,a+b-1}_{1,\;a-1,\;b} \left( 1-\bM^{a, b}_{a,1,b-1} \right)f\bbv^{(J)}.
\end{eqnarray*}
Note that $ \bM^{a, b}_{a,1,b-1} $ commutes past $f$ by Remark~\ref{movethrough}, and therefore we have $ \bM^{a, b}_{a,1,b-1}  f\bbv = \sum_{s=0}^{b-1}q^s f\bbv$, using Lemma~\ref{maxparab}, \eqref{eqhash} and $T_i\bbv=-\bbv$ by \eqref{easyformel}.
Altogether we obtain
\begin{eqnarray}
\label{S1andS2}
\qq& =& e_{\bu}\left( \bM^{1,a+b-1}_{1,a,b-1}\bM^{a+1, b-1}_{a,\;1,\;b-1}   - \bM^{1,a+b-1}_{1,\;a-1,\;b}\left(\sum_{s=1}^{b-1}q^s\right)\right)f\bbv^{(J)}.
\end{eqnarray}
We can then rewrite the two summands, which we denote by  $\qqone$ and $\qqtwo$ respectively,as in the following two lemmas.

\begin{lemma}
\label{S2}
The second summand in \eqref{S1andS2} equals 
\begin{eqnarray*}
\qqtwo&=&
-\left( \sum_{s=1}^{b-1}q^s  \right)  \bC^{1,d_1-1,\;\;\;\;d_2,\;\;\;\;\ldots, d_e}_{1,a_1-1,b_1,a_2,b_2,\ldots, a_e,b_e}   \sigma_{\bu_J}^{\bu_K}   \left(\prod_{k=a+1}^n\prod_{j=2}^a \gt_{j,k}\right)   f\bbv^{(K)}.
\end{eqnarray*}
where $\sigma_{\bu_J}^{\bu_K} $ is as in \eqref{oursigma}.
\end{lemma}
\proof
First we analyse which idempotents $e$ can appear to the right of the merge for the result not to be annihilated by  $e_{\bu}$, i.e. 
$$\bM^{1,a+b-1}_{1,\;a-1,\;b}e\left(\sum_{s=1}^{b-1}q^s\right)f\bbv^{(J)}\neq 0.$$
Clearly, any such $e$ must be a summand of $\e_{(1|1^{a_1-1}, 2^{a_2},\ldots , e^{a_e}|1^{b_1}, \ldots , e^{a_e})}.$  Moreover, note that $ \sigma_{\bu_J}^{\bu_K}=\sigma^{(1|1^{d_1-1}, 2^{d_2},\ldots , e^{d_e})}_{(1|1^{a_1-1}, 2^{a_2},\ldots , e^{a_e}|1^{b_1}, \ldots , e^{b_e})}$.
Then 
$$e_{\bu}\bM^{1,a+b-1}_{1,a-1,b} f\bbv^{(J)} =   \bC^{1,d_1-1,\;\;\;\;d_2,\;\;\;\;\ldots, d_e}_{1,a_1-1,b_1,a_2,b_2,\ldots, a_e,b_e}   \sigma_{\bu_J}^{\bu_K}f\bbv^{(J)}$$  by the inductive hypothesis, since  $a+b-1 < a+b$. 
\endproof

\begin{lemma}
\label{S1}
The first summand $\qqone$ in \eqref{S1andS2} equals 
\begin{eqnarray*}
e_{\bu}\big[\bC_{1,a_1,b_1-1,a_2,b_2,\ldots,a_e,b_e}^{1,d_1-1,\;\;\;\;d_2,\;\;\;\;\ldots,d_e} \sigma_1 P 
\unclear \sigma_2   \prod_{i=1}^{a} \gt_{i,a+1}\big] f\bbv^{(K)} \\
 +\quad e_{\bu}
 \big[\sum_{t=2}^e \bC_{1, a_1-1,b_1,a_2,b_2,\ldots,a_{t}+1,b_{t}-1\ldots, a_e,b_e }^{1,d_1-1,\;\;\;\;d_2,\;\;\;\;\;\ldots,d_t,\;\;\:\;\;\;\;\;\;\;\;\;\;\;\;\;d_e} Z  \prod_{i=1}^{a} \gt_{i,a+1} \big] f\bbv^{(K)}&
\end{eqnarray*}
where $P=\prod_{k=a+2}^n\prod_{j=2}^{a+1} \gt_{j,k}$, and $Z=\sigma_3 P \gt_{j,k} \bC_{a_1,\ldots, a_t,1,a_{t+1},\ldots,a_e,b-1}^{a_1,\ldots, a_t+1,a_{t+1},\ldots,a_e,b-1} \sigma_4$. The elements $\sigma$ are as in Remark~\ref{sigmas}.
\end{lemma}
\proof
To have $e_{\bu}\bM^{1,a+b-1}_{1,a,b-1}\e_{\bu_{1,a,b-1}}\not=0$ we need $\bu_{1,a,b-1}$ to be of the form $(t|\dots)$ for some $t\in\{1,2,\ldots, e\}$. 
We distinguish two cases, namely 
\begin{equation*}
e_{\bu}\;\bM^{1,a+b-1}_{1,a,b-1}\;  \e_{(1|1^{a_1}\ldots e^{a_e}|1^{b_1-1}2^{b_2}\ldots e^{b_e})} \; \bM^{a+1, b-1}_{a,1,b-1}\; \e_{(1^{a_1},\ldots, e^{a_e}|1|1^{b_1-1},2^{b_2},\ldots, e^{b_e})}
\end{equation*}
and 
 \begin{equation*}
e_{\bu}\;\bM^{1,a+b-1}_{1,a,b-1}\;  \e_{(t|1^{a_1}\ldots e^{a_e}|1^{b_1}\ldots t^{b_t-1}\ldots e^{b_e})}\;  \bM^{a+1, b-1}_{a,1,b-1}\;  \e_{(1^{a_1},\ldots, e^{a_e}|t|1^{b_1},\ldots, t^{b_t-1},\ldots, e^{b_e})}.
\end{equation*}
for $t=2,\ldots, e$.
Then the claim follows directly from the definition of the permutations $\sigma$ in Remark~\ref{sigmas} and the induction hypothesis.
\endproof

\begin{lemma} The second summand in Lemma~\ref{S1}, denoted by $\qqonetwo$,  equals
\begin{eqnarray}
\label{secondsummandofQ1}
e_{\bu}\sum_{t=2}^e \bC_{1, a_1-1,b_1,a_2,b_2,\ldots,a_t,1,b_t-1,a_{t+1} \ldots, a_e,b_e }^{1,d_1-1,\;\;\;\;d_2,\ldots,\;\;\;\;d_t,\;\;\;\;\;\;\;\;\;\;\;\;d_{t+1},\;\;\; d_e}\sigma_3\sigma_4 P  \prod_{i=1}^{a} \gt_{i,a+1} f\bbv^{(K)}.\quad
\end{eqnarray}
\end{lemma}
\proof
We first rewrite the term $Z$ appearing in Lemma~\ref{S1} as
\begin{eqnarray*}
&\sigma_3 P
\bC_{a_1,\ldots, a_t,\;1,\;a_{t+1},\ldots,a_e,b-1}^{a_1,\ldots, a_t+1,a_{t+1},\ldots,a_e,b-1} \sigma_4
\quad =\quad\sigma_3\bC_{a_1,\ldots, a_t,\;1,\;a_{t+1},\ldots,a_e,b-1}^{a_1,\ldots, a_t+1,a_{t+1},\ldots,a_e,b-1} \sigma_4P &\\
&=\quad\bC_{a_1,b_1,\ldots, a_t,1,b_t-1,a_{t+1},b_{t+1},\;\ldots,\;a_e,b_e}^{a_1,b_1,\ldots, a_t+1,b_t-1,a_{t+1},b_{t+1},\ldots, a_e,b_e}\sigma_3\sigma_4 P&
\end{eqnarray*}
where the first equality uses that  $P=\prod_{k=a+2}^n\prod_{j=2}^{a+1} \gt_{j,k}$ is $\sigma_4$-invariant and Remark~\ref{movethrough}. For the second equality one checks that $\sigma_3$ from Remark~\ref{sigmas} commutes with $\bC_{a_1,\ldots, a_t,\;1,\;a_{t+1},\ldots,a_e,b-1}^{a_1,\ldots, a_t+1,a_{t+1},\ldots,a_e,b-1}$. The claim follows by substituting this into the formula in Lemma~\ref{S1} and using associativity of merges.
\endproof
\begin{lemma} The first summand in Lemma~\ref{S1},  denoted by $\qqoneone$,  equals
\label{firstsummandofQ1}
\begin{eqnarray*}
\qqoneone&=&e_{\bu}\big[\bC_{1,d_1-1,a_2,b_2,\ldots,a_e,b_e}^{1,d_1-1,d_2,\;\;\;\;\ldots,d_e}
\bC_{1,a_1-1,1,b_1-1,a_2,b_2,\ldots,a_e,b_e}^{1,d_1-1,\;\;\;\;\;\;\;\;\;\;\;a_2,b_2,\ldots,a_e,b_e} \sigma_{\bu_J}^{\bu_K}P
  \\&+&\bC_{1,a_1,b_1-1,a_2,b_2,\ldots,a_e,b_e}^{1,d_1-1,\;\;\;\;d_2,\;\;\;\;\ldots,d_e}
\!\!\!(s_1s_2\cdots s_{a_1}\sigma_{\bu_J}^{\bu_K})\prod_{k=a+2}^n\prod_{j=1}^{a} \gt_{j,k}       \big] \prod_{i=1}^{a} \gt_{i,a+1}f\bbv^{(K)}
\end{eqnarray*}
\end{lemma}
We first use the special case \eqref{specialcase} of Lemma ~\ref{cosetsum} to write 
\begin{eqnarray}
\unclear& =&  \uunclear+s_1\cdots s_{a_1}
  \end{eqnarray}
 
 This element obviously commutes with $P=\prod_{k=a+2}^n\prod_{j=2}^{a+1} \gt_{j,k}$ by (H-7). The same holds for 
 $\sigma_2$. Hence $\qqoneone$ equals 
\begin{eqnarray}
&&e_{\bu}\big[\bC_{1,a_1,b_1-1,a_2,b_2,\ldots,a_e,b_e}^{1,d_1-1,\;\;\;\;d_2,\;\;\;\;\ldots,d_e} \sigma_1 P\unclear \sigma_2        \big]w    \nonumber \\ [\dist] 
 &=&e_{\bu}\bC_{1,a_1,b_1-1,a_2,b_2,\ldots,a_e,b_e}^{1,d_1-1,\;\;\;\;d_2,\;\;\;\;\ldots,d_e} \sigma_1  \big[\uunclear \sigma_2 P+Ps_1\cdots s_a     \big] w\nonumber \\ [\dist] 
 &=&e_{\bu}\big[\bC_{1,a_1,b_1-1,a_2,b_2,\ldots,a_e,b_e}^{1,d_1-1,\;\;\;\;d_2,\;\;\;\;\ldots,d_e} \sigma_1  \uunclear \sigma_2P   
  \\ [\dist] &&\quad +\bC_{1,a_1,b_1-1,a_2,b_2,\ldots,a_e,b_e}^{1,d_1-1,\;\;\;\;d_2,\;\;\;\;\ldots,d_e}
 \sigma_1 
s_1\cdots s_a \prod_{k=a+2}^n\prod_{j=1}^{a} \gt_{j,k}       \big] w\label{middle}
\end{eqnarray}
where we have abbreviated  $w= \prod_{i=1}^{a} \gt_{i,a+1}f\bbv^{(K)}$, and, for the last equality, used that  $s_1\cdots s_a$ maps the set $\{1,2,\ldots, a\}$ to $\{2,3,\ldots, a+1\}$. 

Next, observe that $ \sigma_1$ and $\uunclear$ commute. Therefore, 
\begin{eqnarray}
&&\bC_{1,a_1,b_1-1,a_2,b_2,\ldots,a_e,b_e}^{1,d_1-1,\;\;\;\;d_2,\;\;\;\;\ldots,d_e} \sigma_1\uunclear \sigma_2\nonumber\\ [\dist] 
&=&\bC_{1,a_1-1,1,b_1-1,a_2,b_2,\ldots,a_e,b_e}^{1,d_1-1,\;\;\;\;\;\;\;\;\;\;\;\;d_2,\;\;\;\;\ldots,d_e} \sigma_{\bu_J}^{\bu_K}\nonumber\\[\dist] 
&=&
\bC_{1,d_1-1,a_2,b_2,\ldots,a_e,b_e}^{1,d_1-1,d_2,\;\;\;\;\ldots,d_e}
\bC_{1,a_1-1,1,b_1-1,a_2,b_2,\ldots,a_e,b_e}^{1,d_1-1,\;\;\;\;\;\;\;\;\;\;\;a_2,b_2,\ldots,a_e,b_e} \sigma_{\bu_J}^{\bu_K}
\label{end}
\end{eqnarray}
by the associativity for merges and the factorisation from Remark~\ref{sigmas} for the first equality and again associativity for the last equality. The claim follows by substituting \eqref{end} into \eqref{middle} and  using $\sigma_1 s_1\cdots s_a=s_1s_2\cdots s_{a_1}\sigma_{\bu_J}^{\bu_K}$.
\endproof
Altogether, the left hand side of the asserted formula in Proposition~\ref{mergeactionschur} equals
\begin{eqnarray}
\label{action}
e_{\bu}\corho(\bb^1_{K,J})f\bbv^{(J)}&=&\qqoneone+\qqonetwo+\qqtwo\nonumber\\
&=&e_{\bu}\bC_{1,d_1-1,a_2,b_2,\ldots,a_e,b_e}^{1,d_1-1,d_2,\;\;\;\;\ldots,d_e}Y\prod_{k=a+1}^n\prod_{j=2}^{a} \gt_{j,k}  f\bbv^{(K)},\quad\quad
\end{eqnarray}
where $Y=\eqref{no1}+\eqref{no2}+\eqref{3}-\eqref{4}$ with the summands given by

\begin{eqnarray}
&&\bC_{1,a_1-1,1,b_1-1,a_2,b_2,\ldots,a_e,b_e}^{1,d_1-1,\;\;\;\;\;\;\;\;\;\;\;a_2,b_2,\ldots,a_e,b_e} \sigma_{\bu_J}^{\bu_K}\gt_{1,a+1} \prod_{k=a+2}^{a} \gt_{a+1,k}\label{no1}\\
&& \bC_{1,a_1,b_1-1,a_2,b_2,\ldots,a_e,b_e}^{1,d_1-1,\;\;\;\;a_2,b_2,\ldots,a_e,b_e}
 \sigma_1 s_1\cdots s_a \prod_{k=a+1}^n \gt_{1,k}  \label{no2}\\
 &&\sum_{t=2}^e \bC_{1, a_1-1,b_1,a_2,b_2,\ldots,a_t,1,b_t-1,a_{t+1} \ldots,  a_e,b_e }^{1,d_1-1,\;\;\;\;a_2,b_2,\ldots, \;\;\;\;\;\;\;\;\;\;\;\;\;\;\;\;\;\;\;\;\;\;\; \;\;\;\;\;a_e,b_e}\sigma_3\sigma_4 \gt_{1,a+1} \prod_{k=a+2}^{a} \gt_{a+1,k}\quad\quad\label{3}\\
 &&\left( \sum_{s=1}^{b-1}q^s  \right)  \bC^{1,d_1-1,\;\;\;\;a_2,b_2,\ldots, a_e,b_e}_{1,a_1-1,b_1,a_2,b_2,\ldots, a_e,b_e}.    \label{4}
\end{eqnarray}

We rewrite \eqref{3}. First $\sum_{t=2}^e 
\bC_{1, a_1-1,b_1,a_2,b_2,\ldots,a_t,1,b_t-1,a_{t+1} \ldots, a_e,b_e }^{1,d_1-1,\;\;\;\;a_2,b_2,\ldots, \;\;\;\;\;\;\;\;\;\;\;\;\;\;\;\;\;\;\;\;\; \;\;\;\;\;a_e,b_e}\sigma_3\sigma_4$ equals
\begin{eqnarray}
&&\bC_{1, a_1-1,b_1 }^{1,d_1-1}\sum_{t=2}^e \bC_{1, a_1-1,b_1,a_2,b_2,\ldots,a_t,1,b_t-1,\ldots, a_e,b_e }^{1, a_1-1,b_1,a_2,b_2,\ldots,a_t,b_t, \ldots,a_e,b_e}\sigma_{\bu_J}^{\bu_K}\sigma_5 \nonumber\\
&=&\bC_{1, a_1-1,b_1 }^{1,d_1-1}\sigma_{\bu_J}^{\bu_K} \sum_{t=2}^e\bC_{a+b_1+\cdots b_{t-1}, 1,b_t-1, b-(b_1+\cdots +b_{t})}^{a+b_1+\cdots b_{t-1}, b_t, b-b_1-\cdots - b_{t}} \sigma_5\nonumber\\
&=&\bC_{1, a_1-1,b_1 }^{1,d_1-1}\sigma_{\bu_J}^{\bu_K} \sum_{t=2}^e \sum_{l=0}^{b_t-1} s_{a+\sum_{j=1}^{t-1}b_j+l}\cdots   s_{a+\sum_{j=1}^{t-1}b_j}\cdots s_{a+2}s_{a+1}\nonumber\\
&=&\bC_{1, a_1-1,b_1 }^{1,d_1-1}\sigma_{\bu_J}^{\bu_K} \sum_{l=b_1}^{b-1}s_ls_{l-1}\cdots s_{a+1}\nonumber\\
&=&\bC_{1, a_1-1,b_1 }^{1,d_1-1}\sigma_{\bu_J}^{\bu_K}(\bC_{a,1,b-1}^{a,b}-\bC_{a,1,b_1-1,b-b_1}^{a,b_1,b-b_1}),
\end{eqnarray}
where for the first equality we have used Remark~\ref{sigmas} and rewritten the merge as a product of two (non-interacting) merges, for the second equality the commutativity of $\sigma_{\bu_J}^{\bu_K}$ with the respective merges, for the third equality Lemma~\ref{maxparab} and the definition of $\sigma_5$, and finally for the last equality formula \eqref{specialcase}. 

On the other hand,  by Lemma~\ref{tech3} we have 
\begin{eqnarray}
\bC_{a,1,b-1}^{a,b}\gt_{1,a+1}\left( \prod_{k=a+2}^{n} \gt_{a+1,k}\right)-\left( \sum_{s=1}^{b-1}q^s \right)
&=&\prod_{k=a+1}^{n} \gt_{1,k}.
\end{eqnarray}
 which simplifies \eqref{3}+\eqref{4} further. Altogether we obtain 

\begin{eqnarray}\label{eq301}
Y&=&\bC_{1,a_1-1,1,b_1-1,a_2,b_2,\ldots,a_e,b_e}^{1,d_1-1,\;\;\;\;\;\;\;\;\;\;\;a_2,b_2,\ldots,a_e,b_e} \sigma_{\bu_J}^{\bu_K}\gt_{1,a+1} \prod_{k=a+2}^{n} \gt_{a+1,k}\label{84}\\
&&+\quad\bC_{1,a_1,b_1-1,a_2,b_2,\ldots,a_e,b_e}^{1,d_1-1,\;\;\;\;a_2,b_2,\ldots,a_e,b_e}
s_1s_2\cdots s_{a_1}\sigma_{\bu_J}^{\bu_K}\prod_{k=a+1}^n \gt_{1,k}  \label{85}\\
&&+\quad  \bC^{1,d_1-1,\;\;\;\;a_2,b_2,\ldots, a_e,b_e}_{1,a_1-1,b_1,a_2,b_2,\ldots, a_e,b_e}    
\end{eqnarray}
\begin{eqnarray}
&&+\quad  \bC^{1,d_1-1,\;\;\;\;a_2,b_2,\ldots, a_e,b_e}_{1,a_1-1,b_1,a_2,b_2,\ldots, a_e,b_e}    \sigma_{\bu_J}^{\bu_K}\prod_{k=a+1}^{n} \gt_{1,k}\label{86}\\
&&-\quad\bC_{1, a_1-1,b_1 }^{1,d_1-1}\sigma_{\bu_J}^{\bu_K}\bC_{a,1,b_1-1,b-b_1}^{a,b_1,b-b_1}\gt_{1,a+1} \prod_{k=a+2}^{n} \gt_{a+1,k}.\quad\label{87}
\end{eqnarray}
Since $\sigma_{\bu_J}^{\bu_K}\bC_{a,1,b_1-1,b-b_1}^{a,b_1,b-b_1}\; =\; \bC_{a_1,1,b_1-1}^{a_1,b_1}\sigma_{\bu_J}^{\bu_K}\;=\; \bC_{1,a_1-1,1,b_1-1}^{1,a_1-1,b_1}\sigma_{\bu_J}^{\bu_K}$, 
the terms \eqref{84} and \eqref{87} cancel.  Applying Lemma~\ref{cosetsum} to \eqref{85}+\eqref{86} gives 
\begin{eqnarray}
Y&=&\bC_{a_1,b_1,a_2,b_2,\ldots,a_e,b_e}^{d_1,\;\;\;\;a_2,b_2,\ldots,a_e,b_e}\sigma_{\bu_J}^{\bu_K}\prod_{k=a+1}^n \gt_{1,k}. 
\end{eqnarray}
Substituting this back into \eqref{action}, we obtain that $e_{\bu_K}\corho(\bb^1_{K,J})f\bbv^{(J)}$ equals
\begin{eqnarray*}
e_{\bu_K}\bC_{d_1,a_2,b_2,\ldots,a_e,b_e}^{d_1-1,d_2,\;\;\;\;\ldots,d_e}\;\bC_{a_1,b_1,a_2,b_2,\ldots,a_e,b_e}^{d_1,\;\;\;\;a_2,b_2,\ldots,a_e,b_e}\sigma_{\bu_J}^{\bu_K}\prod_{k=a+1}^n\prod_{j=1}^{a} \gt_{j,k}  f\bbv^{(K)}.
\end{eqnarray*}
The associativity property of merges gives finally the desired formula from Proposition~\ref{mergeactionschur}. This finishes the proof of  Proposition~\ref{mergeactionschur}.
\subsection{A refined generating set of the affine Schur algebra $\cS$}

In this section, we improve on our generating set for the algebra $\cS$ from Corollary~\ref{corgenerating}.
\begin{proposition}
\label{generatingbetter}
The algebra $\cS$ is generated by
\begin{equation}
\{\bb^1_{K_2,K_1}, \bb^p_{J,J}\mid   K_1,K_2 \subseteq \bbI, p \in \DM_{I_{\bu_J}}\}.
\end{equation}
In other words, the algebra is generated by the subalgebra $\cQ$ from Proposition~\ref{Qsub} and the splits and merges $\bb^1_{K_1,K_2}$.
\end{proposition}
\proof
By Proposition~\ref{Propschurbasis2} the proposed generating set  together with  $\bb^d_{dJ,J}$, where $d \in D^\bbI_{K_2,K_1}$, $J = K_1\cap d^{-1}K_2$ generate the algebra $\cS$. Hence it suffices to show that the $\bb^d_{dJ,J}$ are redundant. First note that $d\in D_{dJ,J}$, since for any $i\in J$ we have $l(ds_id^{-1}d)=l(ds_i)>l(d_i)$, because $J\subset K_1$ and $d \in D_{K_2,K_1}$. Therefore $d$ permutes the blocks in $J$ without changing the order inside the blocks. Moreover,  $d\in D_{dJ,J}$ implies that $\bb^d_{dJ,J}(\bv_J)=\bv_{dJ}T_d$ by \eqref{niceformula1}. This also implies that without loss of generality we may assume that $d$ only swaps two neighbouring parts of $J$, as an arbitrary permutation of parts can be written as a composition of swapping neighbouring ones. Since the arguments are all local we can even assume that $J$ contains only two parts, i.e. $J=\bbI\setminus{a}$. Set $b=n-a$. Note that in this case $d\in D_{dJ,J}$ is then the shortest coset representative of the longest element in $\fS$.  Hence $T_d^\#$ is the summand corresponding to the longest element $d=d_{a,b}^{a+b}$ appearing in $\bC^{a+b}_{a,b}$. Define
$\overline{\bM^{a+b}_{a,b}}=\bM^{a+b}_{a,b}-T_d^\#. $
  By Example~\ref{indstart1} we have $(\bM^{a+b}_{a,b})^\# \bv_{J}=(\bM^{a+b}_{a,b} \bbv_{J})^\# =b^1_{\bbI,J}(\bv_{J})$. Hence it suffices to show that  $\overline{\bM^{a+b}_{a,b}}^\#\bv_{J}$ can be expressed in terms of simple splits and merges applied to $\bv_J$.
We argue by induction on $a+b=n$. The base case $a=b=1$, so $n=2$, is obvious. For the general case, using \eqref{f2}, we obtain
\begin{equation}
\label{xx1}
\overline{\bM^{a+b}_{a,b}}=\bM^{1,a+b-1}_{1,a,b-1}\bM^{a+1, b-1}_{a,1,b-1}   -\bM^{1,a+b-1}_{1,a-1,b}  \bM^{a, b}_{a,1,b-1}  -T_d^\#+\bM^{1,a+b-1}_{1,a-1,b}.
\end{equation}
On the other hand if we abbreviate $d_1={d_{1,a,b-1}^{1,a+b-1}}$ and $d_2={d_{a,1,b-1}^{a+1, b-1}}$ and  set $D_1=T^\#_{d_1}$ and $D_2=T^\#_{d_2}$, then $\bM^{1,a+b-1}_{1,a,b-1}\bM^{a+1, b-1}_{a,1,b-1}$ is equal to
\begin{equation*}\begin{split}
\label{xx2}
&(\overline{\bM^{1,a+b-1}_{1,a,b-1}}+D_1)(\overline{\bM^{a+1, b-1}_{a,1,b-1}}+D_2)\\
&=\overline{\bM^{1,a+b-1}_{1,a,b-1}}\;\overline{\bM^{a+1, b-1}_{a,1,b-1}}+D_1\overline{\bM^{a+1, b-1}_{a,1,b-1}}+\overline{\bM^{1,a+b-1}_{1,a,b-1}}D_2+D_1D_2\\
&={\bM^{1,a+b-1}_{1,a,b-1}}\;\overline{\bM^{a+1, b-1}_{a,1,b-1}}+\overline{\bM^{1,a+b-1}_{1,a,b-1}}\;{\bM^{a+1, b-1}_{a,1,b-1}}-\overline{\bM^{1,a+b-1}_{1,a,b-1}}\;\overline{\bM^{a+1, b-1}_{a,1,b-1}}+D_1D_2
\end{split}\end{equation*}
where for the last line we used the equalities $D_1={\bC^{1,a+b-1}_{1,a,b-1}}-\overline{\bC^{1,a+b-1}_{1,a,b-1}}$ and the analogous one for $D_2$.

By Lemma~\ref{maxparab} we have  $d_1=(s_{b-1}s_{b}\cdots s_{n-1})(s_3s_4\cdots s_{a+1})(s_2s_3\cdots s_{a})$ and $d_2=s_1s_2\cdots s_{a-1}$, in particular $d=d_1d_2$ with $l(d)=l(d_1)+l(d_2)$. But this implies $T_d^\#=D_1D_2$ and thus we obtain from \eqref{xx1} and \eqref{xx2} the following 
\begin{eqnarray*}
\overline{\bM^{a+b}_{a,b}} &=&{\bM^{1,a+b-1}_{1,a,b-1}}\;\overline{\bM^{a+1, b-1}_{a,1,b-1}}+\overline{\bM^{1,a+b-1}_{1,a,b-1}}\;{\bM^{a+1, b-1}_{a,1,b-1}}-\overline{\bM^{1,a+b-1}_{1,a,b-1}}\;\overline{\bM^{a+1, b-1}_{a,1,b-1}}\\
&&  -\bM^{1,a+b-1}_{1,a-1,b}  \bM^{a, b}_{a,1,b-1}+\bM^{1,a+b-1}_{1,a-1,b}.
\end{eqnarray*}
Applying ${}^\#$ to the whole equation and using the inductive hypothesis, the right hand side of the equation is in the subalgebra generated by our proposed generating set, hence so is the left hand side as desired.
\endproof

\section{Quiver Hecke algebras and the isomorphism theorem}
In this section, we finally connect the constructions developed so far with the so-called quiver Hecke algebras originally introduced by Khovanov-Lauda \cite{KL} using diagrammatics and by Rouquier \cite{Rouquier} using algebraic and categorical constructions, and later connected to flagged quiver representations in \cite{VV}. The quiver Schur algebra treated in the next section is a generalisation of the quiver Hecke algebra introduced in \cite{SW} using flagged quiver representations where,  generalising \cite{VV}, partial flags are used instead of full flags only. 

\subsection{The quiver Hecke algebra}\label{klrsec}


We identify the fixed representatives $1,\ldots, e$ of $\mathbb{Z}/e\mathbb{Z}$ with the vertices in the affine Dynkin diagram $\Gamma=\Gamma_e$ attached to the affine Kac-Moody Lie algebra $\hat{\mathfrak{sl}}_e$ with the vertices numbered clockwise from $1$ to $e$, and encode the fixed ordering on the representatives by a clockwise orientation of the diagram. Recall our $\bi \in \mZ^n$ from Definition~\ref{idef}.

\begin{definition}
We denote by $\bR_\bi$ the {\it quiver Hecke algebra} associated to $\bi$. This is the unital $\FF$-algebra generated by elements $$\{e(\bu)\mid\bu \in \fS\bi\} \cup \{\psi_1, \dots, \psi_{n-1}\} \cup \{x_1, \dots, x_n\}$$ subject to relations
\begin{align*}
&e(\bu) e(\bu') = \delta_{\bu,\bu'} e(\bu);  \qquad\sum_{\bu \in \fS\bi} e(\bu) = 1;\\
&x_r e(\bu) = e(\bu) x_r; \qquad
\psi_r e(\bu) = e(s_r\cdot\bu) \psi_r;\quad x_r x_s = x_s x_r;\\
&\psi_r \psi_s = \psi_s \psi_r\qquad\hbox{if $|r-s|>1$};\\
&\psi_r x_s  = x_s \psi_r \qquad\hbox{if $s \neq r,r+1$};\\
&\psi_r x_{r+1} e(\bu) = (x_{r}\psi_{r}+\delta_{u_r,u_{r+1}})e(\bu);\quad 
x_{r+1} \psi_re(\bu) =(\psi_{r} x_{r}+\delta_{u_r,u_{r+1}}) e(\bu);
\\
&\psi_r^2e(\bu) = 
\begin{cases}
0&\hbox{if $u_r = u_{r+1}$},\\
e(\bu)&\hbox{if $u_{r+1}\neq u_r \pm 1, u_r$},\\
(x_{r+1}-x_{r})e(\bu)&\hbox{if $u_{r+1} =u_r+1, e \neq 2$},\\
(x_{r} - x_{r+1})e(\bu)&\hbox{if $u_{r+1}= u_r -1, e \neq 2$;}\\
(x_{r+1}-x_r)(x_r - x_{r+1})e(\bu)&\hbox{if $u_{r+1}= -u_r, e=2$}
\end{cases}\\
&\psi_{r}\psi_{r+1} \psi_{r} e(\bu)
=
\begin{cases}
(\psi_{r+1} \psi_{r} \psi_{r+1} +1)e(\bu)&\hbox{if $u_{r+2}=u_r =u_{r+1} -1,e \neq 2$},\\
(\psi_{r+1} \psi_{r} \psi_{r+1} -1)e(\bu)&\hbox{if $u_{r+2}=u_r = u_{r+1}+1, e \neq 2$},\\
 (\dagger)&\hbox{if $u_{r+2}=u_r = -u_{r+1}, e= 2$},\\
\psi_{r+1} \psi_{r} \psi_{r+1} e(\bu)&\hbox{otherwise}.
\end{cases}
\end{align*}
where $(\dagger)=(\psi_{r+1} \psi_{r} \psi_{r+1} -x_r-x_{r+2}+2x_{r+1})e(\bu).$
\end{definition}
The commutative subalgebra of $\bR_\bi$ generated by $\{e(\bu)\mid \bu \in \fS\bi\}\cup \{x_1, \dots, x_n\}$ is denoted by $\bP_\bi$. 

The following can be found in \cite{KL} or \cite{Rouquier} and can be easily verified.

\begin{lemma}
The algebra $\bR_\bi$ has a faithful representation on $$\Fa_\bi=\bigoplus_{\bu \in \fS\bi}e(\bu) \FF[x_1,\dots x_n]\cdot \mathbbm{1}$$ where the action of $\bP_\bi$ is the regular action and  
\begin{eqnarray}
\label{actionpsi}
\psi_r e(\bu) \cdot \mathbbm{1} &=& 
\left\{\begin{array}{ll}0&\hbox{if $u_r = u_{r+1}$},\\ (x_{r}-x_{r+1})e(s_k\cdot\bu)\cdot \mathbbm{1} & \hbox{if $u_{r+1}=u_r+1$},  \\ e(s_k\cdot \bu)\cdot \mathbbm{1}  & \hbox{if $u_{r+1}\neq u_r, u_r+1$}.  \end{array}\right.
\end{eqnarray}
\end{lemma}

Again, we complete our algebra, this time at the sequence of ideals $\bJ_m = \bR_\bi \bI^m \bR_\bi$ where $\bI$ is the ideal in  $\FF[x_1,\dots x_n]$ generated by all $x_i, i=1 \dots n$. We denote the completed algebra by $\cbR_\bi$, its polynomial subalgebra generated by $\{e(\bu)|\bu \in \fS\bi\}\cup \{x_1, \dots,  x_n\}$ by $\cbP_\bi$ and complete our faithful representation to obtain $\cFa_\bi= \cbR_\bi \otimes_{\bR_\bi}\Fa_\bi$. 

\subsection{The isomorphism $\ccH_\bi \cong \cbR_\bi$}\label{heckeklrisosec}

Next, we provide an explicit isomorphism between the completed algebras  $\ccH_\bi$ and $\cbR_\bi$. A similar result can be found in \cite{Ben}. Note that our approach differs from that used in \cite{Ben}, in that we do not use exponentials, but rather an affine shift following the ideas of \cite{BK}, where a corresponding isomorphism for cyclotomic quotients was established.

First observe that there is an isomorphism 
\begin{eqnarray}
\label{isogamma}
\gamma: \ccP_\bi\to \cbP_\bi: &&(X_i-q^{u_i})e_\bu \mapsto -q^{u_i}x_ie(\bu)
\end{eqnarray}
 which induces an isomorphism
\begin{eqnarray*}
\coFH_\bi \to \cFa_\bi :&& \prod_{i=1}^n X_i^{a_i}e_\bu\bbv \mapsto  \prod_{i=1}^n (q^{u_i}(1-x_i))^{a_i}e(\bu)
\end{eqnarray*}
between the restrictions of the respective faithful representations to the subalgebras $\ccP_\bi$ respectively $\cbP_\bi$. Direct computation then verifies the following.

\begin{theorem}\label{heckeklriso}
The isomorphism $\gamma$ from \eqref{isogamma} can be extended to an isomorphism of algebras $\gamma: \ccH_\bi \to \cbR_\bi$, via $\gamma(e_{s_r\cdot\bu}\Phi_r) = A^\bu_r \psi_r e(\bu)$ where
\begin{eqnarray*}
A_r^\bu& = &
\begin{cases}
\; 1-q-x_r +qx_{r+1}&\hbox{if $u_{r+1}=u_r$},\\[2mm]
\;\frac{-q}{\left(1-q-x_{r+1}+qx_r\right)}&\hbox{if $u_{r+1}=u_r+1$,}\\[2mm]
\;\frac{q^{u_{r+1}}(1-x_r)-q^{u_r+1}(1-x_{r+1}) }{q^{u_r}(1-x_{r+1})-q^{u_{r+1}}(1-x_r)} & \hbox{if $u_{r+1}\neq u_r, u_r+1$}
\end{cases}
\end{eqnarray*}
\end{theorem}

\section{Quiver Schur algebras}

In this section we establish our main isomorphism theorem by connecting the (algebraically defined) affine Schur algebra with the (geometrically defined) quiver Schur algebra from \cite{SW}. We do this via an auxiliary {\it modified} quiver Schur algebra:
\subsection{The modified quiver Schur algebra}
Recall $\bi$ from Definition~\ref{idef}. 
For $J\in \bbI$, and $\bu_J =(u_1,\cdots, u_{t_1}|u_{t_1+1},\cdots, u_{t_2}|\cdots |u_{t_{r-1} +1}, \cdots, u_{t_r})\in \bU_J$ with dimension vector $\dimv=\dimv(\bi) =\dimv(\bu,J)=(d_1,d_2,\ldots,d_e)$ we define
\begin{eqnarray*}
\Lambda_{\bu_J}=\FF[y_{1,1}, \ldots, y_{1,d_1},y_{2,1}, \ldots, y_{2,d_2},\ldots, y_{e,d_e}]^{\fS_{\bu_J}}
 \hbox{ and }
  \Lambda=\bigoplus_{\substack{J\subseteq\bbI\\ \bu_J\in\bU_J}}\Lambda_{\bu_J}.
  \end{eqnarray*}

  \begin{definition}  
Let $1\leq i\leq e$. For $1\leq k\leq r$ let $c(k)_i=\sum_{j=1}^k d_i^j$, using Definition~\ref{dimmatrix} and $c(0)_i=0$. Then the {\it total reversed Euler class of $\bu_J$} is 
\begin{eqnarray}
\label{Etotal}
\mathtt{E}_{\bu_J} &=& \prod_{i=1}^e \prod_{s=1}^{r-1}\prod_{j=c(s-1)_i+1}^{c(s)_i} \prod_{k=c(s)_{i+1}+1}^{d_{i+1}} (y_{i,j}-y_{i+1, k}).\end{eqnarray}
with $\mathtt{E}_{\bu_\bbI}:=1$. The {\it total symmetriser} is defined as as
\begin{eqnarray}
\label{Stotal}
\mathtt{S}_{\bu_J} &=& \prod_{i=1}^e \prod_{s=1}^{r-1}\prod_{j=c(s-1)_i+1}^{c(s)_i} \prod_{k=c(s)_{i}+1}^{d_{i}} (y_{i,j}-y_{i, k})
\end{eqnarray}
with $\mathtt{S}_{\bu_\bbI}:=1$. 
More generally, assume $J\subset K$ and let $\bu_K \in \bU_K$ be a merge of $\bu_J$. Then their {\it relative reversed Euler class} and the {\it relative symmetriser} are 
\begin{eqnarray}
\label{relative}
\mathtt{E}_{\bu_J}^{\bu_K}\;= \;\frac{\mathtt{E}_{\bu_J}}{\mathtt{E}_{\bu_K}}&\quad\text{respectively}\quad&\mathtt{S}_{\bu_J}^{\bu_K}\;= \;\frac{\mathtt{S}_{\bu_J}}{\mathtt{S}_{\bu_K}}.
 \end{eqnarray}
 In particular, the special case $K=\bbI$ gives the total reversed Euler class respectively the total symmetriser. Note that the relative Euler class and relative symmetriser are again polynomials.
\end{definition} 
\begin{example}
Note that $\mathtt{E}_{(1|2)}=y_{1,1}-y_{2,1}$, and $\mathtt{E}_{(2|1)}=1$ if $e>2$ whereas $\mathtt{E}_{(2|1)}=y_{2,1}-y_{1,1}$ if $e=2$ and $\mathtt{E}_{(1|1)}=1$. Moreover, $\mathtt{S}_{(1|2)}=\mathtt{S}_{(2|1)}=1$, whereas $\mathtt{E}_{(1|1)}=y_{1,1}-y_{1,2}$.
\end{example}
\begin{example}
Let for instance $\bu_J = (1,2|1,1|1,2|1)$, $\bu_K = (1,2|1,1,1,2|1)$.  Then for $e\geq 3$ we have $\mathtt{E}_{\bu_J}=(y_{1,1}-y_{2,2})(y_{1,2}-y_{2,2})(y_{1,3}-y_{2,2})=:E$ and $\mathtt{E}_{\bu_K}=(y_{1,1}-y_{2,2})$ and therefore $\mathtt{E}_{\bu_J}^{\bu_K}=(y_{1,2}-y_{2,2})(y_{1,3}-y_{2,2})$, whereas for $e=2$ we have 
$\mathtt{E}_{\bu_J}=E(y_{2,1}-y_{1,2})(y_{2,1}-y_{1,3})(y_{2,1}-y_{1,4})(y_{2,1}-y_{1,5})(y_{2,2}-y_{1,5})$ but again $\mathtt{E}_{\bu_J}^{\bu_K}=(y_{1,2}-y_{2,2})(y_{1,3}-y_{2,2})$. On the other hand, for any $e\geq 2$, we have $\mathtt{S}_{\bu_J}^{\bu_K}=(y_{1,2}-y_{1,4})(y_{1,3}-y_{1,4})$.
\end{example}

\begin{definition}
\label{modifiedSchur} We define the {\it modified quiver Schur algebra} 
$\bQ_\bi$
as the subalgebra of $\End_\FF(\Lambda)$ generated by the following endomorphisms:
\begin{trivlist}
\item   $\quad\bullet$  the {\it idempotents} $e(\bu_J)$ for $\bu_J\in\bU_J$ for any $J$, projecting onto $\Lambda_{\bu_J}$,
\item   $\quad\bullet$  the {\it polynomial}  $e(\bu_J) pe(\bu_J)$ for $\bu_J\in\bU_J$ for any $J$, and $p\in \Lambda_{\bu_J}$, 
\item \quad\quad defined as multiplication by $p$ on the summand $\Lambda_{\bu_J}$.
\item   $\quad\bullet$  the {\it splits} $\bigcurlyvee_{\bu_K}^{\bu_{J}}$ for $J\subset K$, $\bu_J=(\bu,J)\in\bU_J$, which are just the embedding
\item \quad\quad of the summand $\Lambda_{\bu_K}$ into the summand $\Lambda_{\bu_{J}}$. 

\item   $\quad\bullet$  the {\it merges} $\bigcurlywedge_{\bu_{J}}^{\bu_K}$ for $J\subset K$ $\bu_J=(\bu,J)\in\bU_J$, 
defined on  
$f\in \Lambda_{\bu'_{J'} } $ by 
\begin{eqnarray}
\label{Dem1}
f&\mapsto 
\begin{cases}
\Delta( \mathtt{E}_{\bu_{J}}^{\bu_K}  f)\in  \Lambda_{\bu_K}&\text{if $\bu_J =\bu'_{J'}$},\\ 
0 &\text{otherwise}.
\end{cases}
\end{eqnarray}
\item \quad\quad where 
$\Delta=\Delta_{\bu_{J}}^{\bu_K}$ sends an element $f$ to the total invariant
$\bC_{I_{\bu_J}}^{I_{\bu_K}} \left( \frac{f}{\mathtt{S}_{\bu_{J}}^{\bu_K} }\right). $
\end{trivlist}
Using a reformulation in terms of Demazure operators, see Proposition~\ref{matchDemazure}, it follows that $\bC_{I_{\bu_J}}^{I_{\bu_K}} \left( \frac{f}{\mathtt{S}_{\bu_{J}}^{\bu_K} }\right)$ is indeed again a polynomial.
\end{definition}

\begin{example}
Consider for instance $\bu_{J}=(1|1)$ and $\bu_{K}=(11)$. Then for $f\in\Lambda_{\bu_K}=\FF[y_{1,1},y_{1,2}]$ we have  $\bigcurlywedge_{\bu_{J}}^{\bu_K}(f)=\Delta(\frac{f}{y_{1,1}-y_{1,2}})=2f_2$, where $f=f_1+(y_{1,1}-y_{1,2})f_2$ with (uniquely determined) $f_1,f_2\in  \Lambda_{\bu_J}$.
\end{example}

\begin{example}
\label{mergeexplicit}
Let us describe the merge endomorphism explicitly in the simplest case where  $\bu_J =(1^{a_1},2^{a_2},\ldots, e^{a_e}|1^{b_1},2^{b_2},\ldots, e^{b_e} )$ has only two parts, hence $\bu_K=(1^{d_1},2^{d_2},\ldots, e^{d_e})$ with $d_i=a_i+b_i$.  Then our formulae give
%
\begin{eqnarray*}
\Delta( \mathtt{E}_{\bu_{J}}^{\bu_K}  f) & =&\bC_{a_1,b_1,a_2,b_2,\ldots,a_e,b_e}^{d_1,\;\;\;\; d_2,\;\;\;\;\;\ldots,d_e} \left(\frac{\displaystyle\prod_{i=1}^e \displaystyle\prod_{j=1}^{a_{i}}\displaystyle\prod_{k=a_{i+1}+1}^{d_{i+1}}(y_{i,j}-y_{i+1,k})}{\displaystyle\prod_{i=1}^e\displaystyle\prod_{j=1}^{a_i}\prod_{k=a_{i}+1}^{d_i}( y_{i,j}-y_{i,k})} f \right) .
\end{eqnarray*}
\end{example}

We denote by $\cbQ_\bi$ the completion of $\bQ_\bi$ at the ideal generated by all $e(\bu_J)pe(\bu_J)$ for all  $J\subset\bbI,\bu_J\in \bU_J$ and all $p\in \Lambda_{\bu_J}$ with zero constant term. 

Then  $\cbQ_\bi$ has a faithful representation on $$\cLa=\bigoplus_{\substack{J\subseteq \bbI\\ \bu_J\in\bU_J}}\cLa_{\bu_J}$$ where $\cLa_{\bu_J}$ is the completion of $\Lambda_{\bu_J}$ at all polynomials with zero constant term.

\subsection{The quiver Schur algebra}

Here we recall the definition of the quiver Schur algebra $\bA_\bi$, introduced by the second author and Webster in \cite{SW}.
For $J\in \bbI$, and $\bu_J =(u_1,\cdots, u_{t_1}|u_{t_1+1},\cdots, u_{t_2}|\cdots |u_{t_{r-1} +1}, \cdots, u_{t_r})\in \bU_J$ with dimension vector $\dimv=\dimv(\bi)=\dimv(\bu,J)=(d_1,d_2,\ldots,d_e)$  we define
\begin{eqnarray}\label{polyringQS} 
\round{\Lambda}_{\bu_J}=\FF[z_{1,1}, \ldots, z_{1,d_1},z_{2,1}, \ldots, z_{2,d_2},\ldots z_{e,d_e}]^{\fS_{\bu_J}},
&& \round{\Lambda}=\bigoplus_{\substack{J\subseteq\bbI\\ \bu_J\in\bU_J}}\round{\Lambda}_{\bu_J}.\quad\quad
\end{eqnarray}
\begin{definition}
Let $1\leq i\leq e$. For $1\leq k\leq r$ let $c(k)_i=\sum_{j=1}^k d_i^j$, using Definition~\ref{dimmatrix}, and $c(0)_i=0$. 
The {\it total Euler class} 
for ${\bu_J}$ is defined as
\begin{eqnarray*}
\round{\mathtt{E}}_{\bu_J} &=& \prod_{i=1}^e \prod_{s=1}^{r-1}\prod_{j=c(s-1)_{i+1}+1}^{c(s)_{i+1}} \prod_{k=c(s)_{i}+1}^{d_{i}} (z_{i+1,j}-z_{i, k}).
\end{eqnarray*}
(notice that this is $\fS_{\bu_J}$-invariant), and its {\it symmetriser} is defined as
\begin{eqnarray*}
\round{\mathtt{S}}_{\bu_J} &= &\prod_{i=1}^e \prod_{s=1}^{r-1}\prod_{j=c(s-1)_i+1}^{c(s)_i} \prod_{k=c(s)_{i}+1}^{d_{i}} (z_{i,j}-z_{i, k}).
\end{eqnarray*}
(Note that $\round{\mathtt{S}}_{\bu_J}$ is the same as $\mathtt{S}_{\bu_J}$, only written in variables $z_{i,j}$ instead of $y_{i,j}$.)
\end{definition}

More generally, assume $J\subset K$ and let $\bu_K \in \bU_K$ be a merge of $\bu_J$. Then their {\it relative  Euler class} respectively the {\it relative symmetriser}  are defined as
\begin{eqnarray*}
\round{\mathtt{E}}_{\bu_J}^{\bu_K}\;=\; \frac{\round{\mathtt{E}}_{\bu_J}}{\round{\mathtt{E}}_{\bu_K}}&\quad\text{respectively}\quad&\round{\mathtt{S}}_{\bu_J}^{\bu_K}\;= \;\frac{\round{\mathtt{S}}_{\bu_J}}{\round{\mathtt{S}}_{\bu_K}}.
\end{eqnarray*}

The following was introduced in \cite{SW}.
\begin{definition}
\label{QS}
The {\it quiver Schur algebra} $\bA_\bi$ 
is the subalgebra of $\End_\FF(\round\Lambda)$ generated by the following endomorphisms:

\begin{trivlist}
\item $\quad\bullet$ the {\it idempotents $e(\bu_J)$ for $\bu_J\in\bU_J$ for any $J$, 
projecting onto $\round\Lambda_{\bu_J}$,
\item $\quad\bullet$ the {\it polynomial}  $e(\bu_J) pe(\bu_J)$ for any $J$ and $p\in \round\Lambda_{\bu_J}$, 
\item \quad\quad defined as multiplication by $p$ on $\round\Lambda_{\bu_J}$.
\item $\quad\bullet$ the {\it splits} $\round{\bigcurlyvee}_{\bu_K}^{\bu_{J}}$ for $J\subset K$ and $\bu_J = (\bu,J) \in \bU_J$, defined as
\begin{eqnarray*}
f&\mapsto 
\begin{cases}
\round{\mathtt{E}}_{\bu_J}^{\bu_K}f \in \round\Lambda_{\bu_{J}} &\text{if $\bu'_{K'} =\bu_K$},\\ 
0 &\text{otherwise}.
\end{cases}
\end{eqnarray*}
\item\quad\quad for $f\in  \round\Lambda_{\bu'_{K'}}$. In other words, a split is just the embedding of the summand \item\quad\quad $\round\Lambda_{\bu_K}$ into the summand $\round\Lambda_{\bu_{J}}$, followed by multiplication with $\round{\mathtt{E}}_{\bu_J}^{\bu_K}$. 
\item $\quad\bullet$ the {\it merges} $\round{\bigcurlywedge}_{\bu_{J}}^{\bu_K}$ for $J\subset K$ and $\bu_{J}=(\bu,J) \in \bU_J$,
defined as 
\begin{eqnarray}
f&\mapsto 
\begin{cases}
\round\Delta( f)\in  \round\Lambda_{\bu_K}&\text{if $\bu'_{J'} =\bu_{J}$},\\ 
0 &\text{otherwise}.
\end{cases}
\label{goodmerge}
\end{eqnarray}
\item\quad\quad for $f\in \round\Lambda_{\bu'_{J'} }$, where  $\round\Delta=\round\Delta_{\bu_{J}}^{\bu_K}$ sends an element $f$ to the total invariant
\begin{eqnarray}
\label{Demazurelike}
\round\Delta(f)&=&\bC_{I_{\bu_J}}^{I_{\bu_K}} \left( \frac{f}{\mathtt{S}_{\bu_{J}}^{\bu_K} }\right). 
\end{eqnarray}
}
\end{trivlist}
\end{definition}
Note that again the translation into Demazure operators from Proposition~\ref{matchDemazure} ensures that $\round\Delta(f)$ is in fact a polynomial.

Again we have simple splits and merges: In case $|K\setminus J|=1$ and  $\bu_J=(\bu,J)\in\bU_J$, we call  $\round{\bigcurlyvee}_{\bu_K}^{\bu_{J}}$  a {\it simple split} and 
$\round{\bigcurlywedge}_{\bu_{J}}^{\bu_K}$ a {\it simple merge}. If $K=\bbI$ and $J=K\setminus\{a\}$,  we  also denote these by $\round{\bigcurlyvee}_{\bf{a},\bf{b}}^{\bf{a}+\bf{b}}$ respectively $\round{\bigcurlywedge}^{\bf{a},\bf{b}}_{\bf{a}+\bf{b}}$, where $a_i=d_i^1$ and $b_i=d_i^2$ with the notation from Definition~\ref{dimmatrix}.

\begin{remark}
In \cite{SW}, the quiver Schur algebra was only defined over $\mathbb{C}$, since the involved geometry would require more advanced tools. However, the faithful representation defined in \cite{SW} makes sense over any field, so we {\it define} the quiver Schur algebra over an arbitrary field as in Definition~\ref{QS}. In characteristic zero it agrees with the one defined in \cite{SW} by  Remark~\ref{foranyfield} below.
\end{remark}

\begin{example}
Note that in the case of a simple merge $\round{\bigcurlywedge}_{\bu_{J}}^{\bu_K}$ of the form $\bu_J =(1^{a_1},2^{a_2},\ldots, e^{a_e}|1^{b_1},2^{b_2},\ldots, e^{b_e})$ and
$K=\bbI$, hence $\bu_K=(1^{d_1},2^{d_2},\ldots, e^{d_e})$, formula \eqref{goodmerge} simplifies to $\round{\bigcurlywedge}_{\bu_{J}}^{\bu_K}(f)=
\round{\bigcurlywedge}_{\bf{a},\bf{b}}^{\bf{a}+\bf{b}}(f)=
\round\Delta(  f)  $, which yields
\begin{eqnarray*}
\round{\bigcurlywedge}_{\bu_{J}}^{\bu_K}(f)
&=&\bC_{a_1,b_1,a_2,b_2,\ldots,a_e,b_e}^{d_1,\ldots,d_e} \;\frac{f}{\displaystyle\prod_{i=1}^e\displaystyle\prod_{j=1}^{a_i}\prod_{k=a_i+1}^{a_i+d_i} (z_{i,j}-z_{i,k}) }\;
\end{eqnarray*}
with the relative  Euler class $\round{\mathtt{E}}_{\bu'_J}^{\bu_K} = \prod_{i=1}^e \prod_{k=1}^{c_{i+1}}\prod_{j=c_i+1}^{d_i+c_i}(z_{i+1,k}-z_{i,j}).$
This, indeed, corresponds to the formulae given in \cite{SW}.
\end{example}

 \begin{remark}
 \label{foranyfield}
Assume $K=\bbI$, $J=K\setminus\{a\}$ and $\bu_J=(\bu,J)\in \bU_J$. Let $(\bf{a},\bf{b})$ be the dimension vector of $\bu$, where $a_i=d_i^1$ and $b_i=d_i^2$ with the notation from Definition~\ref{dimmatrix}. Write $S_{\mathbf{a}+\mathbf{b}}$ for $W_{I_{\bu_K}}$ and $S_{\mathbf{a},\mathbf{b}}$ for $W_{I_{\bu_J}}$. Sending $f$ to
\small
\begin{equation*}
\displaystyle\sum_{w\in S_{\mathbf{a}+\mathbf{b}}} (-1)^{l(w)} w(f)\prod_{i=1}^e\frac{1}{a_i!b_i!}
\frac{\displaystyle w\bigg(\prod_{1\leq j<k\leq a_i} (z_{i,j}-z_{i,k})\prod_{a_i<j<k\leq a_i+b_i}(z_{i,j}-z_{i,k})\bigg)}
{\displaystyle \prod_{1\leq j<k\leq a_i+b_i} (z_{i,j}-z_{i,k})}.
\end{equation*}
\normalsize
is the action of a simple merge $\round{\bigcurlywedge}_{\bf{a},\bf{b}}^{\bf{a}+\bf{b}}$ on the faithful representation  \eqref{polyringQS} defined in \cite{SW}. Note that, in contrast to formula \eqref{goodmerge}, this expression does not make sense in positive characteristic in general. In characteristic zero however, this expression coincides with \eqref{goodmerge}, since we have
\small
\begin{eqnarray*}
&&\sum_{w\in S_{\mathbf{a}+\mathbf{b}}}w(f)  \prod_{i=1}^e\frac{1}{a_i!b_i!}
w\left(\frac{ \prod_{1\leq j<k\leq a_i} (z_{i,j}-z_{i,k})\prod_{a_i<j<k\leq a_i+b_i}(z_{i,j}-z_{i,k})}
{ \prod_{1\leq j<k\leq a_i+b_i} (z_{i,j}-z_{i,k})}\right)\\
&=&\sum_{w\in S_{\mathbf{a}+\mathbf{b}}} w(f) \prod_{i=1}^e\frac{1}{a_i!b_i!}
w\left(\frac{1}
{ \prod_{j=1}^{a_i}\prod_{k=a_i+1}^{a_i+b_i} (z_{i,j}-z_{i,k}) }\right)\\
&=&\sum_{w\in D^{I_{\bu_K}}_{\emptyset, I_{\bu_J}}} w(f) \prod_{i=1}^e
w\left(\frac{1}
{ \prod_{j=1}^{a_i}\prod_{k=a_i+1}^{a_i+b_i} (z_{i,j}-z_{i,k}) }\right)\\
&=&\sum_{w\in D^{I_{\bu_K}}_{\emptyset, I_{\bu_J}}} w\left(f \prod_{i=1}^e
\frac{1}
{ \prod_{j=1}^{a_i}\prod_{k=a_i+1}^{a_i+b_i} (z_{i,j}-z_{i,k}) }\right)\\
&=& \bC_{a_1,\;\:b_1,a_2,\;\;b_2,\ldots,a_e,\;b_e}^{a_1+b_1,a_2+b_2,\ldots,a_e+b_e}\left(f\prod_{i=1}^e
\frac{1}
{ \prod_{j=1}^{a_i}\prod_{k=a_i+1}^{a_i+b_i} (z_{i,j}-z_{i,k}) }\right).
\end{eqnarray*}
\normalsize
\end{remark}

\subsection{Demazure or Bernstein-Gelfand-Gelfand difference operators}
\label{SectionDem}
In this subsection we connect our merging formulae to the classical difference operators.  For $1\leq i\leq n-1$, the $i$th {\it Demazure operator} or {\it divided difference operator} from \cite{Demazure} or \cite{BGG} is the endomorphism 
\begin{equation}\label{Demazure}
\Delta_i \colon\FF[X_1,\dots,X_{n}]\rightarrow\FF[X_1,\dots,X_{n}], \quad f\mapsto 
\frac{f-s_i(f)}{X_i-X_{i+1}}.
\end{equation}
For $f,g\in \FF[X_1,\dots,X_{n}]$, we have $\Delta_i(fg)=\Delta_i(f)g+s_i(f)\Delta_i(g)$, in particular   $\Delta_i(fg)=f\Delta_i(g)$ if $f$ is $s_i$-invariant, and  $\Delta_i^2=0$. Moreover, for a reduced expression $w=s_{i_1}s_{i_2}\cdots s_{i_r}$, the operator $\Delta_w=\Delta_{s_{i_1}}\Delta_{s_{i_2}}\cdots \Delta_{s_{i_r}}$ is independent of the chosen reduced expression. For future reference, we record the following.
\begin{lemma}
\label{bekommeneins}
Let $w_0=w_0(n)\in\fS$ be the longest element, then 
\begin{eqnarray*}
\Delta_{w_0}(X_1^{n-1}X_2^{n-2}\cdots  X_{n-2}^2X_{n-1})&=&1.
\end{eqnarray*}
In particular if $w_K$ is  the longest element of some parabolic subgroup $W_K$ in $\fS$, then there exists some polynomial $h$ such that $\Delta_{w_K}(h)=1$. 
\end{lemma}

There exists in fact a closed formula for $\Delta_{w_0}$, namely
\begin{eqnarray}
\label{Fulton}
\Delta_{w_0}&=&\frac{1}{\blacktriangle}\sum_{w\in\fS}(-1)^{l(w)}w = \sum_{w\in\fS} w \frac{1}{\blacktriangle},
\end{eqnarray}
where $\blacktriangle=\prod_{1\leq i<j\leq n}(X_i-X_j)$. The first equality can, for example,  be found in \cite[Lemma 12]{Fulton}, the second equality is an elementary calculation observing that a simple transposition changes the sign of $\blacktriangle$ by $-1$.

The merges from \eqref{Demazurelike} can then be rephrased in terms of Demazure operators as follows (explaining the notations $\Delta$ and $\round\Delta$)
\begin{proposition}
\label{matchDemazure}
Assume we are in the setup from  \eqref{Demazurelike} and abbreviate $J'=I_{\bu_J}$ and $K'=I_{\bu_K}$ using Definition~\ref{WuJ}. Then we have the equality 
\begin{eqnarray}
\bC_{J}^{K} \left( \frac{f}{\mathtt{S}_{\bu_{J}}^{\bu_K} }\right)&=&\Delta_{d^K_J}.
\end{eqnarray}
on $\Lambda_{\bu_{J}}$, where $d^K_J\in D^{K'}_{J'}$ is of maximal length (i.e. the representative of the longest element in $W_{J'}\subset W_{K'}$). 
\end{proposition}
\proof
Let $f\in \Lambda_{\bu_{J}}$ and let $w_J$ be the longest element in $W_{J'}\subset W_{K'}$ and $w_K$ the longest element in $W_{K'}$. 
Then 
$$\Delta_{d^K_J}(f)=\Delta_{d^K_J}(f\cdot1)=\Delta_{d^K_J}(f\cdot\Delta_{w_J}(h))=\Delta_{d^K_J}(\Delta_{w_J}(fh))=\Delta_{w_0}(fh)$$
with $h$ as in Lemma~\ref{bekommeneins}. Here, for the penultimate equality, we have used that $f$ is $W_J$-invariant and for the last equality that $dw_J=w_K$. With the explicit formula from \eqref{Fulton}, we obtain that $\Delta_{w_0}(fh)$ equals 
\begin{eqnarray*}
\frac{1}{\blacktriangle}\sum_{w\in W_{K'}}(-1)^{l(w)}w(fh)
\;=\;\frac{1}{\blacktriangle}\left(\sum_{d\in D^{K'}_{J'}}(-1)^{l(d)} d(f) \right)\left(\sum_{w\in W_{J'}}(-1)^{l(w)} w(h)\right),
\end{eqnarray*}
with $\blacktriangle$ the Vandermonde determinant $\prod_{i=1}^e \prod_{1\leq j<k\leq d_i} (z_{i,j}-z_{i, k})$, equal to
$$\prod_{i=1}^e \prod_{s=1}^{r-1}\prod_{j=c(s-1)_i+1}^{c(s)_i} 
\prod_{k=c(s)_{i}+1}^{d_{i}} (z_{i,j}-z_{i, k})\cdot\prod_{i=1}^e \prod_{s=1}^{r-1}\prod_{c(s-1)_i+1\leq j<k\leq c(s)_i} (z_{i,j}-z_{i, k}).$$
From \eqref{Fulton} we obtain $\Delta_{w_0}(fh)=\;\bC_{J}^{K} \left( \frac{f}{\mathtt{S}_{\bu_{J}}^{\bu_K} }\right)\Delta_{w_J}(h)$, and we are done.
\endproof

\subsection{The shifted quiver Schur algebra $\bB_\bi$}
We now define the {\it shifted quiver Schur algebra} $\bB_\bi$ in almost the same way, except that the  Euler class moves from the split to the merge. (Note that this algebra is again defined for any field.) More precisely, define 
$\bB_\bi$ as the subalgebra of $\End_\FF(\round\Lambda)$ generated by the {\it idempotents} and {\it polynomial}  as in $\bA_\bi$, and splits and merges now defined as 
\begin{trivlist}
\item $\quad\bullet$ the {\it splits}: $\lrarrow\bigcurlyvee_{\bu_K}^{\bu'_{J}}$ for $J\subset K$ and $\bu_J = (\bu,J) \in \bU_J$, given on 
$f
\in \round\Lambda_{\bu'_{K'}}$ 
by 
\begin{eqnarray*}
f&\mapsto 
\begin{cases}
f \in \round\Lambda_{\bu_{J}} &\text{if $\bu'_{K'} =\bu_K$},\\ 
0 &\text{otherwise}.
\end{cases}
\end{eqnarray*}
\item \quad\quad In other words, a split is just the embedding of $\round\Lambda_{\bu_K}$ into  $\round\Lambda_{\bu_{J}}$.
\item $\quad\bullet$ the {\it merges}: $\lrarrow\bigcurlywedge_{\bu_{J}}^{\bu_K}$ for any $J\subset K$  and $\bu_J = (\bu,J) \in \bU_J$, given 
 by 
\begin{eqnarray*}
f&\mapsto 
\begin{cases}
\round\Delta(\round{\mathtt{E}}_{\bu_J}^{\bu_K} f)\in  \round\Lambda_{\bu_K}&\text{if $\bu'_{J'} =\bu_{J}$},\\ 
0 &\text{otherwise}.
\end{cases}
\end{eqnarray*}
\end{trivlist}

\section{The main result: The Isomorphism Theorem}\label{mainiso}
The goal of this section is to prove the main Isomorphism Theorem~\ref{IsoTheorem} between the completed affine Schur algebra and the quiver Schur algebra.

\subsection{The isomorphism $\ccS_\bi\cong \cbQ_\bi$.}

We now compare the faithful representation of the modified quiver Schur algebra $\bQ_\bi$ with the faithful representation of the completed affine Schur algebra $\ccS_\bi$. The following isomorphism of vector spaces
\begin{eqnarray*}
\tau:  \FF[y_{1,1}, \ldots, y_{1,d_1},\ldots, y_{e,d_e}]&\to&R_+ = \FF[Y_{1,1}, \ldots, Y_{1,d_1},\ldots, Y_{e,d_e}]\\
y_{c,j}&\mapsto&  1-q^{-c}Y_{c,j}
\end{eqnarray*}
induces an isomorphism $\tau_{\bu_K}: \cLa_{\bu_K}\to\coFS^{\bu_K}_\bi=\hat R_{+,\bu_K}\bbv^{(K)}, 
f\bbv^{(K)}\mapsto  \tau(f)\bbv^{(K)}$
of vector spaces and thus a total isomorphism 
\begin{eqnarray}
\label{tautotal}
\boldsymbol{\tau} = \bigoplus_{\substack{K\subseteq \bbI\\ \bu_K\in\bU_K}} \tau_{\bu_K}\; : &&  \cLa\to\coFS_\bi.
\end{eqnarray}
From now on we will identify these two vector spaces via our chosen isomorphism. With this identification we can compare our endomorphism algebras:

\begin{proposition}\label{StoQ}

The isomorphism $\boldsymbol{\tau}$ can be extended to an algebra isomorphism  $\boldsymbol{\tau}: \cbQ_\bi\to\ccS_\bi$ which 
\begin{itemize}
\item identifies  the subalgebra of $\cbQ_\bi$ generated by all $e(\bu_J)pe(\bu_J)$ for all  $J\subset\bbI,\bu_J\in \bU_J$ and $p\in \Lambda_{\bu_J}$with the algebra  $\ccQ_\bi$  from Section~\ref{Qsection},
\item identifies splits in the sense that, for any $J\subset K\subseteq$ and $\bu_J=(\bu,J) \in \bU_J$, it maps $\bigcurlyvee_{\bu_K}^{\bu_{J}}$ to  $\e_{\bu_J}\bb^1_{J,K}\e_{\bu_K}$, 
\item identifies merges in the sense that, in case  $J\subset K$, $|J|=|K|-1$ and $\bu_J=(\bu,J) \in \bU_J$, the generator $\bigcurlywedge_{\bu_J}^{\bu_{K}}$ maps to $\e_{\bu_K}\bb^1_{K,J}\e_{\bu_J}P^{-1}$ for an invertible power series $P$. 
\end{itemize}
\end{proposition}
From our identification of local and global indices in \eqref{zeta}, it is immediate that $\boldsymbol{\tau}$ extends the isomorphism $\eqref{isogamma}$ given in Section~\ref{heckeklrisosec}.
\proof
It follows from Proposition~\ref{cfaithful} and Lemma~\ref{pmult} that the action of the algebra  $\ccQ_\bi$ coincides with the action of the subalgebra of $\cbQ_\bi$ generated by all $e(\bu_J)pe(\bu_J)$ under the identification $\boldsymbol{\tau}^{-1}$. Hence the first claim holds. It is also clear from Example~\ref{specialmorph} that for $J\subset K$ and $\bu_J=(\bu,J) \in \bU_J$, the action  of $\e_{\bu_J}\bb^1_{J,K}\e_{\bu_K}$ on $\coFS_\bi$ translates directly to the action of $\bigcurlyvee_{\bu_K}^{\bu_{J}}$ under $\boldsymbol{\tau}^{-1}$. Hence the second assertion holds as well. 
We now claim that  for $J\subset K$, $|J|=|K|-1$ and $\bu_K = (\bu',K) \in \bU_K$ a  simple merge of $\bu_J =(\bu,J)\in \bU_J$, the action  of $\e_{\bu_K}\bb^1_{K,J}\e_{\bu_J}$ on $\coFS_\bi$ translates into the action of $\bigcurlywedge_{\bu_J}^{\bu_{K}} P$ for an invertible power series $P$.
Again, to ease terminology, we check this in the case of $K=\bbI$ and $\bu_J=(1^{a_1},2^{a_2},\cdots, e^{a_e}|1^{b_1},2^{b_2},\cdots, e^{b_e})$. Since the calculations are local this is sufficient.
Recall from Proposition~\ref{mergeactionschur} that for $f\in \hat{R}_{+,\bu_J}$ 
\begin{equation*}
e_{\bu'}\corho(\bb^1_{K,J})f\bbv^{(J)}= e_{\bu'}\bC_{a_1,b_1,a_2,b_2,\ldots,a_e,b_e}^{d_1,\;\;\ldots\;\; d_2,\;\ldots\; d_n}\sigma_{\bu_J}^{\bu'_K} \prod_{k=a+1}^n\prod_{j=1}^a \gt_{j,k}f\bbv^{(K)}.\end{equation*}

Translating this into the variables $Y_{i,j}$, notice that $\sigma_{\bu_J}^{\bu_K}$ becomes superfluous as it is precisely the element mapping $(i,j)_{\bu_J}$ from Definition~\ref{WuJ} to $(i,j)_{\bu'_K}$ and is hence the identity on the variable $Y_{i,j}$. We obtain
\begin{eqnarray*}
e_{\bu'}\corho(\bb^1_{K,J})f\bbv^{(J)}= e_{\bu'}\bC_{a_1,b_1,a_2,b_2,\ldots,a_e,b_e}^{ d_1,\;\;\;\ldots\;\;\;d_2, \ldots, d_e} 
\prod_{i=1}^e\prod_{s=1}^e\prod_{j=1}^{a_i}\prod_{k=a_s+1}^{ d_s} \frac{qY_{i,j}-Y_{s,k}}{Y_{i,j}-Y_{s,k}}
\end{eqnarray*}

Under $\boldsymbol{\tau}^{-1}$, multiplication by $\prod_{i=1}^e\prod_{s=1}^e\prod_{j=1}^{a_i}\prod_{k=a_s+1}^{ d_s} \frac{qY_{i,j}-Y_{s,k}}{Y_{i,j}-Y_{s,k}}
$ translates to multiplication by
\begin{eqnarray*}
&&\prod_{i=1}^e\prod_{s=1}^e\prod_{j=1}^{a_i}\prod_{k=a_s+1}^{ d_s} \frac{q^{i+1}(1-y_{i,j})-q^s(1-y_{s,k})}{q^{i}(1-y_{i,j})-q^s(1-y_{s,k})}\\
&=&\prod_{i=1}^e\prod_{j=1}^{a_i}\left(\prod_{k=a_i+1}^{ d_i}      \frac{1-q+qy_{i,j}-y_{i,k}}{y_{i,j}-y_{i,k}}    \prod_{k=a_{i+1}+1}^{ d_{i+1}}   \frac{-q(y_{i,j}-y_{i+1,k})}{1-q-y_{i,j}+qy_{i+1,k}}\mathbf{R}\right)\\
&=&\left(\prod_{i=1}^e\prod_{j=1}^{a_i}\frac{ \prod_{k=a_{i+1}+1}^{ d_{i+1}} y_{i,j}-y_{i+1,k} }{\prod_{k=a_i+1}^{ d_i} y_{i,j}-y_{i,k} }\right)\mathbf{P}, \text{   where}
\end{eqnarray*} 
\begin{eqnarray*}
\mathbf{R}&=&\prod_{\substack{s=1\\s\neq i,i+1}}^e \prod_{k=a_s+1}^{d_s} \frac{q^{i+1}(1-y_{i,j})-q^s(1-y_{s,k})}{q^{i}(1-y_{i,j})-q^s(1-y_{s,k})}\\
\mathbf{P}& = &\prod_{i=1}^e\prod_{j=1}^{a_i}(\prod_{k=a_i+1}^{ d_i}  (1-q+qy_{i,j}-y_{i,k})   \prod_{k=a_{i+1}+1}^{ d_{i+1}}   \frac{-q}{1-q-y_{i,j}+qy_{i+1,k}}\mathbf{R}).
\end{eqnarray*}
Note that $\mathbf{P}$ is an invertible power series in  $\Lambda_{\bu_K}$.
Hence the third assertion holds. By definition of the modified quiver Schur algebra we have mapped all generators to the corresponding elements in $\ccS_\bi$ by identifying their action on the faithful representations. Thus, $\boldsymbol{\tau}$ is injective, and hence an  isomorphism, as the image of $\boldsymbol{\tau}$ contains a generating set for $\ccS_\bi$ by Proposition~\ref{generatingbetter}.
\endproof

\subsection{The isomorphism $\bB_\bi\cong\bA_\bi$ of (shifted) quiver Schur algebras}
We next show that the shifted quiver Schur algebra is isomorphic to the ordinary quiver Schur algebra. We start with some preparation. First, we again identify the vector spaces underlying the faithful representations. For $\bu_J\in\bU_J$, we set
${}^\downarrow\round\Lambda_{\bu_J}:=E_{\bu_J}\round\Lambda_{\bu_J}$ and ${}^\downarrow\round\Lambda= \bigoplus_{\substack{J\subseteq \bbI\\ \bu_J\in\bU_J}}{}^\downarrow\round\Lambda_{\bu_J}$. Fix the vector space isomorphism
\begin{eqnarray*}
\shift_{\bu_J}\;:\;\round\Lambda_{\bu_J} &\to & {}^\downarrow\round\Lambda_{\bu_J}, \quad 
 f \mapsto E_{\bu_J}f \hbox{ and } \shift=\bigoplus_{\substack{J\subseteq \bbI\\ \bu_J\in\bU_J}} \shift_{\bu_J}:\round\Lambda \to  {}^\downarrow\round\Lambda.
\end{eqnarray*}

\begin{lemma}
\label{repshifted}
Endowing ${}^\downarrow\round\Lambda$ with a representation of $\bB_\bi$ via $\kappa$, the induced action is given by 
the same formulae as the action of $\bA_\bi$ on $\round\Lambda$ for idempotents and polynomials and, for splits and merges as follows.
\begin{trivlist}
\item  $\quad\bullet$ The {\it split} $\lrarrow\bigcurlyvee_{\bu_K}^{\bu_{J}}$ for $J\subset K$, $\bu_J = (\bu,J) \in \bU_J$, acts on $f\in  {}^\downarrow\round\Lambda_{\bu'_{K'}}$ by 
\begin{eqnarray*}
f&\mapsto 
\begin{cases}
\round{\mathtt{E}}_{\bu_J}^{\bu_K}f \in {}^\downarrow\round\Lambda_{\bu_{J}} &\text{if $\bu'_{K'} =\bu_K$},\\ 
0 &\text{otherwise}.
\end{cases}
\end{eqnarray*}
\item $\quad\bullet$The {\it merge} $\lrarrow\bigcurlywedge_{\bu_{J}}^{\bu_K}$ for $J\subset K$, $\bu_J = (\bu,J) \in \bU_J$ act on 
$f\in \round\Lambda_{\bu'_{J'} } $ by 
\begin{eqnarray*}
f&\mapsto 
\begin{cases}
\round\Delta( f)\in  {}^\downarrow\round\Lambda_{\bu_K}&\text{if $\bu'_{J'} =\bu_{J}$},\\ 
0 &\text{otherwise}.
\end{cases}
\end{eqnarray*}

\end{trivlist}
\end{lemma}

\proof
The actions of idempotents and polynomials are immediate. In order to compare the actions of splits and merges, consider the diagrams
$$\xymatrix{ \round\Lambda_{\bu_K}\ar^{1}[d] \ar^{\cdot\round{\mathtt{E}}_{\bu_K}}[rr]&&{}^\downarrow\round\Lambda_{\bu_K}\ar[d]\\
 \round\Lambda_{\bu_J} \ar^{\cdot\round{\mathtt{E}}_{\bu_J}}[rr]&&{}^\downarrow\round\Lambda_{\bu_J}
}\qquad \qquad\xymatrix{ \round\Lambda_{\bu_K}\ar^{\cdot\round{\mathtt{E}}_{\bu_K}}[rr]&& {}^\downarrow\round\Lambda_{\bu_K}\\
 \round\Lambda_{\bu_J} \ar^{\round{\Delta}(\round{\mathtt{E}}_{\bu_J}^{\bu_K} - ) }[u] \ar^{\round{\cdot\mathtt{E}}_{\bu_J}}[rr]&& \ar[u]{}^\downarrow\round\Lambda_{\bu_J}.
}$$
Then, for splits, the claim of the lemma is equivalent to the commutativity of the first diagram, which is equivalent to the fact that the induced action of $\lrarrow\bigcurlyvee_{\bu_K}^{\bu_{J}}$ on ${}^\downarrow\round\Lambda$ is indeed given by multiplication with $\round{\mathtt{E}}_{\bu_J}/\round{\mathtt{E}}_{\bu_K} = \round{\mathtt{E}}_{\bu_J}^{\bu_K}. $

For merges, we need to verify commutativity of the second diagram, which again stems from the fact that the action of $\lrarrow\bigcurlywedge_{\bu_{J}}^{\bu_K}$ on ${}^\downarrow\round\Lambda$ is given by
$$f \mapsto  \round{\mathtt{E}}_{\bu_K}\round{\Delta}(\round{\mathtt{E}}_{\bu_J}^{\bu_K} \round{\mathtt{E}}_{\bu_J}^{-1} f ) =\round{\mathtt{E}}_{\bu_K}\round{\Delta}(\round{\mathtt{E}}_{\bu_K}^{-1} f ).  $$
Since $\round{\mathtt{E}}_{\bu_K}$ commutes with $\bC_{I_{\bu_J}}^{I_{\bu_K}}$ by $\fS_{\bu_K}$-invariance of $\round{\mathtt{E}}_{\bu_K}$, we are done.
\endproof

\begin{lemma}
\label{boring}
The representation of  $\bB_\bi$ on  ${}^\downarrow\round\Lambda$ is faithful. Moreover, the action of $\bA_\bi$ restricted to ${}^\downarrow\round\Lambda$ is equal to the action of $\bB_\bi$.
\end{lemma}
\proof
Directly from the proof of the Lemma~\ref{repshifted}, we see that the action of $\bA_\bi$ on $\round\Lambda$ when restricted to ${}^\downarrow\round\Lambda$ is equal to the action of $\bB_\bi$. 
Since the representation of $\bA_\bi$ on $\round\Lambda$ was faithful, the representation of $\bB_\bi$ on ${}^\downarrow\round\Lambda$ is faithful as well. The second statement follows from the commutative diagrams above. 
\endproof

The canonical embedding $\iota\colon {}^\downarrow\round\Lambda \hookrightarrow \round\Lambda, \, f\mapsto f$ induces an algebra isomorphism:
\begin{proposition}\label{shiftA}
The algebras $\bA_\bi$ and $\bB_\bi$ are isomorphic.
\end{proposition}
\proof
Since the action of $\bA_\bi$ restricted to ${}^\downarrow\round\Lambda$ is equal to the action of $\bB_\bi$ by Lemma~\ref{boring}, we see that ${}^\downarrow\round\Lambda$ is a faithful subrepresentation of $\round\Lambda$ for $\bA_\bi$. 
The algebra $\bA_\bi$ is therefore completely defined by its action on ${}^\downarrow\round\Lambda$, and we obtain the desired isomorphism $\bA_\bi\cong\bB_\bi$ from Lemma~\ref{repshifted}.
\endproof

\subsection{The isomorphism $\bQ_\bi\cong\bB_\bi$.}
In order to prove that $\ccS_\bi$ and $\bA_\bi$ are isomorphic, it now remains to show that $\bQ_\bi$ and $\bB_\bi$ are isomorphic.

In order to do this, we define a bijection 
\begin{eqnarray*}\widetilde{}\quad\colon \bigcup_{J \subset \bbI} \bU_J&\to &\bigcup_{J \subset \bbI} \bU_J, \quad
\bu_J\mapsto \widetilde{\bu_J}.\end{eqnarray*}
where for $\bu_J=(u_1,\cdots, u_{t_1}|u_{t_1+1},\cdots, u_{t_2}|\cdots |u_{t_{r-1} +1}, \cdots, u_{t_r})$, we set 
$$\widetilde{\bu_J}=(u_{t_{r-1} +1}, \cdots, u_{t_r}|u_{t_{r-2} +1}, \cdots, u_{t_{r-1}}|\cdots|u_1,\cdots, u_{t_1}).$$

We further define the inner automorphism  of $\fS_n$ which is given by conjugation with the longest element $w_0^J$ of $D^\bbI_{\emptyset,J}$ and notice that this interchanges $\fS_{\bu_J}$ and $\fS_{\widetilde{\bu_J}}$.
It induces an isomorphism $\theta$ of vector spaces
$$\xymatrix{\FF[y_{1,1}, \ldots, y_{1,d_1},y_{2,1}, \ldots, y_{2,d_2},\ldots, y_{e,d_e}]\ar[d]_\theta&y_{c,j}\ar@{|->}[d]\\
\FF[z_{1,1}, \ldots, z_{1,d_1},z_{2,1}, \ldots, z_{2,d_2},\ldots, z_{e,d_e}]&z_{c,w_0^J(j) }=-z_{c,d_c+1-j}
}$$

which restricts to an isomorphism $\theta$ of vector spaces
$\theta \colon\Lambda_{\bu_J}\overset{\sim}{\to} \round\Lambda_{\widetilde{\bu_J}}.$

\begin{example}
Considering $\bu_J = (1,1,1,1,2,2,3|1,2,3|1,1,2,3,3)$, we have $\widetilde{\bu_J} = (1,1,2,3,3,|1,2,3|1,1,1,1,2,2,3)$ and $\W_J \cong \fS_7\times \fS_3\times\fS_4$. Further, $$\fS_{\bu_J} \cong (\fS_4\times \fS_2\times\{1\}) \times(\{1\}\times\{1\}\times \{1\}) \times(\fS_2\times\{1\}\times \fS_2).$$
Under conjugation by $w_0^J$ this is sent to $$(\fS_2\times\{1\}\times \fS_2) \times(\{1\}\times\{1\}\times \{1\}) \times(\fS_4\times\fS_2\times \{1\})\cong \fS_{\widetilde{\bu_J}}.$$
\end{example}

\begin{proposition}\label{twistA}
There is an isomorphism of algebras $\bQ_\bi \rightarrow \bB_\bi$ given by
$$e(\bu_J)\mapsto e(\widetilde{\bu_J}), \quad pe(\bu_J)\mapsto\theta(p)e(\widetilde{\bu_J}), \quad{\bigcurlyvee}_{\bu_K}^{\bu_{J}} \mapsto  \lrarrow\bigcurlyvee_{\widetilde{\bu_K}}^{\widetilde{\bu_{J}}},\quad
{\bigcurlywedge}_{\bu_{J}}^{\bu_K}\mapsto\lrarrow\bigcurlywedge_{\widetilde{\bu_{J}}}^{\widetilde{\bu_K}}.$$
%
\end{proposition}
\proof
We check that the isomorphism $\theta: \Lambda \to \round\Lambda$ intertwines the actions of $\bQ_\bi$ and $\bB_\bi$ with respect to the isomorphism given in the proposition.
It is obvious that this is true for the idempotents and the polynomials, as well as the splits. From the action of merges and the fact that conjugation with $w_0^J$ sends $\bC_{I_{\bu_J}}^{I_{\bu_K}}$ to $\bC_{I_{\widetilde{\bu_J}}}^{I_{\widetilde{\bu_K}}}$, we see that, in order to prove the proposition, it suffices to show that
$\theta(\mathtt{E}_{\bu_J}) = \round{\mathtt{E}}_{\widetilde{\bu_J}}  \hbox{ and }  \theta(\mathtt{S}_{\bu_J}) = \round{\mathtt{S}}_{\widetilde{\bu_J}}.$
Notice that the term $y_{i,j}-y_{i+1,k}$ appears in $\mathtt{E}_{\bu_J}$ (and thus the term $z_{i+1,d_{i+1}+1-k}-z_{i,d_i+1-j}$ appears in $\theta(\mathtt{E}_{\bu_J})$) if and only if the $j$th appearance of $i$ in $\bu_J$ is in an earlier segment than the $k$th appearance of $i+1$. As applying $\;\widetilde{}\;$ reverses segments, this is equivalent to the $(d_i+1-j)$th appearance of $i$ in $\widetilde{\bu_J}$ being in a later segment than the $(d_{i+1}+1-k)$th appearance of $i$; or to the term $z_{i+1,d_{i+1}+1-k}-z_{i,d_i+1-j}$ appearing in $\round{\mathtt{E}}_{\widetilde{\bu_J}}$.
The claim that $\theta(\mathtt{S}_{\bu_J}) = \round{\mathtt{S}}_{\widetilde{\bu_J}}$ is checked analogously.
\endproof
\subsection{The main theorem.}
We are now prepared to prove our main result:
\begin{theorem}[Isomorphism Theorem]
\label{IsoTheorem}
There is an isomorphism of algebras
\begin{eqnarray*}
\ccS_\bi&\cong &\cbA_\bi.
\end{eqnarray*}
Via this isomorphism $\ccS_\bi$ inherits a grading from $\cbA_\bi$, i.e. the category of representations of $\ccS$ with fixed central character corresponding to $\bi$ inherits a grading.
\end{theorem}
\proof Composing the isomorphism from Proposition~~\ref{StoQ} with the completions of the respective isomorphisms in Propositions~\ref{twistA} and ~\ref{shiftA} provides the required isomorphism. In formulae, the isomorphism is the composition
\begin{equation*}
\xymatrix{
\ccS_\bi\ar[rrr]^{\text{Proposition~\ref{StoQ}}\quad}&&&\cbQ_\bi\ar[rrr]^{\text{Proposition~\ref{twistA}}}&&&\cbB_\bi\ar[rrr]^{\text{Proposition~\ref{shiftA}}}&&& \cbA_\bi. \hfill\qed
}
\end{equation*} 
\section{The  example $\mathrm{GL}_2(\Q_5)$ in characteristic $3$}
\label{Lastsection}
We finish with an explicit example. Consider  the unipotent block  $\cB$ (the block  containing the trivial representation) of the category of smooth representations for $\mathrm{GL}_2(\Q_5)$ over an algebraically closed field $\FF$ of characteristic~$3$, so $e=2$.

\subsection{The (completed) quiver Schur algebra}
 Let $\cB^1$ as in \eqref{BandS} and $\cB^1_{\ba}$ the full subcategory of $\cB^1$ consisting of representations with generalized central character $\chi_\ba$ where $\ba=(q,q^2)$. Equivalently, $\cB^1_{\ba}
$ is the full subcategory of $\cS-\mathrm{Mod}$ of all representations with generalized central character $\chi_\ba$. 

Recall that the path algebra of a quiver is the $\FF$-algebra with basis all possible paths obtained by concatenating the arrows, including the paths of length zero corresponding to the vertices of the graph.  The multiplication of two paths is the path obtained by concatenation if this makes sense and zero otherwise.
\begin{theorem}
\label{ThmEx}
Let $n=2=e$. Then the quiver Schur algebra $\bA_\bi$ for $\bi=(1,2)$ is (as graded algebra) isomorphic to the path algebra $B$ of the following quiver  with grading given by putting the horizontal arrows in degree $1$ and the loops in degree $2$ 
\begin{eqnarray}
\label{quiver}
\begin{tikzcd}
(1|2)\arrow[bend left]{rr}{\bigcurlywedge_{(1|2)}^{(1,2)}}\ar[loop below]{}{x_{2,1}e(1|2)} \ar[loop above]{}{x_{1,1}e(1|2)} 
&&(1,2)\arrow[bend left]{rr}{\bigcurlyvee_{(1,2)}^{(2|1)}}\arrow[bend left]{ll}{\bigcurlyvee_{(1,2)}^{(1|2)}}\ar[loop below]{}{x_{2,1}e(1,2)} \ar[loop above]{}{x_{1,1}e(1,2)} 
&&(2|1)\arrow[bend left]{ll}{\bigcurlywedge_{(2|1)}^{(1,2)}}\ar[loop below]{}{x_{2,1}e(2|1)} \ar[loop above]{}{x_{1,1}e(2|1)} 
\end{tikzcd}
\end{eqnarray}
modulo the following (homogeneous) relations: 
\begin{eqnarray*}
\bigcurlyvee_{(12)}^{(1|2)}\bigcurlywedge_{(1|2)}^{(12)} &=&(x_{2,1}-x_{1,1})e((1|2)),\\
\bigcurlyvee_{(12)}^{(2|1)}\bigcurlywedge_{(2|1)}^{(12)} &=& (x_{1,1}-x_{2,1})e((2|1)),\\
\bigcurlywedge_{(1|2)}^{(12)} \bigcurlyvee_{(12)}^{(1|2)}&=&-\bigcurlywedge_{(2|1)}^{(12)}\bigcurlyvee_{(12)}^{(2|1)}\;=\; (x_{1,1}-x_{2,1})e((1,2)),\\
x_{i,1}\bigcurlyvee_{(12)}^{\bu_J}& =& \bigcurlyvee_{(12)}^{\bu_J}x_{i,1} \quad \hbox{for } i\in\{1,2\}, \bu_J\in\{(1|2),(2|1)\},\\
x_{i,1}\bigcurlywedge^{(12)}_{\bu_J}& =& \bigcurlywedge^{(12)}_{\bu_J}x_{i,1} \quad \hbox{for } i\in\{1,2\}, \bu_J\in\{(1|2),(2|1)\}.
\end{eqnarray*}
 \end{theorem}
\begin{remark}
\label{idemp}
Since the algebra is non-negatively graded and $3$-dimensional in degree $0$, the three idempotents $e(1|2)$, $e(2|1)$, and $e(1,2)$  must be primitive. 
 \end{remark}
\proof

By Remark~\ref{idemp}, the given three idempotent are primitive and by definition pairwise orthogonal, hence the quiver has three vertices.
The idempotents together with the elements corresponding to the arrows generate the quiver Schur algebra by Definition~\ref{QS}. The relations are easily verified on the faithful representation from Definition~\ref{QS}. That these are all the relations is again checked by a direct calculation, or follows from the basis theorem in \cite{SW}.
\endproof

\begin{remark}
\label{Heckesub}
Note that the elements $e(1|2)$, $e(2|1)$, $x_{i,1}e(1|2)$, $x_{i,1}e(1|2)$ with $i=1,2$ together with ${{\bigcurlyvee}^{(2|1)}\above -40pt{\bigcurlywedge}_{(1|2)}}=\bigcurlyvee^{(2|1)}_{(1,2)}\bigcurlywedge_{(1|2)}^{(1,2)}$  and ${{\bigcurlyvee}^{(1|2)}\above -30pt{\bigcurlywedge}_{(2|1)}}=\bigcurlyvee^{(1|2)}_{(1,2)}\bigcurlywedge_{(2|1)}^{(1,2)}$ generate a graded subalgebra of $B\cong\bA_\bi$ isomorphic to the quiver Hecke or KLR algebra attached in \cite{KL} and \cite{Rouquier} to the cyclic quiver and the sequence $\bi=(1,2)$.
\end{remark}

From our main theorem we get the following consequence.

\begin{corollary}
Let $\G=\op{GL}_2(\Q_5)$ and assume $\ell=3$, hence $e=2$. Then the category $\cB^1_{\ba}$ of representations in $\cB^1$ with generalised central character $\chi_\ba$ 
 is equivalent to the category of $\hat{B}$-modules, where  $\hat{B}$ is the completion of $B$ at the maximal ideal $(x_{1,1}, x_{2,1})$ of $\FF[x_{1,1}, x_{2,1}]\subset B$. 
\end{corollary}
\proof
By Theorem~\ref{IsoTheorem} and Theorem~\ref{ThmEx}, $\hat{B}$ is isomorphic to the completed affine Schur algebra from Proposition~\ref{complS}. Hence it is isomorphic to the completion of the endomorphism ring of a projective progenerator of $\cB^1$ by \eqref{BandS}, the module category over which gives precisely the category of objects in $\cB^1$ with the given generalised central character.
\endproof

\begin{remark}
Since every irreducible representation in $\cB$ is smooth and therefore admissible (see e.g.  \cite[Theorem 4.42]{Blondel} or \cite[Theorem 3.25]{BZ}), it has a central character by Schur's Lemma. The category of objects in $\cB^1$ with some generalised central character thus includes all finite length objects in $\cB^1$.
\end{remark}

Note that $\op{End}_B(Be)\cong\FF[x_1,x_2]$, generated by $ex_{1,1}e$ and $ex_{1,2}e$ for any $e\in\{e(1|2),e(2|1),e(12)\}$. Moreover, $\op{Hom}_B(Be,Be')\cong\FF[x_1,x_2]$ as vector spaces for any pair $(e,e')$ of these idempotents. 
It is a free left $\op{End}_B(Be)$-module and a free right $\op{End}_B(Be')$-module of rank $1$ with basis the minimal degree morphism in $\op{Hom}_B(Be,Be')$. Hence $B$ can be viewed as a $\FF[x_1,x_2]$-algebra. As such it is quadratic, i.e. generated in degree one (by the morphisms corresponding to the arrows given by simple merges and splits) with relations in degree two.

\subsection{Indecomposable projectives}
The indecomposable projective $B$-modules $P((1|2))$
and $P((1,2))$ are shown in the pictures below, where the numbers stand for simple objects and the lines for a basis vector in $\operatorname{Ext}^1$.

The part indicated by the non-dashed lines should be extended to infinity at the bottom and then the whole resulting part is copied infinitely many times (indicated by the dashed lines), once for each power of $(x_{1,1}+x_{2,1})$. The structure of $P((2|1))$ is similar to that of ${P((1|2))}$, with $(1|2)$ and $(2|1)$ swapped.
\vspace{0.1cm}

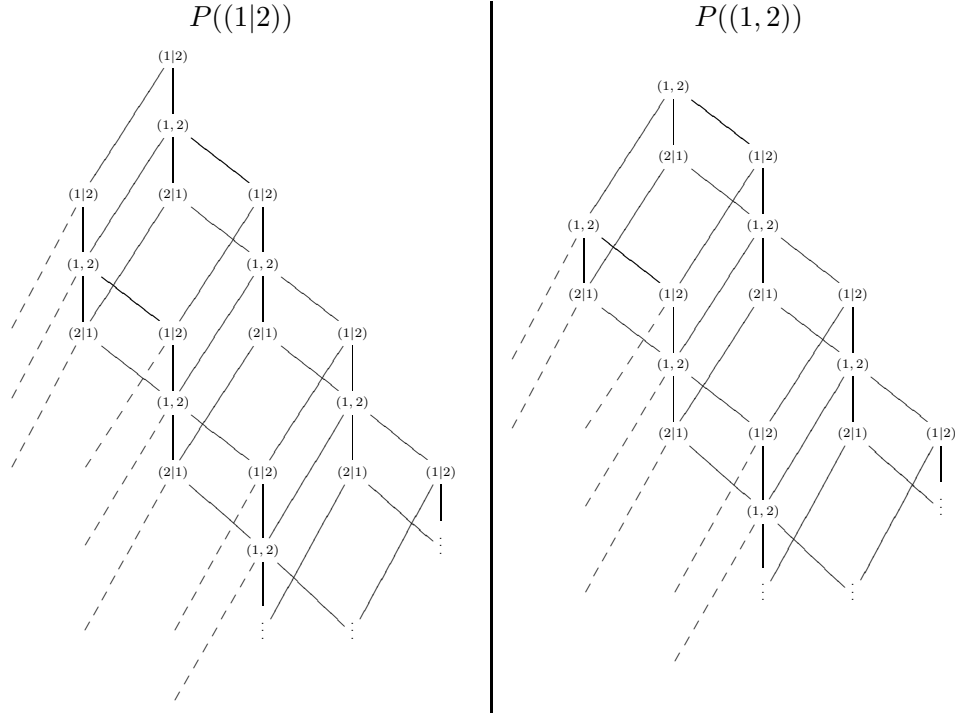
\begin{array}[t]{c|c}
{P((1|2))}&{P((1,2))}\\
\begin{minipage}{6.3cm}
\vspace{0.1cm}
\tiny
\scalebox{0.7}{\xymatrix{
&& (1|2)\ar@{-}[d]\ar@{-}[ddl]\\
&& (1,2)\ar@{-}[d]\ar@{-}[dr]\ar@{-}[dr]\ar@{-}[ddl]\\
&(1|2)\ar@{-}[d] \ar@{--}[ddl]&(2|1)\ar@{-}[dr]\ar@{-}[ddl]&(1|2)\ar@{-}[d]\ar@{-}[ddl]\\
& (1,2)\ar@{--}[ddl]\ar@{-}[d]\ar@{-}[dr]\ar@{-}[dr] & &(1,2)\ar@{-}[dr]\ar@{-}[d]\ar@{-}[ddl]\\
& (2|1)\ar@{--}[ddl]\ar@{-}[dr]&(1|2)\ar@{--}[ddl]\ar@{-}[d]&(2|1)\ar@{-}[dr]\ar@{-}[ddl]&(1|2)\ar@{-}[d]\ar@{-}[ddl]\\
 &&(1,2)\ar@{--}[ddl]\ar@{-}[dr]\ar@{-}[d]&&(1,2)\ar@{-}[d]\ar@{-}[dr]\ar@{-}[ddl]\\
&& (2|1)\ar@{--}[ddl]\ar@{-}[dr]&(1|2)\ar@{--}[ddl]\ar@{-}[d]&(2|1)\ar@{-}[dr]\ar@{-}[ddl]&(1|2)\ar@{-}[d]\ar@{-}[ddl]\\
 &&&(1,2)\ar@{--}[ddl]\ar@{-}[d]\ar@{-}[dr]&&\vdots\\
 &&&\vdots&\vdots&\\
 &&&&}
}
\end{minipage}
&
\begin{minipage}{6.5cm}
\tiny
\scalebox{0.7}{\xymatrix{
&& (1,2)\ar@{-}[d]\ar@{-}[dr]\ar@{-}[dr]\ar@{-}[ddl]\\
&&(2|1)\ar@{-}[dr]\ar@{-}[ddl]&(1|2)\ar@{-}[d]\ar@{-}[ddl]\\
& (1,2)\ar@{--}[ddl]\ar@{-}[d]\ar@{-}[dr]\ar@{-}[dr] & &(1,2)\ar@{-}[dr]\ar@{-}[d]\ar@{-}[ddl]\\
& (2|1)\ar@{--}[ddl]\ar@{-}[dr]&(1|2)\ar@{--}[ddl]\ar@{-}[d]&(2|1)\ar@{-}[dr]\ar@{-}[ddl]&(1|2)\ar@{-}[d]\ar@{-}[ddl]\\
 &&(1,2)\ar@{--}[ddl]\ar@{-}[dr]\ar@{-}[d]&&(1,2)\ar@{-}[d]\ar@{-}[dr]\ar@{-}[ddl]\\
&& (2|1)\ar@{--}[ddl]\ar@{-}[dr]&(1|2)\ar@{--}[ddl]\ar@{-}[d]&(2|1)\ar@{-}[dr]\ar@{-}[ddl]&(1|2)\ar@{-}[d]\ar@{-}[ddl]\\
 &&&(1,2)\ar@{--}[ddl]\ar@{-}[d]\ar@{-}[dr]&&\vdots\\
 &&&\vdots&\vdots&&&&\\
 &&&&
}}
\end{minipage}
\end{array}
\normalsize
\subsection{The corresponding irreducible representations}
The labelling of the primitive idempotents in \eqref{quiver} corresponds to a labelling of the three simple modules in $\cB$. Explicitly, we have  
\begin{itemize}
\item $(1|2)$ (corresponding to the trivial representation), 
\item $(2|1)$  (corresponding to the composition of the valuation on $\Q_5$ and the determinant), and
\item $(1,2)$ (corresponding to the cuspidal representation). 
\end{itemize}
To verify this, that the first two idempotents are contained in the quiver Hecke algebra (see Remark~\ref{Heckesub}), hence correspond to the two non-cuspidal simple representations. For these two the identification is a matter of conventions.

\subsection{The Extquiver of $B$}
\begin{corollary}
\label{ThmExt}
In the situation from above, the Ext-quiver of $B$ is 
%
\begin{eqnarray}
\label{quiver2}
\begin{tikzcd}
(1|2)\arrow[bend left]{rr}{\bigcurlywedge_{(1|2)}^{(1,2)}}\ar[loop above]{}{ze(1|2)} 
&&(1,2)\arrow[bend left]{rr}{\bigcurlyvee_{(1,2)}^{(2|1)}}\arrow[bend left]{ll}{\bigcurlyvee_{(1,2)}^{(1|2)}}\ar[loop above]{}{ze(1,2)} 
&&(2|1)\arrow[bend left]{ll}{\bigcurlywedge_{(2|1)}^{(1,2)}} \ar[loop above]{}{ze(2|1)} 
\end{tikzcd}
\end{eqnarray}
and the relations are that $z=ze(1|2)+ze(2|1)+ze(1,2)$ is central and 
\begin{eqnarray*}
\bigcurlywedge_{(1|2)}^{(12)} \bigcurlyvee_{(12)}^{(1|2)}&=&-\bigcurlywedge_{(2|1)}^{(12)}\bigcurlyvee_{(12)}^{(2|1)}\\
\end{eqnarray*}
\end{corollary}
\proof
This follows directly from the theorem by setting $z=x_{1,1}+x_{2,1}$.
\endproof  
In this example one can in fact verify our general expectation that $\cB$ only differs from $\cB^1$ by self-extensions of the simple cuspidal representation, and thus $\cB^1$ contains more or less all information about the unipotent block $\cB$.
\section{The special case $q=1$.}
\label{Appendix}
We finally consider the special case where $q=1$, hence $e=1$. Then $\cH=\FF[\W]$ is the group algebra of the extended affine Weyl group, \eqref{W}. We identify the representative $1\in \mZ/1\mZ$ with the single vertex of the {\it Jordan quiver} $\Gamma_1$ which has one vertex and one loop. The underlying graph is the Dynkin diagram of the Borcherds algebra attached to the Borcherds-Cartan matrix $(0)$. 

Fix a dimension vector ${\bf d}=n\in\mZ_{>0}$. Then ${\op{Rep}_{\bf d}}={\op{Rep}_{n}}$ denotes the set of complex representations $(V,x)$ of  $\Gamma_1$. For a composition $\la$ of $n$ let $\cF_\la$ be the variety of flags $F_\bullet$ in $\mathbb{C}^n$ of type $\la$, i.e. $\op{dim}F_i/F_{i-1}=\la_i$. 

Let $\cQ(\la)\subset{\op{Rep}_{n}}\times\cF_\la$ be the space of strictly stable flags, that means pairs $((V,x),F_\bullet)$ such that $x(F)_i\subset F_{i-1}$.  For compositions $\la$, $\mu$ of $n$ we consider the {\it Steinberg type variety}  $\mathcal{Z}({\la, \mu})=\cQ(\la) \times_{\op{Rep}_{n}}\cQ(\mu)$. The {\it quiver Schur algebra} $\bA_{\bf{d}}=\bA_{\bf{n}}$ attached to $\Gamma_1$ is then the direct sum of the $\op{GL}_{n}(\mathbb{C})$-equivariant Borel-Moore homologies 
$\bA_{n}=\bigoplus_{(\la,\mu)}H_{\op{GL}_{n}}^{\op{BM}}(\mathcal{Z}({\la, \mu})),$
equipped with the convolution product.  By construction, this algebra $\bA_{n}$ comes along with {\it a $\mathbb{Z}$-grading} and with a faithful representation, see \cite[(2.7), Proposition 2.9]{SW}, \cite[Proposition 2.7]{KK}.  The subspace
 \begin{eqnarray*}
\bR_n&=&H_{\op{GL}_{n}}^{\op{BM}}(\mathcal{Z}({(1,1,\ldots 1), (1,1,\ldots 1)})),
\end{eqnarray*}
is a subalgebra which we call the {\it quiver Hecke algebra} of rank $n$ attached to $\Gamma_1$. An explicit description of this algebra was given in \cite[Definition 1.2]{KK}.

\begin{lemma} \label{klrisoe1}
 The following holds for $\bR_n$.  
\begin{enumerate}
\item It is free over $H_{\op{GL}_{n}}^{BM}(\op{pt})=\mathbb{C}[x_1,\ldots, x_n]$ of rank $n$.
\item There is an isomorphism  $\bR_n[x_1^{-1},\ldots x_n^{-1}]/(x_ix_i^{-1}-1,x_i^{-1}x_i-1) \cong\mathbb{C}[\W]$ of graded algebras with $X_i$ of degree two and $s_i$ of degree zero.
\item There is an isomorphism between the completion $\hat\bR_n$ at the ideal generated by the positive degree polynomials and the completion of $\mathbb{C}[\W]$ at the ideal generated by the central character associated to the $n$-tuple $(1,\dots,1)$.
\item  $\bR_n$ has a faithful representation on $\mathbb{C}[x_1,\ldots, x_n]$ such that $x_i$ acts by multiplication  and $\psi_i$ acts by the Demazure operator \eqref{Demazure} followed by multiplication with $x_{i+1}-x_i$. 
\end{enumerate}
\end{lemma}

\begin{proof}
By \cite[Definition 1.2 and Theorem 2.8]{KK}, the assignments $1\mapsto e(1,1,\ldots, 1)$, $s_i\mapsto (\psi_i+1)e(1,1,\ldots, 1)$, $X_j\mapsto x_j$, for $1\leq i\leq n-1$ and $1\leq j\leq n$ defines an isomorphism between the quiver Hecke algebra $\bR_n$ and the subalgebra  $\mathbb{C}[S_n]\otimes\mathbb{C}[X_1,X_2,\ldots X_n]$ of $\mathbb{C}[\W]$ with the choice $P_1(u,v)=u-v$ and $Q_{1,1}(u,v)=0$ in the notation of \cite{KK}. Then the first two statements follow. The third statement follows immediately from the second. The last statement is a special case of \cite[Proposition 1.5]{KK}.
\end{proof}

The faithful representation above extends to a faithful representation of $\hat\bR_n$ on $\mathbb{C}[[x_1,\ldots, x_n]]$, which matches the completion of the faithful (natural) representation of $\mathbb{C}[\W]$ on $\mathbb{C}[[x_1^{\pm1},\ldots, x_n^{\pm1}]]$.
Since both of the faithful representations can be defined over $\mathbb{Z}$, the isomorphism of Lemma 
\ref{klrisoe1} is still valid over $\FF$.

Computations analogous to those in \cite{SW} show that the action of $\bA_n$ on its faithful representation can again be defined over $\mathbb{Z}$ and hence over $\FF$, where a presentation by generators and relations is given precisely as in the case of $e>1$ (except that Euler classes are now taken with respect to the Jordan quiver). Defining $\hat \bA_n$ as before and letting $\hat\cS_{q=1}$ be the completion at the ideal generated by the central character associated to the $n$-tuple $(1,\dots,1)$, we obtain the following theorem.

\begin{theorem}
There are algebra isomorphisms 
\begin{eqnarray*}
\bA_{n}[x_1^{-1},\ldots x_n^{-1}]/(x_ix_i^{-1}-1, x_i^{-1}x_i-1) \cong \cS_{q=1} &\text{and}&  \hat \bA_n \cong\hat\cS_{q=1}.
\end{eqnarray*}
\end{theorem}
\proof
Defining the modified quiver Schur algebra $\bQ_n$ as before, the isomorphism between $\bQ_n[x_1^{-1},\ldots x_n^{-1}]/(x_ix_i^{-1}-1, x_i^{-1}x_i-1)$  and $\cS_{q=1}$ simply identifies the corresponding faithful representations. Indeed, for $J\subset K$ the split gets identified with the element $\bb_{J,K}^1$ which acts as the identity on the faithful representation, and the merge is identified with $\bb_{K,J}^1$ which acts as the sum of elements in $D^K_{\emptyset,J}$.
In order to obtain the isomorphism between the completions, we use an affine shift sending $x_i$ to $x_i-1$.
The isomorphism between the modified quiver Schur algebra and $\bA_n$ is proved exactly as in Section~ \ref{mainiso}.
\endproof

\newpage

\begin{thebibliography}{99999999}
\bibitem[B84]{B84} I.N. Bernstein, \emph{Le centre de Bernstein}, in Representations of reductive groups over a local field, Travaux en Cours, 1--32. hermann, paris, 1984. Written by P. Deligne.
\bibitem[BGG73]{BGG} I.N. Bernstein, I.M. Gelfand, \emph{Schubert cells and the cohomology of a flag space}. Funkcional. Anal. i Priložen. {\bf 7} (1973), no. 1, 64--65.
\bibitem[BZ76]{BZ} I. N. Bernstein and A. Zelevinsky, \emph{Representations of the group {$GL(n,F),$} where {$F$} is a
              local non-{A}rchimedean field}, {Uspehi Mat. Nauk}, {\bf 31}, (1976), 5--70.
\bibitem [Bl11]{Blondel} C. Blondel, \emph{Basic representation theory of reductive $p$-adic groups},  Lecture series Morningside Center of Mathematics, Beijing, June 2011.
\bibitem[BK09]{BK} J. Brundan and A. Kleshchev, \emph{Blocks of cyclotomic Hecke algebras and Khovanov-Lauda algebras}, Invent. Math. {\bf 178} (2009), no. 3, 451--484.

\bibitem[CG10]{CG} N. Chriss, and V. Ginzburg, \emph{Representation theory and complex geometry}. Birkh{\"a}user Boston, 2010.
\bibitem[Dem74]{Demazure} M. Demazure, \emph{D{\'e}singularisation des vari\'{e}t\'{e}s de Schubert g\'{e}n\'{e}ralis\'{e}es.} 
Collection of articles dedicated to Henri Cartan on the occasion of his 70th birthday, I.
Ann. Sci. \'Ecole Norm. Sup. (4) {\bf 7} (1974), 53--88. 
\bibitem[DG07]{DG} S. R. Doty and R. M. Green, 
\emph{Presenting affine $q$-Schur algebras.} Math. Z. {\bf 256} (2007), no. 2, 311--345. 
\bibitem[DJ91]{DipperJames2} R. Dipper and G. James, \emph{q-tensor spaces and q-Weyl modules}, Trans. Amer. Math. Soc. {\bf 327} (1991), 251--282.
\bibitem[DJ89]{DipperJames} R. Dipper and G. James, \emph{The $q$-Schur algebra}, Proc. London Math. Soc. {\bf 59} (1989), 23--50.
{\bf 52} (1986), no.1, 20--52.
\bibitem[Ful97]{Fulton} W. Fulton, \emph{Young tableaux},  London Mathematical Society  Student Texts, {\bf 35}. Cambridge University Press,  1997
\bibitem[GP00]{GeckPfeiffer}  M. Geck and G. Pfeiffer, \emph{Characters of finite Coxeter groups and Iwahori-Hecke algebras},  LMS Monographs. New Series, 21, Oxford University Press (2000).
\bibitem [GV93]{GV} V. Ginzburg and E. Vasserot, \emph{Langlands reciprocity for affine quantum groups of type $A_ n$}. Int. Math Res. Not. {\bf 3}, 67--85 (1993).
\bibitem[Gre02]{Gr1} R. M. Green, 
\emph{On 321-avoiding permutations in affine Weyl groups.} J. Algebraic Combin. {\bf 15} (2002), no. 3, 241--252. 
\bibitem[Gre99]{Gr2} R. M. Green,  \emph{The affine q-Schur algebra}. J. Algebra {\bf 215} (1999), no. 2, 379--411.
\bibitem[Gre80]{Green} J. A. Green, \emph{Polynomial Representations of $\op{GL}_n$}, Springer Lecture Notes {\bf 830}, Springer-Verlag 1980. 
\bibitem[Har08]{Harris} M. Harris, \emph{The local Langlands conjecture for $\op{GL}(n)$ over a $n< p$-adic field,}
Invent. Math. {\bf 134} (1), 177--210 (1998).
\bibitem[HT01]{HT} M. Harris, R. Taylor, \emph{The Geometry and Cohomology of Some Simple Shimura Varieties.} Annals of Mathematics Studies, vol. {\bf 151}. Princeton Univ. Press  (2001).
\bibitem[Hen00]{He} G. Henniart, \emph{Une preuve simple des conjectures de Langlands pour $p$ sur un corps $p$-adique}. Invent. Math. {\bf 139} (2), 439--455 (2000)
\bibitem[IM65] {IM}N. Iwahori and H. Matsumoto, \emph{On some Bruhat decomposition and the structure of the Hecke rings of $p$-adic Chevalley groups.}, Inst. Hautes Études Sci. Publ. Math. no. 25 (1965) 5--48.
\bibitem[KK13]{KK} S.-J. Kang, M. Kashiwara and E. Park, 
\emph{Geometric realization of Khovanov-Lauda-Rouquier algebras associated with Borcherds-Cartan data.} 
Proc. Lond. Math. Soc. (3) {\bf 107} (2013), no. 4, 907--931. 
\bibitem[KLMS12]{KLMS}  M. Khovanov, A. Lauda, M. Mackaay and M. Sto\v{s}i\'c, \emph{Extended Graphical Calculus for Categorified Quantum $\mathfrak{sl}(2)$}, Memoirs of the AMS, {\bf  2019}, no. 126, 2012.
\bibitem[KL10]{KL3} M. Khovanov and A. Lauda,  \emph{A categorification of quantum sl(n)}. Quantum Topol. 1 (2010), no. 1, 1--92.
\bibitem[KL09]{KL} M. Khovanov and A. Lauda,  \emph{A diagrammatic approach to categorification of quantum groups. I.} Represent. Theory {\bf 13} (2009), 309--347.
\bibitem[Mat99]{Ma}  A. Mathas, \emph{Iwahori-Hecke algebras and Schur algebras of the symmetric group}. 
University Lecture Series, 15. American Mathematical Society, 1999.
\bibitem[Lus91]{Lusztigperv} G. Lusztig, {\emph Quivers, Perverse sheaves and quantized enveloping algebras}, JAMS, {\bf 4} (1991) no 2,  no. 2, 365--421.
\bibitem [Lus89]{LuHecke} G. Lusztig, {\emph Affine Hecke algebras and their graded version},  JAMS {\bf 2} (3) (1989),
599--685.
\bibitem [Lus83]{Lu} G. Lusztig, \emph{Some examples of square integrable representations of semisimple p-adic groups}, Trans. Amer. Math. Soc. {\bf 277} (1983), 623--653. 
\bibitem[Mat99]{Mathasbook}  A. Mathas, \emph{Iwahori-Hecke algebras and Schur algebras of the symmetric group}. University Lecture Series, {\bf 15}. AMS, 1999.
\bibitem[MS14]{MS} A. M\'enguez, V. S\'echerre, \emph{Repr\'esentations lisses modulo $\ell$ de $\op{GL}_m(D)$}. 
Duke Math. J. {\bf 163} (2014), no. 4, 795--887. 
\bibitem[Pri19]{Tomasz} T. Przezdziecki, PhD thesis (in preparation). 
\bibitem [Rou12]{Rouquier2} R. Rouquier, \emph{Quiver Hecke algebras and 2-Lie algebras}.  
Algebra Colloq. {\bf 19} (2012), no. 2, 359--410. 
\bibitem[Rou08]{Rouquier} R. Rouquier, {\emph 2-Kac-Moody algebras}. arXiv:0812.5023.
\bibitem[Schi12]{Schiffmann} O. Schiffmann, \emph{Lectures on Hall algebras}. Geometric methods in representation theory. II, 1--141, S\'emin. Congr., 24-II, Soc. Math. France, Paris, 2012.
\bibitem[Scho13]{Scholze} P. Scholze, \emph{The local Langlands correspondence for $\op{GL}_n$  over $p$-adic fields.} Invent. Math. 192 (2013), no. 3, 663--715.
\bibitem[SS14]{SS} V. S\'echerre and S. Stevens, \emph{Block decomposition of the category of $\ell$-modular smooth representations of $\mathrm{GL}_n(F)$ and its inner forms}, preprint arXiv:1402.5349, to appear in Ann. Scient. Ec. Norm. Sup.
\bibitem[SW11]{SW} C. Stroppel and B. Webster,  \emph{Quiver Schur algebras and q-Fock space}, arXiv:1110.1115.
\bibitem[Str05]{StrTL} C. Stroppel, \emph{Categorification of the Temperley-Lieb category, tangles, and cobordisms via projective functors}, Duke Math. J., {\bf 126}, (2005) no.3., 547--596.
\bibitem [Tak96]{Takeuchi}  M. Takeuchi, \emph{The group ring of $\op{GL}_n(\mathbb{F}_q)$ and the $q$-Schur algebra}, J. Math. Soc.
Japan ${\bf 48}$ (1996), 259--274.
\bibitem[Vig03]{V} M.-F.~Vign\'eras, \emph{Schur algebras of reductive $p$-adic groups. I.} Duke Math. J. {\bf 116} (2003), no. 1, 35--75.
\bibitem[Vig98]{Vigneras2} M.-F.~Vign\'eras, \emph{Induced $R$-representations of $p$-adic reductive groups.}
Selecta Math. (N.S.) {\bf 4} (1998), no. 4, 549--623. 
\bibitem[VV11]{VV} M. Varagnolo and  E. Vasserot,  \emph{Canonical bases and KLR-algebras.}.
J. Reine Angew. Math. {\bf 659} (2011), 67--100.
\bibitem[VV04]{VV2} M. Varagnolo and  E. Vasserot,  \emph{From double affine Hecke algebras to quantized affine Schur algebras}. Int. Math. Res. Not.  no. {\bf 26}, (2004),1299--1333.
\bibitem[Web13]{Ben} B. Webster, \emph{A note on isomorphisms between Hecke algebras}, arXiv:1305.0599v1.
\bibitem[Wed08]{Wedhorn} T. Wedhorn, \emph{The local Langlands correspondence for $\op{GL}(n)$ over p-adic fields}. School on Automorphic Forms on $\op{GL}(n)$, 237--320, ICTP Lect. Notes, {\bf 21}, Trieste, 2008. 
\bibitem[Zel80]{Zel1} A.V. Zelevinsky, \emph{Induced representations of reductive $p$-adic groups. II. On irre-
ducible representations of O}. Ann. Sci. Ec. Norm. Super. {\bf 13} (2), 165--210 (1980).
\end{thebibliography}
\end{document}